\documentclass[hidelinks,onefignum,onetabnum]{siamart250211}

\usepackage{lipsum}
\usepackage{amsfonts}
\usepackage{graphicx}
\usepackage{epstopdf}
\usepackage{algorithmic}
\ifpdf
  \DeclareGraphicsExtensions{.eps,.pdf,.png,.jpg}
\else
  \DeclareGraphicsExtensions{.eps}
\fi

\newsiamremark{remark}{Remark}
\newsiamremark{hypothesis}{Hypothesis}
\crefname{hypothesis}{Hypothesis}{Hypotheses}
\newsiamthm{claim}{Claim}
\newsiamremark{fact}{Fact}
\crefname{fact}{Fact}{Facts}

\headers{Controllability of semi-linear parabolic SPDEs}{L. Zhang, F. Xu, and B. Liu}

\title{New global Carleman estimates and null controllability for forward/backward semi-linear parabolic SPDEs\thanks{Submitted to the editors DATE.
\funding{This work was supported by National Key Research and Development Program of China (\#2023YFC2206100), and  National Natural Science Foundation of China (\#12231008).}}}

\author{Lei Zhang\thanks{School of Mathematics and Statistics, Hubei Key Laboratory of Engineering Modeling  and Scientific Computing, Huazhong University of Science and Technology. 
		\email{lei\_zhang@hust.edu.cn }(Corresponding author), \email{d202280019@hust.edu.cn}, \email{binliu@mail.hust.edu.cn}}
	\and Fan Xu\footnotemark[2]
	\and Bin Liu\footnotemark[2]}

\usepackage{amsopn}

\ifpdf
\hypersetup{
  pdftitle={New global Carleman estimates and null controllability for forward/backward semi-linear parabolic SPDEs},
  pdfauthor={L. Zhang, F. Xu, and B. Liu}
}
\fi

\usepackage{mathrsfs}
\usepackage[shortlabels]{enumitem}

\begin{document}

\maketitle

\begin{abstract}
In this paper, we study the null controllability for parabolic SPDEs involving both the state and the gradient of the state. To start with, an improved global Carleman estimate for linear forward (resp. backward) parabolic SPDEs with general random coefficients and square-integrable source terms is derived. Based on this, we further develop a new global Carleman estimate for linear forward (resp. backward) parabolic SPDEs with source terms in the Sobolev space of negative order, which enables us to deal with the null controllability for linear backward (resp. forward) parabolic SPDEs with gradient terms. As a byproduct, a special weighted energy-type estimate for the controlled system that explicitly depends on the parameters $\lambda,\mu$ and  weighted function $\theta$ is obtained, which makes it possible to extend the previous null controllability  to semi-linear backward (resp. forward) parabolic SPDEs by applying the fixed-point argument in appropriate Banach space.
\end{abstract}



\section{Introduction}

The Carleman estimates were originally introduced by T. Carleman  \cite{carleman1939} to study the unique continuation for elliptic PDEs. Today, these estimates have emerged as a powerful tool in the study of PDEs \cite{Fu2012Sharp,Calderon,Tataru1995Unique}, the inverse problems \cite{imanuvilov2003carleman,Yamamoto2017Carleman} and the control problems \cite{Zhang2008On,fursikov1996controllability,zhang2010,Wang2003,Wang2018Time}. In particular, this type of estimates has been extensively employed to analyze the controllability for parabolic-type PDEs, see for instance \cite{Rousseau2011Local,Enrique2000Null,Yamamoto2004,Fern2006Exact,K2020Global,
	Kassab2020Null} and the references therein.

Over the past few years, there has been significant interest in Carleman estimates and controllability for stochastic partial differential equations (SPDEs) \cite{lu2021mathematical,luzhang2021}. 
Since the work by Barbu et al. \cite{Barbu2003Carleman}, the breakthrough in this direction dates back to
Tang and Zhang \cite{tang2009null}, where the authors derived an innovative Carleman estimate for the second-order stochastic parabolic operators with general random coefficients, and established the null controllability for linear forward/backward parabolic SPDEs. Inspired by this work, the controllability and observability arising from the stochastic control problems have been studied for other types of SPDEs, such as \cite{Zhang2007Carleman,luzhang2015,liu2019CARLEMAN,yuzhang2023,Lu2014Exact,Lu2013Observability,fuliuxu2017A} and so on. It is worth noting that the aforementioned studies primarily focused on the linear SPDEs, and the study for nonlinear SPDEs becomes more difficult. As it is pointed out in \cite[Remark 2.5]{tang2009null} and \cite[Chapter 9]{lu2021mathematical}, the main challenge in extending deterministic results to stochastic settings lies in the loss of temporal regularity in solutions and the absence of compact embedding for state spaces, which render the fixed-point argument commonly used in deterministic cases inapplicable.

The primary goal of this paper is to investigate the null controllability for a class of semi-linear parabolic SPDEs that involve both the state and the gradient of the state.  
The main theorems established in the present work provide a partial affirmative answer to the open questions provided in \cite[Remark 2.5]{tang2009null} and \cite[Section 4]{hernandez2023global}. To gain a more comprehensive understanding of the results closely associated with ours, we refer to Subsection \ref{sec1.1} below for details.

Now let us present a concise overview of our results, with detailed statements available in Subsection \ref{mainresult}:
(1) By adopting suitable weighted function, we establish a novel global Carleman estimate for the forward (resp. backward) linear parabolic SPDEs with general random coefficients and $L^2$-valued source terms;
(2) By virtue of the duality argument and penalized HUM method introduced by Lions, we derive a new global Carleman estimate for the forward (resp. backward) parabolic SPDEs with the source terms in $L^2_\mathbb{F}(0,T;H^{-1}(\mathcal {O}))$;
(3) With the above $H^{-1}$-Carleman estimate, we establish a global null controllability for linear backward (resp. forward) parabolic SPDEs involving both the state and the gradient of the state. In the meantime, an interesting energy-type estimate related to the parameters $\lambda,\mu$ and the weighted function $\theta$ is obtained;
(4) By performing a fixed-point argument (without using the compactness embedding results as for deterministic counterparts), we prove  global null controllability for the semi-linear backward (resp. forward) parabolic SPDEs.

\subsection{Comparison with related works}\label{sec1.1}

During the past decades, the null controllability of nonlinear parabolic PDEs has received much attention, see for instance \cite{Doubova2002,fursikov1996controllability,Yamamoto2004,Cara1997,Enrique2000Null,K2020Global}. One of the most effective strategies is first to analyze the controllability of a linearized system, and then extend the results to nonlinear cases by using  a  fixed-point argument.

However, the compactness properties (e.g., the Aubin-Lions lemma) frequently used in the deterministic cases cannot be applied to SPDEs, which constitutes  the main challenge in applying classical methodologies to prove the   controllability of nonlinear stochastic problems \cite{tang2009null,lu2021mathematical,luzhang2021}. In a recent study \cite{Hernandez-Santamaria2022}, the authors investigated the controllability of forward nonlinear parabolic SPDEs through the introduction of an innovative concept termed statistical null-controllability. Later, in their remarkable work \cite{hernandez2023global}, Hern{\'a}ndez-Santamar{\'\i}a  et al. considered the null controllability of semi-linear stochastic heat equations with state-dependent nonlinearity. A significant contribution was the development of new global Carleman estimates, which enabled achieving the controllability results for semi-linear SPDEs through a fixed-point argument. After submitting this work to arXiv, we are pleased to note that Zhang et al. \cite{zhangsen2024global} adapted the methodology in \cite{hernandez2023global} to establish the null controllability for the stochastic Ginzburg-Landau equations with nonlinearity depending on the state. In the meantime, by adopting a duality argument in \cite{liu2014global}, the authors established the controllability for one dimensional semi-linear stochastic Cahn-Hilliard type equations with low-order source terms \cite{zhang2025new}. Recently, Wang and Zhao \cite{wang2025null} carried out an in-depth investigation into the null controllability of semi-discrete stochastic semi-linear parabolic equations. In a more recent work \cite{hernandez2024p}, Hern{\'a}ndez-Santamar{\'\i}a  et al. derived $L^p$-estimates for linear backward SPDEs by means of innovative approaches, and then presented an application to local controllability for semi-linear backward SPDEs. At this stage, let us clarify that our novelty is to improve the results of \cite{hernandez2023global,tang2009null} to the general parabolic SPDEs with Lipschitz nonlinearities depending on both the state and its gradient. It should be noted that the $L^2$-Carleman estimate in \cite{hernandez2023global} cannot be directly applied to current case, and we shall overcome the difficulty by drawing upon the ideas put forward in \cite{liu2014global,tang2009null,imanuvilov2003carleman}.

To the authors' best knowledge, regarding the controllability for parabolic SPDEs with gradient terms, there are only a few results accessible in the literature. In a recent work \cite{baroun2025null} (see also the earlier result by Liu \cite[Remark 1.4]{liu2014global} for an alternative form of Carleman estimates), Baroun et al. considered the null controllability of linear forward/backward parabolic SPDEs involving first-order and zero-order coupling terms. In which the authors successfully solved this problem by establishing new Carleman estimates via the duality method that was developed in \cite{liu2014global}. Nevertheless, it is worth noting that the
$L^2$-Carleman estimates established in \cite{hernandez2023global,tang2009null} fail to be applicable for addressing gradient-dependent terms. Similarly, the newly derived
$H^{-1}$-Carleman estimates presented in \cite{baroun2025null} and \cite[Remark 1.4]{liu2014global} are still insufficient to tackle the nonlinear controllability problem. Another novelty of the present work lies in solving this difficulty by appropriately combining the techniques discovered in \cite{hernandez2023global,tang2009null}, the penalized HUM method by Lions \cite{Lions1988} and the duality argument \cite{liu2014global,baroun2025null}. In this respect, it is worth mentioning the study by Zhang et al. \cite{zhang2025new} that leverages a similar analytical approach in investigating the null controllability of forward stochastic Cahn-Hilliard equations with the principal part characterized by one-dimensional operator $\mathrm{d}u+u_{xxxx} \mathrm{d}t$, which however leads to a limitation that the results cannot be applied to the present work. We remark  that in addition to forward controlled SPDEs, we also study the null controllability for backward parabolic SPDEs, which leads to an improvement of classical Carleman estimates for some forward SPDEs involving controls imposed on both the drift and diffusion terms. This can be seen as a nontrivial extension of the recent result \cite[Theorem 1.2]{baroun2025null} to the nonlinear scenario.

Compared with \cite{hernandez2023global}, another difference in our study lies in the consideration of the parabolic SPDEs with general random coefficients depending on $(\omega,t,x)$. This type of SPDEs has important applications in the stochastic optimal control and filtering theory \cite{lu2021mathematical}, which have received much attention from the perspective of SPDE theory (e.g.,  \cite{krylov2009divergence,agresti2025nonlinear,qiu2020l2,hsu2017stochastic}). Moreover, the appearance of the general coefficients brings several additional interaction terms that do not arise in the context of the Laplacian with constant coefficients, thereby necessitating detailed and precise analytical work. What is particularly interesting to highlight is that the coefficients of principal part possess $W^{2,\infty}$-regularity in space (see (A$_1$) below), a natural question is how to determine the minimal regularity condition. Indeed, existing studies have addressed this problem in deterministic settings (e.g., \cite{doubova2002exact,Rousseau2011Local}), but knowledge regarding the stochastic control problem remains scarce, see \cite[Section 9.7]{lu2021mathematical} for more details.

\subsection{Notations and assumptions}

Let $\mathcal {O} \subset \mathbb{R}^n(n \in \mathbb{N})$ be a bounded domain with a $\mathcal {C}^4$ boundary $\partial \mathcal {O}$, and $\mathcal {O}' \subset \mathcal {O}$ be a nonempty open subset. For any $T>0$, set $\mathcal {O}_T=(0, T) \times \mathcal {O}$, $\Sigma_T=(0, T) \times \partial \mathcal {O}$ and $\mathcal {O}'_T=(0, T) \times \mathcal {O}'$. For any subset $A\subseteq \mathbb{R}^n$, we denote by $\chi_{A} $ the characteristic function of $A$. For a positive integer $k$, we denote by $O(\mu^k)$ a function of order $\mu^k$ for large $\mu$, which is independent of $\lambda$ and $T$.

Let $(\Omega, \mathcal {F}, \mathbb{F}, \mathbb{P})$ be a fixed complete filtered probability space on which a standard one-dimensional Brownian motion $\{W(t)\}_{t \geq 0}$ is defined and such that $\mathbb{F}=\left\{\mathcal{F}_t\right\}_{t \geq 0}$ is the natural filtration generated by $W(\cdot)$, augmented by all the $\mathbb{P}$-null sets in $\mathcal{F}$. Given a Banach space $(H,\|\cdot\|_H)$, let  $L_{\mathcal{F}_t}^2(\Omega ; H)$ be the space of all $\mathcal{F}_t$-measurable random variables $\xi$ such that $\mathbb{E}\|\xi\|_H^2<\infty$. For any $T>0$, let $L_{\mathbb{F}}^2(0, T ; H)$ be the space consisting of all $H$-valued $\mathbb{F}$-adapted processes $X(\cdot)$ such that $\mathbb{E}(\|X(\cdot)\|_{L^2(0, T ; H)}^2)<\infty$; $L_{\mathbb{F}}^{\infty}(0, T ; H)$ be the space consisting of all $H$-valued $\mathbb{F}$-adapted bounded processes; and $L_{\mathbb{F}}^2(\Omega ; \mathcal {C}([0, T] ; H))$ be the space consisting of all $H$-valued $\mathbb{F}$-adapted continuous processes $X(\cdot)$ such that $\mathbb{E}(\|X(\cdot)\|_{\mathcal {C}([0, T] ; H)}^2)<\infty$.  

For the operators $\mathrm{d} u\pm\nabla\cdot(\mathcal {A} \nabla u) \mathrm{d}t$ and the nonlinearities, we assume that
\begin{itemize} [itemsep=0pt, leftmargin=28pt]
	\item [(\textbf{A$_1$})] Let $\mathcal {A}=(a^{ij})_{1\leq i,j\leq n}$ be a $n \times n$ matrix with the random coefficients $a^{ij}: \Omega \times [0,T]\times \overline{\mathcal {O}}\rightarrow \mathbb{R}$ satisfying the following conditions:
	
	\item[]1)  $a^{ij}=a^{ji}$ and $ a^{ij}\in L^\infty_\mathbb{F}(\Omega;\mathcal {C}^1([0,T];W^{2,\infty}(\mathcal {O})))$, $i,j=1,2,...,n$.
	
	\item[]2)  There is a constant $c_0>0$ such that
	$
	(\mathcal {A} \xi,\xi)_{L^2}=\sum_{i,j}a^{ij}(\omega,t,x)\xi_i\xi_j\geq c_0 |\xi|^2,
	$
	for any $ (\omega,t,x,\xi) \in \Omega \times \mathcal {O}_T \times\mathbb{R}^n$. Here and in the sequel, we frequently use the notations $\sum_{i} $ instead of $\sum_{i=1}^n$, $\sum_{i,j} $ instead of $\sum_{i,j=1}^n$, etc.
\end{itemize}
 
\begin{itemize} [itemsep=0pt, leftmargin=28pt]
	\item [(\textbf{A$_2$})] 1) For each $(y,Y) \in H^1_0(\mathcal {O})\times L^2(\mathcal {O})$, $F(\cdot, \cdot, \cdot, y,\nabla y,Y)$ is a $\mathbb{F}$-adapted and $L^2$-valued stochastic processes.
	
	\item[]2)  For any $ \left(\omega, t, x\right) \in \Omega \times \mathcal {O}_T$,
	$
	F\left(\omega, t, x, 0,\textbf{0},0\right)=0.
	$
	
	\item[]3)  There exists a constant $ L>0$ such that
	\begin{fontsize}{9.5pt}{0pt}\begin{equation*}
		\begin{split}
			|F(\omega, t, x, a_1,\textbf{b}_1,c_1)-F(\omega, t, x, a_2,\textbf{b}_2,c_2)| \leq L(|a_1-a_2|+ |\textbf{b}_1-\textbf{b}_2| +|c_1-c_2|),
		\end{split}
	\end{equation*}\end{fontsize}
	for any $ \left(\omega, t, x, a_1, a_2,\textbf{b}_1,\textbf{b}_2,c_1,c_2\right) \in \Omega \times \mathcal {O}_T \times \mathbb{R}^2 \times (\mathbb{R}^n)^2\times \mathbb{R}^2 $.
\end{itemize}

\begin{itemize} [itemsep=0pt, leftmargin=28pt]
	\item [(\textbf{A$_3$})] 1) For each $y\in H^1_0(\mathcal {O})$, $F_i(\cdot, \cdot, \cdot, y,\nabla y)$, $i=1,2$, are $\mathbb{F}$-adapted and $L^2$-valued stochastic processes.
	
	\item[]2)  For any $ \left(\omega, t, x\right) \in \Omega \times \mathcal {O}_T$, we have
	$
	F_i\left(\omega, t, x, 0,\textbf{0}\right)=0$, $i=1,2.
	$
	
	\item[]3)  There exists a constant $ L_i>0$ such that
	\begin{equation*}
		\begin{split} 
			\left|F_i\left(\omega, t, x, a_1,\textbf{b}_1\right)-F_i\left(\omega, t, x, a_2,\textbf{b}_2\right)\right| \leq L_i\left(\left|a_1-a_2\right|+ \left|\textbf{b}_1-\textbf{b}_2\right|\right),~ i=1,2,
		\end{split}
	\end{equation*}
	for any $ \left(\omega, t, x, a_1, a_2,\textbf{b}_1,\textbf{b}_2\right) \in \Omega \times \mathcal {O}_T \times \mathbb{R}^2 \times (\mathbb{R}^n)^2$.
\end{itemize}

\subsection{Main results}\label{mainresult}
The following result is crucial to defining weighted functions.

\begin{lemma}[\cite{fursikov1996controllability}]\label{weight}
	Let $\mathcal {O}_1$ be a nonempty subset of $\mathcal {O}$ such that  $\mathcal {O}_1\subset\subset \mathcal {O}'$ (i.e., $\overline{\mathcal {O}}_1\subset \mathcal {O}'$), then there exists a function $\beta\in \mathcal {C}^4(\overline{\mathcal {O}};[0,1])$ such that
	$
	0< \beta (x)\leq 1~\textrm{in}~\mathcal {O}$, $ \beta(x)=0$ on $\partial \mathcal {O}$ and $\inf_{x \in \mathcal {O} \backslash \overline{\mathcal {O}}_1}|\nabla\beta(x)|\geq a_0>0$.
\end{lemma}

Without loss of generality, in the following sections we assume that $0<T<1$. Inspired by \cite{hernandez2023global} (see \cite{badra2016local} for the deterministic case), for any   numbers $m\geq 1$ and $\mu\geq 1$, let us consider the weighted functions
$\varphi(x,t)= \gamma(t)(e^{\mu (\beta(x)+6m)}-\mu e^{6\mu (m+1)})$, and $\xi(x,t)= \gamma(t) e^{\mu (\beta(x)+6m)}$,
where $\gamma:(0,T]\mapsto \mathbb{R}^+$ is a $\mathcal {C}^2$-function satisfying
\begin{equation}\label{2.3}
	\gamma(t)=\left\{
	\begin{aligned}
		&t^{-m}  &&  \textrm{in}~ (0,T/4],\\
		&\textrm{is decreasing}   &&   \textrm{in}~[T/4,T/2],\\
		&1&&     \textrm{in}~{[T/2,3T/4]},\\
		&1+(1- 4 T^{-1}(T-t) )^\sigma  &&     \textrm{in}~[3T/4,T].
	\end{aligned}
	\right.
\end{equation}
The parameter 
$\sigma= \lambda\mu^2e^{\mu(6m-4)}> 2$,  for all $\lambda\geq 1.$ Furthermore, we  define
\begin{equation}\label{2.4}
	\begin{split}
		\theta (x,t)= e^{\ell(x,t)}\quad \textrm{and} \quad \ell(x,t)= \lambda \varphi (x,t).
	\end{split}
\end{equation}
From the definition of $\gamma$, it is clear that $\lim_{t\rightarrow T^-}\gamma(t)= 2$ and $\lim_{t\rightarrow 0^+}\gamma(t)= +\infty$, which is a bit different from the classical weighted functions used in \cite{tang2009null,liu2014global,lu2021mathematical}.

Our first main goal is to study the null controllability for semi-linear backward parabolic SPDEs (see \eqref{jj6}). To this end, let us consider the linearized system
\begin{equation}\label{2.1}
	\left\{
	\begin{aligned}
		&\mathrm{d} z- \nabla\cdot (\mathcal {A}\nabla z) \mathrm{d}t= \left(\langle \textbf{a}, \nabla z\rangle+ \alpha z+ \phi_1+ \nabla\cdot \textbf{b} \right)\mathrm{d}t+\phi_2\mathrm{d}W_t ~~~\textrm{in}~\mathcal {O}_T,\\
		&z=0~~\textrm{on}~\Sigma_T,~~~
		z(0) = z_0~~\textrm{in}~\mathcal {O},
	\end{aligned}
	\right.
\end{equation}
where $\textbf{a}\in L_\mathbb{F}^\infty(0,T;L^\infty(\mathcal {O};\mathbb{R}^n))$, $\textbf{b} \in L^2_\mathbb{F}(0,T;L^2(\mathcal {O};\mathbb{R}^n))$, $\alpha\in L_\mathbb{F}^\infty(0,T;L^\infty(\mathcal {O}))$, $\phi_1 \in L_\mathbb{F}^2(0,T;L^2(\mathcal {O}))$ and $\phi_2\in L_\mathbb{F}^2(0,T;H^1(\mathcal {O}))$. Under the  condition (A$_1$), for any $z_0\in L^2_{\mathcal {F}_0}(\Omega; L^2(\mathcal {O}))$, it is  well-known \cite[Theorem 12.3]{peszat2007stochastic} that the system \eqref{2.1} has a unique solution
$
z\in \mathcal {W}_T =  L^2_\mathbb{F}(\Omega;\mathcal {C}([0,T];L^2(\mathcal {O})))\bigcap L^2_\mathbb{F}(0,T;H^1_0(\mathcal {O})).
$


\begin{theorem} \label{thm1}
	Assume that $\textbf{b}\equiv\textbf{0} $ in \eqref{2.1} and condition (A$_1$) holds. There exist $\lambda_0,\mu_0>0$ such that, for all $\lambda \geq \lambda_0$, $\mu \geq \mu_0$, the unique solution $z$ to  \eqref{2.1} satisfies
	\begin{fontsize}{8.5pt}{8.5pt}\begin{equation}\label{carleman1}
		\begin{split}
			& \mathbb{E}\int_{\mathcal {O}} e^{2\lambda\varphi(T)}\big(\lambda^{2} \mu^{3} e^{2\mu(6m+1)}   z^2(T)  +  |\nabla z(T)|^2\big)\mathrm{d}x  +\mathbb{E} \int_{\mathcal {O}_T} \lambda\mu^{2} \xi \theta^2\big( | \nabla z |^2 + \lambda^2 \mu^2 \xi^{2}  z^2\big) \mathrm{d}x\mathrm{d}t \\
			& \leq C \Big[\mathbb{E}\int_{\mathcal {O}_T} \theta^2\big(\lambda^{2} \mu^{2}  \xi^{3} \phi_2^2 +| \nabla  \phi_2 |^2   + \phi_1^2\big)  \mathrm{d}x\mathrm{d}t +\mathbb{E}\int_{\mathcal {O}'_T} \lambda^{3}\mu^{4} \xi^{3}\theta^2 z^2 \mathrm{d}x\mathrm{d}t \Big],
		\end{split}
	\end{equation}
\end{fontsize}
\end{theorem}

\begin{remark}
	This type of Carleman estimate was first established  by Badra et al. \cite{badra2016local} to deal with the local trajectory controllability for incompressible
	Navier-Stokes equations. Later, Hernandez-Santamaria et al. \cite{hernandez2023global} developed the ideas in \cite{badra2016local} to investigate the null controllability of stochastic heat equations with Lipschitz nonlinearity depending on the state. Theorem \ref{thm1} improves the result \cite[Lemma 3.2]{hernandez2023global} by considering general random coefficients, which also can be viewed as a refined version of the classical Carleman estimate established in \cite[Theorem 5.2]{tang2009null} and \cite[Lemma 3.2]{baroun2025null}. As mentioned in Subsection \ref{sec1.1}, the extra terms, particularly the interaction terms involving the derivatives of $a^{ij}$ and $\gamma$, make the proof more complicated. Moreover, the problem of lowering the regularity for the coefficients remains unresolved in the stochastic framework, we refer to \cite{doubova2002exact,Rousseau2011Local} for the deterministic cases.
\end{remark}


\begin{theorem} \label{thm2}
	Assume that $\phi_2 \in L_\mathbb{F}^2(0,T;L^2(\mathcal {O}))$, $\textbf{b} \in L^2_\mathbb{F}(0,T;L^2(\mathcal {O};\mathbb{R}^n))$ and the condition (A$_1$) holds. Then for any $z_0 \in L^2_{\mathcal {F}_0}(\Omega; L^2 (\mathcal {O}))$, there exist positive constants $\lambda_1$ and $\mu_1$, depending only on $\mathcal {O},\mathcal {O}'$ and $T$, such that for all $\lambda\geq \lambda_1$ and $\mu \geq \mu_1$, the unique solution $z$ of \eqref{2.1} satisfies
	\begin{equation}\label{carleman2}
		\begin{split}
			&\mathbb{E}\int_\mathcal {O}\lambda\mu^2(\xi\theta^2)(T)z^2(T)\mathrm{d}x+\mathbb{E}\int_{\mathcal {O}_T} \lambda\mu^{2} \xi\theta^2(\lambda^{2}\mu^{2}\xi^{2}  z^2  +  |\nabla z|^2) \mathrm{d}x\mathrm{d}t\\
			& \leq C\Big[\mathbb{E} \int_{\mathcal {O}'_T} \lambda^{3}\mu^{4}\xi^{3} \theta^2  z^2 \mathrm{d}x\mathrm{d}t + \mathbb{E}\int_{\mathcal {O}_T} \big[ \theta^{2} \phi_1^2  + \lambda^{2}\mu^{2}\xi^{2}  \theta^{2}  (\phi_2^2  +    |\textbf{b}|^2) \big] \mathrm{d}x\mathrm{d}t  \Big].
		\end{split}
	\end{equation}
	
\end{theorem}

\begin{remark}
	As far as we are aware, the Carleman estimates \eqref{carleman2} have not been addressed in the literature. Compared with \eqref{carleman1}, the weak derivative of $\phi_2$ is removed from \eqref{carleman2}. Since  $\phi_1+ \nabla\cdot \textbf{b}$ belongs to the Sobolev space $H^{-1}(\mathcal {O})$, the Carleman estimates \eqref{carleman2} cannot be directly obtained by using the identity \eqref{2.9} deduced in the proof of  Theorem \ref{thm1}. Here we shall overcome this difficulty by combining the $L^2$-Carleman estimate in Theorem \ref{thm1} with Lions's HUM method \cite{Lions1988} and a duality argument similar to \cite{baroun2025null,liu2014global}. It should be noted that, owing to the non-singularity of the weight function $\gamma$ at $t=T$, the Carleman estimate in \cite[Theorem 4.1]{baroun2025null} exhibits a slight discrepancy from \eqref{carleman2}, and the main reason lies in the choice of the weight function $\gamma$ (cf. \cite{hernandez2023global}), which deviates from the classical ones used in \cite{baroun2025null,tang2009null}.
\end{remark}

As an application of Theorem \ref{thm2}, let us consider the null controllability of the following backward parabolic SPDEs
\begin{equation}\label{4.1}
	\left\{
	\begin{aligned}
		&\mathrm{d} y+\nabla\cdot (\mathcal {A}\nabla y) \mathrm{d}t= \left(\langle \textbf{a}, \nabla y\rangle+ \alpha y+ \phi  + \nabla\cdot \textbf{b} +\textbf{1}_{\mathcal {O}'}u\right)\mathrm{d}t+Y\mathrm{d}W_t ~\textrm{in}~\mathcal {O}_T,\\
	&	y=0~\textrm{on}~\Sigma_T,\quad
		y(T) =y_T~\textrm{in}~\mathcal {O},
	\end{aligned}
	\right.
\end{equation}
where the pair $(y,Y)$ is the unique solution associated with the control variable $u$ and the terminal state $y_T$. In \eqref{4.1}, we assume that $\textbf{a}\in L_\mathbb{F}^\infty(0,T;W^{1,\infty}(\mathcal {O};\mathbb{R}^n))$, $\alpha \in L_\mathbb{F}^\infty(0,T;L^\infty(\mathcal {O}))$, $ \phi \in L^2_\mathbb{F}(0,T;L^2(\mathcal {O}))$ and $  \textbf{b} \in L^2_\mathbb{F}(0,T;L^2(\mathcal {O};\mathbb{R}^n))$.

\begin{theorem} \label{lem4.1}
	Assume that the condition (A$_1$) holds. Then for each terminal state $y_T \in L^2_{\mathcal{F}_T}(\Omega;L^2(\mathcal {O}))$, there exists a control   $\hat{u}\in L^2_\mathbb{F}(0,T;L^2(\mathcal {O}'))$ such that the corresponding solution $(\hat{y}, \hat{Y})$ to \eqref{4.1} satisfies $
	\hat{y}(0)=0$ in $\mathcal {O}$, $\mathbb{P}$-a.s.
	Moreover, there exists a constant $C>0$ depending on $\mathcal {O}$ and $\mathcal {O}'$ such that
	\begin{equation}\label{4..3}
		\begin{split}
			&\mathbb{E}\int_{\mathcal {O}_T}  \theta^{-2}\big[\hat{y}^2   +  \lambda^{-2} \mu^{-2} \xi^{-3}  (|\nabla\hat{y}|^2   +  \hat{Y}^2) \big] \mathrm{d}x\mathrm{d}t +\mathbb{E}\int_{\mathcal {O}'_T} \lambda^{-3} \mu^{-4} \xi^{-3} \theta ^{-2}\hat{u}^2 \mathrm{d}x\mathrm{d}t\\
			& \leq C \Big[\lambda^{-1}\mu^{-2} e^{4\lambda \mu e^{6\mu (m+1)}-6\mu m}  \|y_T\|_{L_{\mathcal{F}_T}^2(\Omega ; L^2(\mathcal {O}))}^2+ \mathbb{E}\int_{\mathcal {O}_T} \big(\lambda^{-3} \mu^{-4} \xi^{-3} \theta ^{-2}\phi^2 \\
			&  +  \lambda^{-1} \mu^{-2} \xi^{-1} \theta ^{-2}|\textbf{b}|^2\big)\mathrm{d}x\mathrm{d}t\Big],
		\end{split}
	\end{equation}
	for all parameters $\lambda,\mu \geq 1$ sufficiently large.
\end{theorem}

\begin{remark}
	Theorem \ref{lem4.1} may be regarded as a stochastic version of the null controllability results obtained in \cite[Lemma 3.1]{imanuvilov2003carleman} and \cite[Lemma 2.1]{Fern2006Global}.  Another novelty is the estimate \eqref{4..3}, which provides an uniform bound for the quadruple $(\hat{y}, \nabla \hat{y},\hat{Y},\hat{u})$ in suitable weighted Sobolev spaces. As we shall see later, \eqref{4..3} plays an important role in defining a contraction mapping $\mathscr{K}$ (see \eqref{fff}) in a suitable weighted Banach space, which enables us to extend the linear null controllability in Theorem \ref{lem4.1} to the case of semi-linear SPDEs.
\end{remark}

\begin{theorem} \label{thm-nonlinear1}
	Assume that the conditions (A$_1$)-(A$_2$) hold. Then for any $y_T \in L^2_{\mathcal{F}_T}(\Omega;L^2(\mathcal {O}))$, there exists a control variable $u\in L^2_\mathbb{F}(0,T;L^2(\mathcal {O}'))$ such that the associated unique solution $(y, Y)$ to the controlled system
	\begin{equation}\label{jj6}
		\left\{
		\begin{aligned}
			&\mathrm{d} y+\nabla\cdot (\mathcal {A}\nabla y) \mathrm{d}t= \left(F(\omega,t,x,y,\nabla y,Y) +\textbf{1}_{\mathcal {O}'}u\right)\mathrm{d}t+Y\mathrm{d}W_t ~ \textrm{in}~~\mathcal {O}_T,\\
			&y=0~ \textrm{on}~\Sigma_T,\quad
			y(T) =y_T~~\textrm{in}~\mathcal {O}
		\end{aligned}
		\right.
	\end{equation}
	satisfies
	$  y(0)=0$ in $\mathcal {O}$, $\mathbb{P}$-a.s.
\end{theorem}

\begin{remark} \label{rem22}
Recently, Baroun et al. \cite{baroun2025null} studied the null controllability for linear parabolic SPDEs,  Theorem \ref{thm-nonlinear1}  extends their result (cf. \cite[Theorem 1.2]{baroun2025null}) to nonlinear settings, which is a non-trivial task as it necessitates the derivation of new Carleman estimates to overcome the difficulty caused by the lack of compact embedding in state spaces. Due to technique reasons, it still remains to be open to consider the controlled system with general nonlinearities, and we refer to \cite{Fern2006Exact,Fern2006Global,Kassab2020Null} for results concerning the deterministic PDEs with super-linear nonlinearities.
\end{remark}

Another closely related issue is to study the null controllability of semi-linear forward parabolic SPDEs. To do so, let us consider the backward SPDEs 
\begin{equation} \label{back}
	\left\{
	\begin{aligned}
		&\mathrm{d} z+\nabla\cdot (\mathcal {A}\nabla z) \mathrm{d}t= \left(\langle \textbf{c}, \nabla z\rangle+ \rho_1 z+ \rho_2Z+ \phi + \nabla\cdot \textbf{b} \right)\mathrm{d}t+Z\mathrm{d}W_t ~ \textrm{in}~\mathcal {O}_T,\\
		&z=0 ~ \textrm{on}~\Sigma_T,\quad
		z(T) =z_T ~ \textrm{in}~\mathcal {O},
	\end{aligned}
	\right.
\end{equation}
where $(z,Z)$ denotes the solution associated with   terminal data $z_T$. For the parameters in \eqref{back}, we assume that $\textbf{c}\in L_\mathbb{F}^\infty(0,T;L^\infty(\mathcal {O};\mathbb{R}^n))$, $\rho_1$, $\rho_2\in L_\mathbb{F}^\infty(0,T;L^\infty(\mathcal {O}))$, $\phi  \in L_\mathbb{F}^2(0,T;L^2(\mathcal {O}))$ and
$\textbf{b} \in L^2_\mathbb{F}( 0,T;L^2(\mathcal {O};\mathbb{R}^n))$.


\begin{theorem} \label{thm4}
	Assume that the assumption (A$_1$) holds, then there exist $\lambda_0>0$ and $\mu_0>0$,  depending only on $\mathcal {O},\mathcal {O}'$ and $T$,  such that the unique solution $(z,Z) \in \mathcal {W}_T \times L^2_\mathbb{F}(0,T;L^2(\mathcal {O}))$ of \eqref{back} with respect to $z_T\in L^2_{\mathcal {F}_T}(\Omega; L^2(\mathcal {O}))$ satisfies
	\begin{equation} \label{carleman3}
		\begin{split}
			&\mathbb{E}\int_\mathcal {O} \lambda \mu^2 e^{6\mu m} e^{2\lambda\varphi(0)}  z^2(0) \mathrm{d}x+\mathbb{E}\int_{\mathcal {O}_T} \mathring{\theta}^2\big(\lambda^3 \mu^4 \mathring{\xi}^3 z^2 +  \lambda \mu^2 \mathring{\xi}  |\nabla z|^2\big)\mathrm{d}x\mathrm{d}t\\
			&  \leq C\Big[\mathbb{E}\int_{\mathcal {O}'_T}\lambda ^{3}\mu ^{4}\mathring{\theta}^{2} \mathring{\xi}^{3}z^2 \mathrm{d}x \mathrm{d}t  +   \mathbb{E} \int_{\mathcal {O}_T} \mathring{\theta}^{2}\big[\phi^2  +\lambda^{2} \mu^{2} \mathring{\xi}^{3}  (Z^2 +  |\textbf{b}|^2) \big]  \mathrm{d}x \mathrm{d}t\Big],
		\end{split}
	\end{equation}
	for all $\lambda \geq \lambda_0$ and $\mu \geq \mu_0$, where the $\mathcal {C}^2$ weighted function $\mathring{\gamma}(t)$ is now defined by
	\begin{equation*} 
		\mathring{\gamma}(t)=\left\{
		\begin{aligned}
			&1+(1- 4T^{-1}t)^\sigma  && \textrm{in}~ {[0,T/4]},\\
			&1  && \textrm{in}~{[T/4,T/2]},\\
			&\textrm{is increasing}  && \textrm{in}~{[T/2,3T/4 ]},\\
			&(T-t)^{-m}  &&  \textrm{in}~{[3T/4 ,T)},
		\end{aligned}
		\right.
	\end{equation*}
	while
	$
	\mathring{\ell} $, $\mathring{\xi} $ and $\mathring{\theta} $ are defined
	by replacing $\gamma $ with $\mathring{\gamma} $ in \eqref{2.4}, respectively.
\end{theorem}

\begin{remark}
	Notice that the exponent of the weighted function $\mathring{\xi}$ in \eqref{carleman3} is cubic rather than  a quadratic one as that in \cite[Theorem 6.1]{tang2009null} (see also \cite[Theorem 3.1]{baroun2025null}), which was caused by the non-degeneracy of weighted function at $t=0$. Theorem \ref{thm4} improves the Carleman estimates in \cite{tang2009null} by considering a source term in Sobolev space of negative order. It will be of interest to extend the Carleman estimate \eqref{carleman3} to the stochastic fourth order parabolic system considered in \cite[Theorem 1.8]{wangyu2022}; see also \cite[Proposition 2.4]{Kassab2020Null} for recent deterministic results.
\end{remark}

\begin{theorem} \label{thm-nonlinear2}
	Assume that the conditions (A$_1$) and (A$_3$) hold. Then, for each initial state $y_0 \in L^2_{\mathcal {F}_0}(\Omega;L^2(\mathcal {O}))$, there exists a control pair $(u,U)\in L^2_\mathbb{F}(0,T;L^2(\mathcal {O}'))\times L^2_\mathbb{F}(0,T;L^2(\mathcal {O}))$ such that the unique solution $y$ to the system
	\begin{equation}\label{jj8}
		\left\{
		\begin{aligned}
			&\mathrm{d} y-\nabla\cdot (\mathcal {A}\nabla y) \mathrm{d}t= \left(F_1(\omega,t,x,y,\nabla y) +\textbf{1}_{\mathcal {O}'}u\right)\mathrm{d}t \\
			&\quad\quad\quad\quad\quad\quad  \quad~~+\left(F_2(\omega,t,x,y,\nabla y)+U\right)\mathrm{d}W_t ~~\textrm{in}~\mathcal {O}_T,\\ &y=0~~\textrm{on}~\Sigma_T,\quad
			y(0)=y_0~~\textrm{in}~\mathcal {O}
		\end{aligned}
		\right.
	\end{equation}
	satisfies $y(T)=0 $ in $\mathcal {O}$, $\mathbb{P}$-a.s.
\end{theorem}

\begin{remark}
	The controllability for \eqref{jj8} with general nonlinearities $F_1(\cdot)$ and $F_2(\cdot)$, such as the super-linear nonlinearity considered for deterministic parabolic PDEs \cite{Fern2006Exact,Fern2006Global,Kassab2020Null}, is still an interesting but challenging problem. Moreover, Theorem \ref{thm-nonlinear2} improves \cite[Theorem 1.1]{baroun2025null} to the nonlinear case.
\end{remark}

\begin{remark}
	As the control $U$ acts on the whole domain $\mathcal {O}$, the state $y$ still satisfies the controllability property with the control pair
	$(u,U^*)$, where $U^*= U-F_2(\omega,t,x,y,\nabla y) \in L^2_\mathbb{F}(0,T;L^2(\mathcal {O})).$ Note that the control $U^*$ is well-defined according to the condition (A$_3$) and the fact of $y\in L^2_\mathbb{F}(0,T;H^1_0(\mathcal {O}))$. Therefore, the proof of Theorem \ref{thm-nonlinear2} reduces to the case of $F_2(\cdot)\equiv 0$.
\end{remark}

\begin{remark}
	Theorem \ref{thm-nonlinear2} requires an extra control $U \in L^2_\mathbb{F}(0,T;L^2(\mathcal {O}))$ on the diffusion term, which is non-trivial due to the randomness of the coefficients $\mathcal {A}$, $F_1(\cdot)$ and  $F_2(\cdot)$. An open question is that whether system \eqref{jj8} is still null controllable without the control variable $U$ or if the control $U$ acts only on a sub-domain of $\mathcal {O}$.
\end{remark}
 
\subsection{Organization of the paper}
In Section 2, we shall establish the global Carleman estimates for the linear forward parabolic SPDEs with $L^2$-valued source terms, i.e, Theorem \ref{thm1}. Section 3 is devoted to the proof of the Carleman estimates for linear forward parabolic SPDEs with $H^{-1}$-valued source terms (i.e., Theorem \ref{thm2}), which is then applied to prove the null controllability for linear and semi-linear backward parabolic SPDEs in the Theorem \ref{lem4.1} and Theorem \ref{thm-nonlinear1}, respectively. In Section 4, we provide the sketch of the proof for  Theorem  \ref{thm4} and Theorem \ref{thm-nonlinear2}.

\section{An improved $L^2$-Carleman estimates}

\begin{proof}[Proof of Theorem \ref{thm1}] The proof is based on the method developed in \cite{tang2009null} and the techniques in \cite{hernandez2023global}; however, two differences require highlighting. First,  the random coefficients $ (a^{ij}(\omega,t,x))_{1\leq i,j\leq n}$ introduce additional terms that involve derivatives with respect to $t$- and $x$-variables, which indicates that the Carleman estimates developed for the Laplacian operator \cite{hernandez2023global} cannot be directly applied to the problem under consideration. Second, instead of the classical weighted function used in \cite{baroun2025null}, we have adopted a novel weighted function introduced in \cite{hernandez2023global}; see \cite{badra2016local} for the deterministic case, which renders \cite[Lemma 2.2]{baroun2025null} to be inapplicable to current case. 
	
\emph{Step 1.} Recall that $\theta = e^\ell$ and $\ell=\lambda \varphi$ in \eqref{2.4}, where $\varphi$ is defined in Subsection \ref{mainresult}.
Setting $h=\theta z$, we have
$\theta (\mathrm{d} z -\sum_{i,j} (a^{ij}z_{x_i})_{x_j}\mathrm{d}t )= I_1+I_2 \mathrm{d}t$,
where
$
I_1= \mathrm{d} h + 2  \sum_{i,j}a^{ij}\ell_{x_i}h_{x_j} \mathrm{d}t +2 \sum_{i,j}a^{ij}\ell_{x_ix_j} h \mathrm{d} t$, $
I_2= \mathcal {L} h-\sum_{i,j} (a^{ij}h_{x_i})_{x_j}$ with $
\mathcal {L}  =  \sum_{i,j}\big(a^{ij}_{x_j}\ell_{x_i}-a^{ij}\ell_{x_i}\ell_{x_j}-a^{ij} \ell_{x_ix_j}\big)- \ell_t$. 
By applying the same approach in \cite[Theorem 3.1]{tang2009null} and \cite[Theorem 9.26]{lu2021mathematical}, it is directly to use the It\^{o} formula (cf. \cite[Theorem 4.32]{da2014stochastic}) to derive the following fundamental identity:
\begin{fontsize}{10pt}{10pt}
	\begin{equation}\label{2.9}
\begin{split}
	&2 \mathbb{E}\int_{\mathcal {O}_T}\theta I_2\Big(\mathrm{d} z - \sum_{i,j} (a^{ij}z_{x_i})_{x_j}\mathrm{d}t\Big)\mathrm{d}x  =  \mathbb{E}\int_\mathcal {O}  \Big(\sum_{i,j}a^{ij}h_{x_i}h_{x_j}+\mathcal {L}h^2 \Big) (T) \mathrm{d}x \\
	& + 2\mathbb{E}\Big[\int_{\mathcal {O}_T}\Big(I_2^2+\nabla\cdot  V+  \sum_{i,j}B^{ij} h_{x_i}h_{x_j} \Big)\mathrm{d}x\mathrm{d}t -\int_{\mathcal {O}_T}\sum_{i,j} (a^{ij}h_{x_i} \mathrm{d} h )_{x_j}\mathrm{d}x\Big] \\
	&+ \mathbb{E}\int_{\mathcal {O}_T} A h^2   \mathrm{d}x\mathrm{d} t+4 \mathbb{E}\int_{\mathcal {O}_T}      \sum_{i,j}a^{ij}\ell_{x_ix_j} I_2 h  \mathrm{d}x\mathrm{d} t
	\\
	&+\mathbb{E}\int_{\mathcal {O}_T} \Big(-\sum_{i,j}a^{ij}\mathrm{d}h_{x_i} \mathrm{d}h_{x_j} -\mathcal {L} (\mathrm{d} h)^2\Big)\mathrm{d}x=    J_1+J_2+J_3+J_4+J_5,
\end{split}
\end{equation}
\end{fontsize}
where for all $i,j=1,...,n$,
\begin{fontsize}{9pt}{9pt}
	\begin{equation*} 
	\begin{split}
	&A=-2 \sum_{i,j}(\mathcal {L}a^{ij}\ell_{x_i})_{x_j} - \mathcal {L}_t , 
	B^{ij}=  \sum_{k,p}\Big(2 a^{ip}(a^{kj}\ell_{x_k})_{x_p}-(a^{ij}a^{kp}\ell_{x_k})_{x_p}-\frac{1}{2}\delta_{ik}\delta_{jp}a^{kp}_t\Big) ,\\
	&V^{j}= -  2 \sum_{i,k,p}a^{ij}a^{kp}\ell_{x_k}h_{x_i}h_{x_p} + \sum_{i,k,p}a^{ij}a^{kp}\ell_{x_i} h_{x_k}h_{x_p} +\sum_{i}\mathcal {L}a^{ij}\ell_{x_i} h^2,  V=(V^1,...,V^n)^T.
\end{split}
\end{equation*}\end{fontsize}
By comparing with \cite{hernandez2023global}, it is evident that the $a^{ij}$-, $a^{ij}_{x_k}$-, $a^{ij}_t$- and $a^{ij}_{x_kt}$-terms, which are absent in the case of the Laplacian operator, emerge in $I_2$, $A$, $B^{ij}$ and $V^j$. This presence introduces additional difficulty to the estimation of the R.H.S. of \eqref{2.9}.

\emph{Step 2.} We shall estimate $J_i$ ($i=1,2,3,4,5$) by suitable bounds from below.

\textsc{Estimate for $J_1$.} From the definition of $\ell$, we have
	$\ell_{t}(T)  
	 \leq -C \lambda^2 \mu^3 e^{2\mu(6m+1)}$. Moreover, there hold $
	\ell_{x_i}
	=\lambda\mu\beta _{x_i}\xi$, $
	\ell_{x_it}   =\frac{\gamma_t }{\gamma}\lambda\mu\beta _{x_i}\xi$, $\ell_{x_ix_j}
	=\lambda\mu^2 \beta _{x_i}\beta _{x_j} \xi+\lambda\mu\beta _{x_ix_j}\xi = \lambda\mu^2 \beta _{x_i}\beta _{x_j} \xi+\lambda\xi O(\mu)$ and
	$\ell_{x_ix_jt} = \frac{ \gamma_t }{\gamma}(\lambda\mu\beta _{x_ix_j}\xi+ \lambda\mu^2\beta _{x_i} \beta_{x_j}\xi)$,
which indicate that
\begin{equation}\label{2.12}
\begin{split}
 |\ell_{x_ix_j}(T) |&\leq C \lambda \mu^2 e^{2\mu(6m+1)}, ~~~ |(\ell_{x_i} \ell_{x_j})(T)|  \leq C \lambda^2\mu^2 e^{2\mu(6m+1)}.
\end{split}
\end{equation}
By assumption (A$_1$), we have $\sum_{i,j}a^{ij} h_{x_i}h_{x_j}(T)\geq c_0|\nabla h(T)|^2$, then
\begin{equation}\label{2.10}
\begin{split}
	&\Big(\sum_{i,j}a^{ij}h_{x_i}h_{x_j}+\mathcal {L}h^2 \Big) (T)\\
	&\geq   c_0|\nabla h(T)|^2- \ell_t(T)h^2(T) - \sum_{i,j}\big(a^{ij} \ell_{x_i} \ell_{x_j} +a^{ij} \ell_{x_ix_j} -a^{ij}_{x_j} \ell_{x_i} \big)(T)h ^2(T).
\end{split}
\end{equation}
Particular attention must be paid to the last term on the R.H.S. of \eqref{2.10}, which does not appear in \cite{hernandez2023global}. Indeed, by using the estimates \eqref{2.12}-\eqref{2.10}, we get
$
( \sum_{i,j}a^{ij}h_{x_i}h_{x_j}+\mathcal {L}h^2 ) (T)  
\geq  c_0|\nabla h(T)|^2-C \lambda^2 \mu^3 e^{\mu(12m+2)}h^2(T),
$
which implies that
\begin{equation}\label{2.13}
\begin{split}
	J_1 \geq  c_0\int_\mathcal {O}|\nabla h(T)|^2\mathrm{d}x - C \int_\mathcal {O}\lambda^2 \mu^3 e^{\mu(12m+2)}  h^2(T)  \mathrm{d}x.
\end{split}
\end{equation}

\textsc{Estimate for $J_2$.}  Unlike the Laplacian operator with constant coefficients, the general divergence-type operator, being dependent on $(\omega,t,x)$, render the components $V^j$ and $B^{ij}$ more complex, thus necessitating careful analysis.  Indeed, by the Dirichlet boundary condition $z|_{\Sigma_T}=0$ and the construction of the weighted function $\beta$, we infer that  $ h |_{\Sigma_T}=0$ and $\frac{\partial \beta}{\partial \nu}|_{\Sigma_T}\leq 0$. It then follows from the Divergence Theorem that
\begin{fontsize}{9pt}{9pt}\begin{equation} \label{2.14}
\begin{split}
	 &\int_{\mathcal {O}_T} \nabla\cdot  V  \mathrm{d}x\mathrm{d}t-\int_{\mathcal {O}_T}\sum_{i,j}(a^{ij}h_{x_i} \mathrm{d} h )_{x_j}  \mathrm{d}x \\
	&  =  \int_{\Sigma_T} \lambda\mu\xi \sum_{i,j,k,p}a^{ij} a^{kp}\Big(-  2 \frac{\partial\beta}{\partial \nu} \nu^k  \frac{\partial h}{\partial \nu}  \nu^i\frac{\partial h}{\partial \nu} \nu^p + \frac{\partial\beta}{\partial \nu} \nu^i  \frac{\partial h}{\partial \nu}  \nu^k\frac{\partial h}{\partial \nu} \nu^p \Big) \nu^j\mathrm{d}\sigma\mathrm{d}t \\
	  &=  \int_{\Sigma_T}\Big(\sum_{i,j}a^{ij} \nu^i \nu^j\Big)^2\lambda\mu\xi (-\frac{\partial\beta}{\partial \nu}) (\frac{\partial h}{\partial \nu} )^2\mathrm{d}\sigma\mathrm{d}t\geq  0.
\end{split}
\end{equation}
\end{fontsize}
	Note that
$
B^{ij}
= \lambda\mu^2\xi\sum_{k,p}(2 a^{ip} a^{kj} -a^{ij}a^{kp}) \beta _{x_k}\beta _{x_p} -\lambda\xi O(\mu)-O(1)$,
we have
\begin{equation}\label{2.15}
\begin{split}
	&2\mathbb{E}\int_{\mathcal {O}_T}\sum_{i,j}B^{ij} h_{x_i}h_{x_j}\mathrm{d}x\mathrm{d}t
	\geq 2\mathbb{E}\int_{\mathcal {O}_T}(\lambda\xi O(\mu)+O(1)) |\nabla h|^2 \mathrm{d}x\mathrm{d}t\\
	&-2\mathbb{E}\int_{\mathcal {O}_T} \lambda\mu^2\xi \Big(\sum_{k,p}a^{kp} \beta _{x_k}\beta _{x_p}  \Big) \Big(\sum_{i,j}a^{ij} h_{x_i}h_{x_j} \Big)\mathrm{d}x\mathrm{d}t.
\end{split}
\end{equation}
Therefore, we deduce from \eqref{2.14} and \eqref{2.15} that
\begin{equation}\label{2.16}
\begin{split}
	J_2\geq
	    2\mathbb{E}\int_{\mathcal {O}_T}I_2^2\mathrm{d}x\mathrm{d}t-C\mathbb{E}\int_{\mathcal {O}_T}\left(\lambda\xi O(\mu^2)+\lambda\xi O(\mu)+O(1)\right) |\nabla h|^2 \mathrm{d}x\mathrm{d}t.
\end{split}
\end{equation}

\textsc{Estimate for $J_3$.} The estimate corresponding to this term includes several extra terms on the R.H.S of \eqref{2.17} below, which are absent when dealing with the Laplacian with constant coefficients, we have to carefully control the low-order terms by using the high-order ones.  We first observe that
\begin{equation}\label{2.17}
\begin{split}
	A =  -J_{31}-J_{32}+J_{33},
\end{split}
\end{equation}
with $J_{31}=  \sum_{k,p}2(a^{kp}\mathcal {L}_{x_k}\ell_{x_p}+\mathcal {L}a^{kp}\ell_{x_kx_p}+a^{kp}_{x_p}\mathcal {L}\ell_{x_k})$, $ J_{32}=\sum_{k,p}(a^{kp}_{x_pt}\ell_{x_k}-a^{kp}_t\ell_{x_k}\ell_{x_p}-a^{kp}_t \ell_{x_kx_p})$ and  $J_{33}=\ell_{tt}- \sum_{k,p} (a^{kp}_{x_p}\ell_{x_kt}-a^{kp}(\ell_{x_k}\ell_{x_p})_t-a^{kp} \ell_{x_kx_pt} )$.
From the definition of $\mathcal {L}$ and the property \eqref{2.12}, we infer that
\begin{equation}\label{2.18}
\begin{split}
	\mathcal {L}
	&=-\sum_{i,j}\lambda^2\mu^2 a^{ij}\beta _{x_i}\beta _{x_j}\xi^2+  \lambda\xi O(\mu^2)  - \frac{ \gamma_t }{\gamma}\lambda \varphi.
\end{split}
\end{equation}
In a similar manner, since $\ell_{x_kt}=\frac{ \gamma_t }{\gamma}\lambda \mu \beta_{x_k}\xi$, we have
\begin{equation}\label{2.19}
\begin{split}
	\mathcal {L}_{x_k} &=-2\sum_{i,j}\lambda^2\mu^3 a^{ij}\beta _{x_i}\beta _{x_j}\beta_{x_k} \xi^2 +\lambda^2\xi^2 O(\mu^2) +  \lambda \xi O(\mu^3)  - \frac{ \gamma_t }{\gamma}\lambda \mu \beta_{x_k}\xi.
\end{split}
\end{equation}
For $J_{31}$, we get by \eqref{2.18} and \eqref{2.19} that
\begin{equation*}
\begin{split}
	J_{31}
	 &=6\sum_{i,j,k,p}\lambda^3\mu^4 a^{ij}a^{kp}\beta _{x_i}\beta _{x_j}\beta _{x_k}\beta _{x_p}  \xi^3
	+ \lambda^3 \xi^3 O(\mu^3)
	+  \lambda^2 \xi^2 O(\mu^4)
	\\
	&+\frac{ \gamma_t }{\gamma}\Big(2\lambda^2\mu^2\sum_{k,p} a^{kp} \beta _{x_k} \beta _{x_p}\xi^2
	+2\lambda^2\mu^2\sum_{k,p}a^{kp}\beta _{x_k}\beta _{x_p} \varphi\xi + \lambda^2\xi\varphi O(\mu)\Big).
\end{split}
\end{equation*}
For $J_{32}$, simple calculation shows that 
$
		J_{32}
 =\sum_{k,p}[-\lambda\mu a^{kp}_{t x_p}\beta _{x_k}\xi+\lambda^2\mu^2 a^{kp}_t\beta _{x_k}  \beta _{x_p}\xi^2+a^{kp}_t (\lambda\mu^2 \beta _{x_k}\beta _{x_p} \xi+\lambda\xi O(\mu))] 
 =\lambda^2\xi^2O(\mu^2)+\lambda\xi O(\mu^2).
$ For $J_{33}$, one can verify that
$|\gamma_{tt}| \leq C \gamma^3$ for $t\in (0,T/2]$; $
	\gamma_{tt}\equiv 0$  for $t\in [T/2,3T/4]$; $|\gamma_{tt}| \leq C \lambda^2 \mu^4e^{2\mu (6m-4)}$  for $t\in [3T/4,T]$, which implies that $ |\ell_{tt}|\leq C \lambda^3\mu^2\xi^3 $ for all $t \in [0,T]$, and hence
	$$J_{33}
	\geq -C \lambda^3\mu^2\xi^3 +\sum_{k,p}\frac{ \gamma_t }{\gamma}\left(2\lambda^2\mu^2 a^{kp} \beta _{x_k} \beta _{x_p}\xi^2  +\lambda\xi O(\mu^2)\right).
$$

From the above estimates for $J_{31}$-$J_{33}$ and   assumption (A$_1$), we obtain
\begin{fontsize}{9.1pt}{9.1pt}
\begin{equation}\label{2.20}
\begin{split}
	J_3&\geq \mathbb{E}\int_{\mathcal {O}_T} \sum_{i,j}\big(6\lambda^3\mu^4 (a^{ij}\beta _{x_i}\beta _{x_j})^2 \xi^3+ \lambda^3 \xi^3 O(\mu^3)
	+  \lambda^2 \xi^2 O(\mu^4)\big)h^2\mathrm{d}x\mathrm{d}t+\mathbb{E}\int_{\mathcal {O}_T}\frac{ \gamma_t }{\gamma} \\
	&\times\sum_{k,p}\big[4\lambda^2\mu^2 a^{kp} \beta _{x_k} \beta _{x_p}\xi^2
	+2\lambda^2\mu^2a^{kp}\beta _{x_k}\beta _{x_p} \varphi\xi  + \lambda^2\xi\varphi O(\mu)+\lambda\xi O(\mu^2)\big]h^2\mathrm{d}x\mathrm{d}t.
\end{split}
\end{equation}
\end{fontsize} 

\textsc{Estimate for $J_4$.} The challenge in estimating this term stems from the need to handle the derivatives of the coefficients included in the summation $I_2$ and $\mathcal {L}$. The key lies in determining the explicit high-order terms that are necessary for the subsequent steps.
By virtue of the definition of $ I_2$ and the property \eqref{2.12}, we infer that
\begin{equation}\label{2.21}
\begin{split}
	4 \sum_{i,j}a^{ij}\ell_{x_ix_j} I_2 h= J_4^1+ J_4^2+ J_4^3,
\end{split}
\end{equation}
where $J_4^1=\sum_{i,j} 4\lambda\mu^2 a^{ij} \beta _{x_i}\beta _{x_j} \xi\mathcal {L} h^2- \lambda\xi \mathcal {L}h^2 O(\mu)  $, $J_4^2=-\sum_{i,j,k,p} (a^{kp}h_{x_k}h)_{x_p}(4 \lambda\mu^2$ $a^{ij} \beta _{x_i}\beta _{x_j} \xi+\lambda\xi O(\mu) )$ and $J_4^3=\sum_{i,j,k,p}a^{kp}h_{x_k}h_{x_p}(4 \lambda\mu^2 a^{ij} \beta _{x_i}\beta _{x_j} \xi+\lambda\xi O(\mu))$.

For $J^1_4$, we get by \eqref{2.18} that
\begin{equation*}
\begin{split}
	J_4^1
	=&-\sum_{i,j}\left[
	4 \lambda^3\mu^4 ( a^{ij}\beta _{x_i}\beta _{x_j})^2\xi^3
	+\lambda^2\xi ^2O(\mu^4)+\lambda^3  \xi^3 O(\mu^3)+\lambda^2\xi^2 O(\mu^3)\right] h^2
	\\
	&+\frac{ \gamma_t }{\gamma}\sum_{i,j}\left[-4\lambda^2\mu^2 a^{ij}\beta _{x_i}\beta _{x_j} \varphi \xi  +\lambda^2\xi \varphi O(\mu)\right]h^2.
\end{split}
\end{equation*}
For $J^2_4$, by using the identities $fg_{x_p}=(fg)_{x_p}-f_{x_p}g$ and $\xi_{x_p}=\mu \beta_{x_p} \xi$, we have
\begin{fontsize}{8pt}{8pt}\begin{equation*}
\begin{split}
	J_4^2=& - \sum_{i,j,k,p}\Big(\big[a^{kp}h_{x_k}h\big(4 \lambda\mu^2 a^{ij} \beta _{x_i}\beta _{x_j} \xi+\lambda\xi O(\mu)\big)\big]_{x_p}-a^{kp}h_{x_k}h\big(4 \lambda\mu^2 a^{ij} \beta _{x_i}\beta _{x_j} \xi+\lambda\xi O(\mu)\big)_{x_p}\Big)\\
	=&- \sum_{i,j,k,p}\left[a^{kp}h_{x_k}h\left(4 \lambda\mu^2 a^{ij} \beta _{x_i}\beta _{x_j} \xi+\lambda\xi O(\mu)\right)\right]_{x_p}\\
	&+ \sum_{i,j,k,p}a^{kp}h_{x_k}h\left[4 \lambda\mu^2 (a^{ij} \beta _{x_i}\beta _{x_j})_{x_p} \xi+4 \lambda\mu^3 a^{ij} \beta _{x_i}\beta _{x_j}\beta_{x_p}\xi+\lambda\xi O(\mu)\right]\\
	\geq&- \sum_{i,j,k,p}\left[a^{kp}h_{x_k}h\left(4 \lambda\mu^2 a^{ij} \beta _{x_i}\beta _{x_j} \xi+\lambda\xi O(\mu)\right)\right]_{x_p}- C\mu^2|\nabla h|^2 -C\lambda^2 \mu^4\xi^2 h^2,
\end{split}
\end{equation*}
\end{fontsize}where the last inequality used the fact that
\begin{equation*}
\begin{split}
	&\sum_{i,j,k,p}a^{kp}h_{x_k}h\left[4 \lambda\mu^2 (a^{ij} \beta _{x_i}\beta _{x_j})_{x_p} \xi+4 \lambda\mu^3 a^{ij} \beta _{x_i}\beta _{x_j}\beta_{x_p}\xi+\lambda\xi O(\mu)\right]\\
	&  \leq C\left(\lambda\mu\xi |h||\nabla h| + \lambda\mu^3 \xi|h||\nabla h|+ \lambda\mu^2 \xi|h||\nabla h| \right) \leq  C(\mu^2|\nabla h|^2 +\lambda^2 \mu^4\xi^2 h^2).
\end{split}
\end{equation*}
For $J^3_4$, a straightforward application of the assumption (A$_1$) results in the inequality 
$
	J_4^3
	\geq 4  \sum_{i,j,k,p}\lambda\mu^2\xi a^{kp}h_{x_k}h_{x_p} a^{ij} \beta _{x_i}\beta _{x_j}- \lambda \xi O(\mu)|\nabla h|^2.
$

Putting the estimates for $J^1_4$-$J^3_4$ into \eqref{2.21}, integrating by parts for the resulted inequality over $\mathcal {O}_T$ and using the fact of $h|_{\Sigma_T}=0$, we obtain
\begin{fontsize}{9.6pt}{9.6pt}\begin{equation}\label{2.22}
\begin{split}
	&J_4 \geq -\mathbb{E}\int_{\mathcal {O}_T}\sum_{i,j}\left[
	4 \lambda^3\mu^4 ( a^{ij}\beta _{x_i}\beta _{x_j})^2\xi^3
	+\lambda^2\xi^2 O(\mu^4)+\lambda^3  \xi^3 O(\mu^3)+\lambda^2\xi^2 O(\mu^4)\right] \\
	& \times h^2\mathrm{d}x\mathrm{d} t+ \mathbb{E}\int_{\mathcal {O}_T}\sum_{i,j}\frac{ \gamma_t }{\gamma}\left[-4\lambda^2\mu^2 a^{ij}\beta _{x_i}\beta _{x_j} \varphi \xi +\lambda^2\xi \varphi O(\mu)\right]h^2\mathrm{d}x\mathrm{d} t\\
	& + \mathbb{E}\int_{\mathcal {O}_T} \sum_{i,j,k,p}4\lambda\mu^2\xi a^{kp}h_{x_k}h_{x_p} a^{ij} \beta _{x_i}\beta _{x_j} \mathrm{d}x\mathrm{d} t- \mathbb{E}\int_{\mathcal {O}_T}[ O(\mu^2)+ \lambda\xi O(\mu)]|\nabla h|^2\mathrm{d}x\mathrm{d} t.
\end{split}
\end{equation}
\end{fontsize}

\textsc{Estimate for $J_5$.} Since $h=\theta z$, we have $
h_{x_i}= \theta (\ell_{x_i} z+ z_{x_i})$. Moreover, it follows from the $z$-equation that
$
\mathrm{d}h_{x_i}=[\cdots]\mathrm{d}t+ \theta \left(\ell_{x_i} \phi_2 +\phi_{2,x_i}  \right)\mathrm{d}W_t,
$
which implies 
$
\sum_{i,j}a^{ij}\mathrm{d}h_{x_i} \mathrm{d}h_{x_j}= \sum_{i,j}\theta^2 a^{ij} \left(\ell_{x_i} \phi_2 +\phi_{2,x_i}  \right)\left(\ell_{x_j} \phi_2 +\phi_{2,x_j}  \right)\mathrm{d}t.
$
By using \eqref{2.18} and the fact of $|\varphi_t|\leq C \lambda\mu \xi^3$, for all $(t,x)\in \mathcal {O}_T$, we have
\begin{equation}\label{2.23}
\begin{split}
	J_5
	=& -\mathbb{E}\int_{\mathcal {O}_T}\theta^2\sum_{i,j} a^{ij} (\lambda\mu\beta_{x_i}\xi \phi_2 +\phi_{2,x_i}  )(\lambda\mu\beta_{x_j}\xi \phi_2 +\phi_{2,x_j}  ) \mathrm{d}x\mathrm{d}t \\
	&-\mathbb{E}\int_{\mathcal {O}_T} \theta^2 \phi_2^2\Big[\sum_{i,j} (a^{ij}_{x_j}\ell_{x_i}-a^{ij}\ell_{x_i}\ell_{x_j}-a^{ij} \ell_{x_ix_j}) - \lambda\varphi_t\Big]  \mathrm{d}x\mathrm{d}t\\
	\geq & -C\mathbb{E}\int_{\mathcal {O}_T} \lambda^2 \mu^2  \xi ^3 \theta^2\phi_2^2 \mathrm{d}x\mathrm{d}t-C\mathbb{E}\int_{\mathcal {O}_T}\theta^2  | \nabla \phi_2|^2  \mathrm{d}x\mathrm{d}t.
\end{split}
\end{equation} 

Putting the above estimates for $J_i$ ($i=1,...,5$) together, we get from \eqref{2.9} that
\begin{fontsize}{9.8pt}{9.8pt}\begin{subequations}\label{2.24}
\begin{align}
&2 \mathbb{E}\int_{\mathcal {O}_T}\theta I_2\Big(\mathrm{d} z - \sum_{i,j} (a^{ij}z_{x_i})_{x_j}\mathrm{d}t\Big)\mathrm{d}x \label{2.24a}\\
&  \geq 2\mathbb{E}\int_{\mathcal {O}_T}I_2^2\mathrm{d}x\mathrm{d}t +\mathbb{E}\int_\mathcal {O}\big(c_0|\nabla h(T)|^2 + C  \lambda^2 \mu^3 e^{2\mu(6m+1)}h^2(T)\big) \mathrm{d}x \nonumber\\
&  +\mathbb{E}\int_{\mathcal {O}_T} \left[2c_0^2\lambda^3\mu^4\xi^3 |\nabla \beta|^4- \lambda^3 \xi^3 O(\mu^3)
- \lambda^2 \xi^2 O(\mu^4)\right]h^2\mathrm{d}x\mathrm{d}t\label{2.24b}\\
&   +\mathbb{E}\int_{\mathcal {O}_T} \left[2c_0^2\lambda\mu^2\xi |\nabla\beta|^2 -\lambda\xi O(\mu)-O(\mu^2)\right]|\nabla h|^2\mathrm{d}x\mathrm{d}t \label{2.24c}\\
&  -C\mathbb{E}\int_{\mathcal {O}_T} \theta^2(\lambda^2 \mu^2  \xi ^3 \phi_2^2+  | \nabla \phi_2|^2 ) \mathrm{d}x\mathrm{d}t\nonumber\\
&  + \mathbb{E}\int_{\mathcal {O}_T}\sum_{i,j} \frac{ \gamma_t }{\gamma}\Big[-2\lambda^2\mu^2 a^{ij}\beta _{x_i}\beta _{x_j} \varphi \xi+ 4\lambda^2\mu^2 a^{ij} \beta _{x_i} \beta _{x_j}\xi^2
 \label{2.24d}\\
 &-\lambda^2\xi\varphi O(\mu)-\lambda\xi O(\mu^2)\Big]h^2\mathrm{d}x\mathrm{d}t.\nonumber
\end{align}
\end{subequations}
\end{fontsize}

Let us deal with the terms on the R.H.S. of the last inequality. First, by using the equation satisfied by $z$ and the Cauchy inequality, we have
\begin{equation}\label{2.25}
\begin{split}
\eqref{2.24a}
&\leq  \mathbb{E}\int_{\mathcal {O}_T}I_2^2\mathrm{d}x\mathrm{d}t+  C\mathbb{E}\int_{\mathcal {O}_T}\theta ^2 \left(|\nabla z|^2+  z^2+ \phi_1^2 \right)  \mathrm{d}x\mathrm{d}t.
\end{split}
\end{equation}
By using the fact of $\inf_{x \in \mathcal {O} \backslash \overline{\mathcal {O}}_1}|\nabla\beta(x)|\geq a_0>0$, we deduce that
\begin{equation*}
\begin{split}
\eqref{2.24b}\geq  2a_0^4c_0^2\mathbb{E}\int_0^T\int_{\mathcal {O} \backslash \overline{\mathcal {O}}_1} \lambda^3\mu^4\xi^3  h^2\mathrm{d}x\mathrm{d}t- \mathbb{E}\int_{\mathcal {O}_T}  [ \lambda^3 \xi^3 O(\mu^3)+ \lambda^2 \xi^2 O(\mu^4) ]h^2\mathrm{d}x\mathrm{d}t,
\end{split}
\end{equation*}
which implies that, for $\lambda,\mu>0$ large enough,
\begin{equation}\label{2.26}
\begin{split}
\eqref{2.24b}\geq \frac{3}{2}a_0^4c_0^2\mathbb{E} \int_{\mathcal {O}_T} \lambda^3\mu^4\xi^3  h^2\mathrm{d}x\mathrm{d}t-2a_0^4c_0^2\mathbb{E}\int_0^T\int_{\mathcal {O}_1} \lambda^3\mu^4\xi^3  h^2\mathrm{d}x\mathrm{d}t.
\end{split}
\end{equation}
In a similar manner, we also have for $\lambda,\mu>0$ large enough
\begin{equation}\label{2.27}
\begin{split}
\eqref{2.24c}\geq \frac{3}{2}a_0^2c_0^2 \mathbb{E} \int_{\mathcal {O}_T} \lambda\mu^2\xi |\nabla h|^2\mathrm{d}x\mathrm{d}t-2a_0^2c_0^2\mathbb{E}\int_0^T\int_{\mathcal {O}_1} \lambda\mu^2\xi |\nabla h|^2    \mathrm{d}x\mathrm{d}t.
\end{split}
\end{equation} 

The last integral on the R.H.S. of \eqref{2.24} will be divided into three parts with respect to $t$-variable, i.e., $(0,T/2]\cup [T/2,3T/4] \cup [3T/4,T]$.  We utilize an approach similar to that in \cite[Theorem 2.1]{hernandez2023global}, which focuses on the backward SPDEs.  Notably, given the   difference in weighted function, special attention must be paid to both the sign of the derivative and the partitioning of intervals. For simplicity, we use $\eqref{2.24d}|_{[a,b]}$ to represent the integral \eqref{2.24d} restricted to the subset $[a,b]\subseteq [0,T]$.
\begin{itemize} [itemsep=0pt, leftmargin=28pt]
\item [$\bullet$] Recall that $m\geq 1$, we have $
| \gamma_t  |= m\gamma^{1+\frac{1}{m}}  \leq C \gamma^2 $, for all $t \in (0,T/4]$. Since $\gamma (t)\geq 1$ is a decreasing $\mathcal {C}^2$-function on $[T/4,T/2]$, there must be a constant $C>0$ such that $\max_{t\in [T/4,T/2]}| \gamma_t (t)| \leq C \min_{t\in [T/4,T/2]}\gamma^2(t)$. In both cases, we find that
$
| \gamma_t (t)| \leq C \gamma^2(t),
$
for all $t\in(0,T/2]$ and some constant $C>0$. Moreover, by the definition of $\varphi$ and $\xi$, we have
$
|\gamma \varphi|= \gamma^2 (\mu e^{6\mu(m+1)}- e^{\mu(\beta+6m)})
\leq \mu \xi^2 \frac{e^{6\mu(m+1)}}{e^{2\mu (\beta(x)+6m)}}\leq \mu \xi^2.
$
Therefore, by using the inequality $\xi^a\leq \xi^b$ for any $b>a>0$, the assumption (A$_1$) and the boundedness of $\nabla\beta$, we get
\begin{equation*} 
\begin{split}\hspace{-5mm}
	\eqref{2.24d}|_{(0,T/2]}
	&\geq - C\mathbb{E}\int_{\mathcal {O}_T}\gamma\Big( \lambda^2\mu^2  |\varphi| \xi+\lambda^2\mu^2  \xi^2+\lambda^2\xi|\varphi| O(\mu)+\lambda\xi O(\mu^2)\Big)h^2\mathrm{d}x\mathrm{d}t\\
	&\geq - C\mathbb{E}\int_{\mathcal {O}_T} \lambda^2\xi^3O(\mu^3) h^2\mathrm{d}x\mathrm{d}t,
\end{split}
\end{equation*}
where the second inequality used the fact that $0<\gamma(t)\leq \xi(t)$, for all $t\in (0,T]$.

\item [$\bullet$]  Since for any $t\in [T/2,3T/4]$, $\gamma(t)\equiv1$, we have
\begin{equation*} 
\begin{split}
	\eqref{2.24d}|_{[T/2,3T/4]}\equiv 0 .
\end{split}
\end{equation*}

\item [$\bullet$] For any $t\in [3T/4,T]$, it follows from the definition of the function $\gamma$ and $\varphi$ that
$
\varphi(t)<0$,   $\gamma_t (t)= \frac{4}{T}\sigma\left(1- 4(T-t)/T \right)^{\sigma-1}\in [0,4\sigma/T]$  and $\gamma(t) \in [1,2]$,
which yields that $- \gamma_t \varphi \geq 0$ on $ [3T/4,T]$,  and
\begin{equation*}
\begin{split}\hspace{-2mm}
	\eqref{2.24d}|_{[3T/4,T]}&\geq  \frac{1}{2}\mathbb{E}\int_{3T/4}^T \int_{\mathcal {O}}\gamma_t\left(2 c_0\lambda^2\mu^2 |\nabla\beta|^2 |\varphi| \xi+ 4c_0\lambda^2\mu^2|\nabla\beta|^2\xi^2\right)h^2\mathrm{d}x\mathrm{d}t
	\\
	& -\mathbb{E}\int_{3T/4}^T\int_{\mathcal {O}}\gamma_t\left(\lambda^2\xi|\varphi| O(\mu)+\lambda\xi O(\mu^2)\right)h^2\mathrm{d}x\mathrm{d}t.
\end{split}
\end{equation*}
\end{itemize}
According to the above discussion and the property of $\beta$ in Lemma \ref{weight}, we obtain
\begin{equation}\label{2.31}
\begin{split}
&\eqref{2.24d}\geq  c_0\mathbb{E}\int_{3T/4}^T\int_{\mathcal {O} }\lambda^2\mu^2 \gamma_t\left( |\varphi| \xi+   \xi^2\right)h^2\mathrm{d}x\mathrm{d}t\\
&-C\mathbb{E}\int_{3T/4}^T\int_{\mathcal {O} _1}\lambda^2\mu^2\gamma_t\left(  |\varphi| \xi+ \xi^2\right)h^2\mathrm{d}x\mathrm{d}t - C\mathbb{E}\int_{\mathcal {O}_T} \lambda^2\xi^3O(\mu^3) h^2\mathrm{d}x\mathrm{d}t .
\end{split}
\end{equation}

Putting the estimates \eqref{2.25}-\eqref{2.31} together,  we arrive at
\begin{equation}\label{2.32}
\begin{split}
&\mathbb{E} \int_{\mathcal {O}_T} \lambda^3\mu^4\xi^3  h^2\mathrm{d}x\mathrm{d}t+ \mathbb{E} \int_{\mathcal {O}_T} \lambda\mu^2\xi |\nabla h|^2\mathrm{d}x\mathrm{d}t +\mathbb{E}\int_\mathcal {O}|\nabla h(T)|^2\mathrm{d}x  \\
&  +  \mathbb{E}\int_\mathcal {O}\lambda^2 \mu^3 e^{2\mu(6m+1)} h^2(T)  \mathrm{d}x + \mathbb{E}\int_{3T/4}^T\int_{\mathcal {O} }\lambda^2\mu^2 \gamma_t\left( |\varphi| \xi+   \xi^2\right)h^2\mathrm{d}x\mathrm{d}t\\
&  \leq C\Big[\mathbb{E}\int_{3T/4}^T\int_{\mathcal {O}_1}\lambda^2\mu^2\gamma_t (  |\varphi| \xi+ \xi^2 )h^2\mathrm{d}x\mathrm{d}t +\mathbb{E}\int_0^T\int_{\mathcal {O}_1} \lambda^3\mu^4\xi^3  h^2\mathrm{d}x\mathrm{d}t\\
&+ \mathbb{E}\int_0^T\int_{\mathcal {O}_1} \lambda\mu^2\xi |\nabla h|^2    \mathrm{d}x\mathrm{d}t  + \mathbb{E}\int_{\mathcal {O}_T} \theta^2\left( \lambda^2 \mu^2  \xi ^3\phi_2^2 + | \nabla \phi_2|^2 +  \phi_1 ^2\right)  \mathrm{d}x\mathrm{d}t\Big],
\end{split}
\end{equation}
for any $\lambda$, $\mu>0$ large enough.

\emph{Step 3.} Let us first transform \eqref{2.32} by using the solution $z$ of \eqref{2.1}. Since $h= \theta z$, we have the equivalence\footnote{Note that $\nabla h =\theta(\nabla z+\lambda \mu \xi z \nabla \beta)$, $\nabla z= \theta^{-1}(\nabla h-\lambda\mu \xi h \nabla \beta )$ and $\max_{x\in\overline{\mathcal {O}}} |\nabla \beta(x)|\leq c_1$, for some $c_1>0$. By using the basic inequality $(a+b)^2\leq2a^2+2b^2$, straightforward  calculation leads to $\theta^2(| \nabla z |^2 + \lambda^2 \mu^2 \xi^2 z^2 ) \leq 2 |\nabla h|^2+(2c_1^2+1)\lambda^2\mu^2\xi^2h^2$ and
	$| \nabla h |^2 + \lambda^2 \mu^2 \xi^2 h^2 \leq 2 \theta^2|\nabla z|^2+(2c_1^2+1)\lambda^2\mu^2\theta^2\xi^2z^2$, which thereby yield the desired equivalence.}
$\theta^2(| \nabla z |^2 + \lambda^2 \mu^2 \xi^2 z^2 ) \approx  | \nabla h |^2 + \lambda^2 \mu^2 \xi^2 h^2$. By the definition of $\xi$, there holds $\xi (T,x) \leq e^{2\mu(6m+1)}$ for all $x\in \mathcal {O}$. It   follows from \eqref{2.32} that
\begin{equation}\label{2.34}
\begin{split}
&\mathbb{E} \int_{\mathcal {O}_T} \lambda\mu^2\xi \theta^2\left(| \nabla z |^2 + \lambda^2 \mu^2 \xi^2 z^2 \right)\mathrm{d}x\mathrm{d}t +\mathbb{E}\int_{\mathcal {O}}\theta^2 (T)|\nabla z(T)|^2\mathrm{d}x \\
&  +  \mathbb{E}\int_{\mathcal {O}}\lambda^2 \mu^3 e^{2\mu(6m+1)}(\theta^2z^2)(T) \mathrm{d}x + \mathbb{E} \int_{3T/4}^T\int_{\mathcal {O} }\lambda^2\mu^2 \theta^2\gamma_t \left( |\varphi| \xi+   \xi^2\right)z^2\mathrm{d}x\mathrm{d}t \\
&  \leq C\Big(\mathbb{E}\int_{3T/4}^T\int_{\mathcal {O}_1}\lambda^2\mu^2\theta^2\gamma_t |\varphi| \xi z^2\mathrm{d}x\mathrm{d}t +\mathbb{E}\int_0^T\int_{\mathcal {O}_1} \lambda\mu^2\xi \theta^2 | \nabla z |^2  \mathrm{d}x\mathrm{d}t\\
&   +\mathbb{E}\int_{\mathcal {O}'_T} \lambda^3\mu^4 \xi^3 \theta^2  z^2  \mathrm{d}x\mathrm{d}t+C\mathbb{E}\int_{\mathcal {O}_T} \theta^2\left( \lambda^2 \mu^2  \xi ^3\phi_2^2 + | \nabla \phi_2|^2 +  \phi_1 ^2\right)  \mathrm{d}x\mathrm{d}t \Big).
\end{split}
\end{equation}
It remains to establish an appropriate uniform bound for the gradient term  on the R.H.S. of \eqref{2.34}.  To achieve this goal, we shall employ the standard cut-off approach alongside energy estimation method for parabolic SPDEs (cf. \cite[p.315]{lu2021mathematical}). However, special attention must be paid to the terms involving $\gamma_t$, which differ slightly from \cite[(68)-(72)]{hernandez2023global} as our analysis now focuses on the forward SPDEs.

More precisely, since $\mathcal {O}_1\subset\subset \mathcal {O}'$, one can choose a smooth cut-off function $\zeta \in \mathcal {C}_0^\infty(\mathcal {O}';[0,1])$ such that $\zeta \equiv1$ in $\mathcal {O}_1$. By applying the It\^{o} formula, we infer  that
$
\mathrm{d}(\lambda\mu^2\zeta^2\xi \theta^2z^2)=\lambda\mu^2 \zeta^2(\xi \theta^2)_t z^2\mathrm{d}t+2\lambda\mu^2\zeta^2\xi \theta^2 z \mathrm{d}z+ \lambda\mu^2\zeta^2\xi \theta^2(\mathrm{d}z)^2,
$
which together with $\lim_{t\rightarrow 0^+}\theta(t,\cdot)=0$ and the $z$-equation  (i.e., \eqref{2.1} with $\textbf{b}\equiv\textbf{0}$) lead to
\begin{equation}\label{2.35}
\begin{split}
2 \mathbb{E}\int_{\mathcal {O}_T}\sum_{i,j} \lambda\mu^2 \zeta^2\xi \theta^2a^{ij}z_{x_i}z_{x_j} \mathrm{d}x\mathrm{d}t
\leq \mathbb{E}\int_{\mathcal {O}_T}\lambda\mu^2\zeta^2 \xi \phi_2^2 \mathrm{d}x\mathrm{d}t+ \sum_{i=1,2,3}\eqref{2.35}_i,
\end{split}
\end{equation}
where $\eqref{2.35}_1= 2\mathbb{E}\int_{\mathcal {O}_T} \lambda\mu^2\zeta^2\xi \theta^2z\left(\langle\textbf{ a}, \nabla z\rangle+ \alpha z+ \phi_1 \right) \mathrm{d}x\mathrm{d}t$, $\eqref{2.35}_2
		=-2\mathbb{E}\int_{\mathcal {O}_T}\sum_{i,j} $ $\lambda\mu^2 a^{ij}z_{x_i}(\zeta^2\xi \theta^2)_{x_j}z \mathrm{d}x\mathrm{d}t$  and $\eqref{2.35}_3=\mathbb{E}\int_{\mathcal {O}_T}\lambda\mu^2\zeta^2( \xi\theta^2)_tz^2 \mathrm{d}x\mathrm{d}t$.

To completes the proof, it remains to estimate the terms $\eqref{2.35}_1$-$\eqref{2.35}_3$ involved in $\eqref{2.35}$, which can be treated as follows:

\begin{itemize} [itemsep=0pt, leftmargin=28pt]
\item [$\bullet$] By using the Young inequality, we get for any $\epsilon>0$ 
	\begin{equation*}  
		\begin{split}
			\eqref{2.35}_1
			\leq   \epsilon\mathbb{E}\int_{\mathcal {O}_T}  \lambda\mu^2 \zeta^2\xi\theta^2| \nabla z|^2   \mathrm{d}x\mathrm{d}t +C\mathbb{E}\int_{\mathcal {O}_T} \theta^2 \big (\phi_1^2    + \zeta^2\lambda^2\mu^4 \xi^3z^2\big) \mathrm{d}x\mathrm{d}t .
		\end{split}
	\end{equation*}
	
\item [$\bullet$] Since
$
(\zeta^2\xi \theta^2)_{x_j}= 2\zeta \zeta_{x_j}\xi \theta^2+ \mu\zeta^2\xi \theta^2\beta_{x_j}+ 2\lambda\mu\beta_{x_j}\zeta^2\xi^2\theta^2,
$
we have
\begin{equation*} 
	\begin{split}
		\eqref{2.35}_2 
		&\leq  \epsilon \mathbb{E}\int_{\mathcal {O}_T} \lambda\mu^2 \zeta^2\xi \theta^2|\nabla z|^2\mathrm{d}x\mathrm{d}t +C \mathbb{E}\int_{\mathcal {O}_T} \lambda^3\mu^4 \zeta^2\xi^3\theta^2 z^2 \mathrm{d}x\mathrm{d}t.
	\end{split}
\end{equation*}	
	
\item [$\bullet$] 	Because of supp$\zeta \subset \mathcal {O}'$, $ \gamma_t \equiv0$ on $[T/2,3T/4]$ and $(\xi\theta^2)_t= \xi_t \theta^2+  2\frac{\gamma_t }{\gamma}\lambda\varphi\xi \theta^2$,
it follows that
\begin{equation*} 
	\begin{split}
		&\eqref{2.35}_3 =   \mathbb{E}\int_{\mathcal {O}_T}  \lambda\mu^2\zeta^2  \xi_t \theta^2z^2 \mathrm{d}x\mathrm{d}t + 2\mathbb{E}\int_{0}^{T/2}\int_{\mathcal {O}'}\frac{ \gamma_t }{\gamma} \lambda^2\mu^2\zeta^2 \varphi\xi \theta^2z^2 \mathrm{d}x\mathrm{d}t\\
		&+ 2\mathbb{E}\int_{3T/4}^{T}\int_{\mathcal {O}'} \frac{ \gamma_t }{\gamma} \lambda^2\mu^2\zeta^2 \varphi\xi \theta^2z^2 \mathrm{d}x\mathrm{d}t
		=    \eqref{2.35}_{31}+\eqref{2.35}_{32}+\eqref{2.35}_{33}.
	\end{split}
\end{equation*} 

Let us estimate the terms $\eqref{2.35}_{31}$-$\eqref{2.35}_{33}$ as follows: (\textbf{1}) Since $|\xi_t|\leq C \lambda\mu \xi^3$ for all $(t,x)\in \mathcal {O}_T$, supp$ \zeta \subset \mathcal {O}'\Rightarrow \zeta\equiv0 $ on $\mathcal {O}\setminus\overline{\mathcal {O}'}$, and $\zeta$ is bounded on $\overline{\mathcal {O}'}$, we infer that
$
|\eqref{2.35}_{31} |\leq C\mathbb{E}\int_{\mathcal {O}_T}  \lambda^2\mu^3  \xi^3 \theta^2\zeta^2z^2 \mathrm{d}x\mathrm{d}t\leq C\mathbb{E}\int_{\mathcal {O}'_T} \lambda^2\mu^3  \xi^3 \theta^2z^2 \mathrm{d}x\mathrm{d}t
$;
(\textbf{2})  As $| \gamma_t |\leq C \gamma^2$, $|\gamma\varphi|\leq \mu\xi^2$ on $(0,T/2]$, it is easy to derive that
$
|\eqref{2.35}_{32}| \leq C\mathbb{E}\int_{0}^{T/2}\int_{\mathcal {O}'} \lambda^2\mu^3 \xi^3 \theta^2z^2 \mathrm{d}x\mathrm{d}t 
$; (\textbf{3}) 
Because of $\gamma_t (t)\geq 0$,  $\varphi(t)\leq 0$ on $ [3T/4,T]$ and $1\leq \gamma(t)\leq 2$, we have
$
-\eqref{2.35}_{33}\geq C\mathbb{E}\int_{3T/4}^{T}\int_{\mathcal {O}'} \lambda^2\mu^2\zeta^2 \gamma_t |\varphi|\xi \theta^2z^2 \mathrm{d}x\mathrm{d}t.
$ Therefore, by inserting estimates for $\eqref{2.35}_{31}$-$\eqref{2.35}_{33}$ into \eqref{2.35}$_{3}$, we get
\begin{equation*} 
\begin{split}
&\eqref{2.35}_3 \leq   C\mathbb{E}\int_{\mathcal {O}_T}  \lambda^2\mu^3  \xi^3 \theta^2\zeta^2z^2 \mathrm{d}x\mathrm{d}t+C\mathbb{E}\int_{0}^{T/2}\int_{\mathcal {O}'} \lambda^2\mu^3 \xi^3 \theta^2z^2 \mathrm{d}x\mathrm{d}t\\
&-C\mathbb{E} \int_{3T/4}^{T}\int_{\mathcal {O}'} \lambda^2\mu^2\zeta^2 \gamma_t |\varphi|\xi \theta^2z^2 \mathrm{d}x\mathrm{d}t.
\end{split}
\end{equation*}
\end{itemize}

Substituting the estimates for \eqref{2.35}$_{1}$-\eqref{2.35}$_{3}$ into \eqref{2.35}, leveraging the assumption (A$_1$) and the fact of
$\zeta\equiv 1$ on $\mathcal {O}_1$, and choosing sufficiently small
$\epsilon$, we obtain
\begin{equation}\label{2.40}
\begin{split}
&\mathbb{E}\int_{0}^{T}\int_{\mathcal {O}_1}\lambda\mu^2\xi \theta^2 |\nabla z|^2\mathrm{d}x\mathrm{d}t+ \mathbb{E} \int_{3T/4}^{T}\int_{\mathcal {O}_1} \lambda^2\mu^2  \gamma_t  |\varphi|\xi \theta^2 z^2 \mathrm{d}x\mathrm{d}t \\
& \leq  C\Big(\mathbb{E}\int_{\mathcal {O}_T} \theta^2  \phi_1^2   \mathrm{d}x\mathrm{d}t +  \mathbb{E}\int_{\mathcal {O}'_T} \lambda^3\mu^4 \xi^3\theta^2 z^2 \mathrm{d}x\mathrm{d}t \Big).
\end{split}
\end{equation}
Using the property of
$\gamma(t)$ on $[4T/3,T]$, the   Carleman estimate \eqref{carleman1} is derived by combining the estimates \eqref{2.34} and \eqref{2.40}.
The proof of Theorem \ref{thm1} is completed.
\end{proof}

\section{Controllability of backward parabolic SPDEs}\label{sec3}

\subsection{A new Carleman estimate}
To prove Theorem \ref{thm2}, by using Lions's HUM method \cite{Lions1988} and a duality argument similar to \cite{liu2014global}, let us first consider the following forward deterministic system:
\begin{equation}\label{7.1}
\left\{
\begin{aligned}
&z_t-\nabla\cdot (\mathcal {A}\nabla z)+ \langle \textbf{a}, \nabla z\rangle+ \alpha z= \phi_1 ~ \textrm{in}~\mathcal {O}_T,\\
&z=0 ~\textrm{on}~\Sigma_T,\quad
z(0)=z_0 ~~\textrm{in}~\mathcal {O},
\end{aligned}
\right.
\end{equation}
where $\textbf{a}\in L^\infty(\mathcal {O}_T;\mathbb{R}^n)$, $\alpha\in L^\infty(\mathcal {O}_T)$ and $\phi_1 \in L ^2(0,T;L^2(\mathcal {O}))$. As a special case of   Theorem \ref{thm1} (by choosing $\textbf{b}\equiv\textbf{0}$ and $\phi_2\equiv0$), we have the following lemma.
\begin{lemma} \label{thm7.1}
For any $z_0 \in L^2_{\mathcal {F}_0}(\Omega;L^2(\mathcal {O}))$, there exist $\lambda_0,\mu_0>0$ such that the unique solution $z$ to the system \eqref{7.1} satisfies
\begin{equation}\label{carleman7.1}
\begin{split}
&  \int_{\mathcal {O}} e^{2\lambda\varphi(T)} (|\nabla z(T)|^2 +  \lambda^{2} \mu^{3}    z^2(T)) \mathrm{d}x +   \int_{\mathcal {O}_T} \lambda\mu^{2} \xi\theta^2 \big(| \nabla z |^2+\lambda^2\mu^{2} \xi^2 z^2\big) \mathrm{d}x\mathrm{d}t \\
&  \leq C \Big(  \int_{\mathcal {O}_T} \theta^2 \phi_1^2 \mathrm{d}x\mathrm{d}t +  \int_{\mathcal {O}'_T} \lambda^{3}\mu^{4} \xi^{3}\theta^2 z^2 \mathrm{d}x\mathrm{d}t \Big),
\end{split}
\end{equation}
for all $\lambda \geq \lambda_0$ and $\mu \geq \mu_0$.
\end{lemma}

Now we consider the following backward stochastic parabolic equation:
\begin{equation}\label{7.2}
\left\{
\begin{aligned}
&\mathrm{d}r +\nabla\cdot (\mathcal {A}\nabla r)\mathrm{d}t = \left[\lambda^{3}\mu^{4}\xi^{3} \theta^2  z- \nabla\cdot\left(\lambda\mu^{2}\xi\theta^2 \nabla z\right)+ \textbf{1}_{\mathcal {O}'} v  \right] \mathrm{d}t+  R\mathrm{d}W_t ~\textrm{in}~\mathcal {O}_T,\\
&r=0 ~~\textrm{on}~\Sigma_T,\quad
r(T)=r_T ~~\textrm{in}~\mathcal {O},
\end{aligned}
\right.
\end{equation}
where $z$ is the given solution of \eqref{2.1}, $v$ is the control variable and the pair $(r,R)$ denotes the state variable associated with the terminal state $r_T$.  

\begin{lemma}\label{lem7.1}
Let $z$ be the solution of \eqref{2.1} associated with $z_0\in L^2_{\mathcal {F}_0}(\Omega; L^2(\mathcal {O}))$. Then for any $r_T \in L^2_{\mathcal {F}_T}(\Omega;L^2(\mathcal {O}))$, there exists a control $\tilde{v} \in L^2_\mathbb{F}(0,T;L^2(\mathcal {O}')) $ such that the associated solution $(\tilde{r},\tilde{R})$ to \eqref{7.2} verifies
$ \tilde{r}(0)=0$ in $\mathcal {O}$, $\mathbb{P}$-a.s.
Moreover, there exists a positive constant $C$ depending only on $\mathcal {O}$ and $\mathcal {O}'$ such that
\begin{equation}\label{umd}
\begin{split}
&\mathbb{E}\int_{\mathcal {O}_T}\theta ^{-2}\big [\tilde{r}^2 + (\lambda \mu \xi)^{-2}  (|\nabla \tilde{r}|^2+\tilde{R}^2)\big]\mathrm{d}x\mathrm{d}t + \mathbb{E}\int_{\mathcal {O}'_T}\lambda^{-3} \mu^{-4} \xi^{-3}  \theta^{-2} \tilde{v}^2 \mathrm{d}x\mathrm{d}t
\\
&\leq C\Big[\mathbb{E}\int_\mathcal {O}\lambda^{-2} \mu^{-2} \theta^{-2}(T) r^2_T \mathrm{d}x +\mathbb{E}\int_{\mathcal {O}_T}\lambda\mu^{2} \xi\theta^2 \big(\lambda^{2}\mu^{2}\xi^{2} z^2  +    | \nabla z|^2\big) \mathrm{d}x\mathrm{d}t\Big].
\end{split}
\end{equation}
\end{lemma}

\begin{proof} Define 
$
\theta_\epsilon=e^{\lambda \varphi _{\epsilon}}$, where $\varphi_{\epsilon}(x,t)= \gamma_{\epsilon}(t)\big(e^{\mu (\beta(x)+6m)}-\mu e^{6\mu (m+1)}\big)
$
and
\begin{equation}\label{7.3}
\gamma_{\epsilon}(t)=\left\{
\begin{aligned}
&\gamma (t+\epsilon) &&  \textrm{in}~ [0,T/2-\epsilon], \\
&1&&     \textrm{in}~{[T/2-\epsilon,3T/4]},\\
&1+\left(1- 4(T-t)/T\right)^\sigma  &&     \textrm{in}~[3T/4,T].
\end{aligned}
\right.
\end{equation}
From the property of $\gamma$, we see that $\varphi_{\epsilon} $ is non-degenerate at endpoints $t=0,T$, and $\gamma(t)\geq\gamma_{\epsilon}(t)$ for   $t\in [0,T]$.  Consider the control problem $
(\textrm{\textbf{P}}_\epsilon)$: $ \inf_{v\in \mathcal {U}} J_\epsilon(v)$ subject to \eqref{7.2},
where 
$
\mathcal {U} =   \{v\in L^2_\mathbb{F}(0,T;L^2(\mathcal {O}'));~ \mathbb{E} \int_{\mathcal {O}'_T} \lambda^{-3} \mu^{-4} \xi^{-3} \theta ^{-2}|v|^2 \mathrm{d}x\mathrm{d}t<\infty  \}
$, and
\begin{equation*}
\begin{split}
J_\epsilon(v) = \frac{1}{2}\mathbb{E} \int_{\mathcal {O}'_T} \lambda^{-3} \mu^{-4} \xi^{-3} \theta ^{-2}v^2 \mathrm{d}x\mathrm{d}t +\frac{1}{2}\mathbb{E}\int_{\mathcal {O}_T} \theta_\epsilon^{-2} r^2 \mathrm{d}x\mathrm{d}t+\frac{1}{2\epsilon}\mathbb{E}\int_{\mathcal {O}} r^2(0)\mathrm{d}x.
\end{split}
\end{equation*}
It is easy to check that, for any $\epsilon>0$,  the functional $J_\epsilon(v)$ is continuous, strictly convex and coercive (e.g., \cite{lu2021mathematical}). Hence, (\textbf{P}$_\epsilon$) admits a unique optimal control $v_\epsilon \in \mathcal {U}$, and the associated optimal solution to the system \eqref{7.2} is denoted by $(r_\epsilon,R_\epsilon) \in [L^2_\mathbb{F}(\Omega;\mathcal {C}([0,T];L^2(\mathcal {O})))\cap L^2_\mathbb{F}(0,T;H^1(\mathcal {O}))]\times L^2_\mathbb{F}(0,T;L^2(\mathcal {O}))$. By using a duality argument similar to \cite{lions1971optimal,imanuvilov2003carleman}, one can deduce from the Euler-Lagrange equation $J_\epsilon'(r_\epsilon,R_\epsilon)=0$ ($J_\epsilon'$ denotes the Fr\'{e}chet derivative of $J_\epsilon$) that
\begin{equation}\label{7.4}
\begin{split}
v_\epsilon=\lambda^{3} \mu^{4} \xi^{3} \theta ^{2}q_\epsilon  \hspace{2mm} \textrm{in}\hspace{2mm}\mathcal {O}_T,~~\mathbb{P}\textrm{-a.s.},
\end{split}
\end{equation}
where $q_\epsilon$ is the solution to the following linear random equation:
\begin{equation}\label{7.5}
\left\{
\begin{aligned}
&\mathrm{d}q_\epsilon -\nabla\cdot(\mathcal {A}\nabla q_{\epsilon})\mathrm{d}t=\theta_\epsilon^{-2} r_\epsilon \mathrm{d}t ~\textrm{in}~\mathcal {O}_T,\\
&q_\epsilon=0 ~\textrm{on}~\Sigma_T,\quad
q_\epsilon(x,0)= \frac{1}{\epsilon} r_\epsilon(x,0) ~\textrm{in}~\mathcal {O}.
\end{aligned}
\right.
\end{equation}

\textsc{Estimate for $q_\epsilon$}. Let us first apply the It\^{o} formula to the process $q_\epsilon r_\epsilon$ and then 
take the expectation, we deduce from \eqref{7.4} that for any $\epsilon>0$
\begin{equation}\label{7.6}
\begin{split}
&\frac{1}{\epsilon}\mathbb{E}\int_\mathcal {O} r_\epsilon^2(0)\mathrm{d}x +\mathbb{E}\int_{\mathcal {O}_T}\theta_\epsilon^{-2} r_\epsilon^2 \mathrm{d}x\mathrm{d}t+ \mathbb{E}\int_{\mathcal {O}'_T} \lambda^{-3} \mu^{-4} \xi^{-3} \theta^{-2} v_\epsilon^2 \mathrm{d}x\mathrm{d}t
\\
& \leq  \epsilon\Big[\mathbb{E}\int_{\mathcal {O}_T} \lambda\mu^{2} \xi\theta^2(\lambda^{2}\mu^{2}\xi^{2}   q_\epsilon^2  +    | \nabla q_\epsilon |^2) \mathrm{d}x\mathrm{d}t+ \mathbb{E}\int_\mathcal {O} \lambda^2\mu^3 e^{2\lambda\varphi(T)}q^2_\epsilon(T)\mathrm{d}x\Big]\\
&   + C\mathbb{E}\int_{\mathcal {O}_T}\lambda\mu^{2} \xi\theta^2( \lambda^{2}\mu^{2}\xi^{2}  z^2 +    | \nabla z|^2) \mathrm{d}x\mathrm{d}t +C\mathbb{E}\int_{\mathcal {O}} \lambda^{-2}\mu^{-3} e^{-2\lambda \varphi(T)}  r_T^2 \mathrm{d}x,
\end{split}
\end{equation}
where the terms on the L.H.S. used the fact of $\theta ^{2}\leq\theta_\epsilon^2 $. To estimate the terms on the R.H.S. of \eqref{7.6}, we observe that \eqref{7.5} is a random linear parabolic PDE, and so one can apply Lemma \ref{thm7.1} (with $\textbf{a}=\textbf{0}$, $\alpha=0$ and $\phi_1=\theta_\epsilon^{-2} r_\epsilon$) to \eqref{7.5} to obtain
\begin{equation} \label{7.7}
\begin{split}
& \mathbb{E}\int_\mathcal {O} \lambda^2\mu^3 e^{2\lambda\varphi(T)}q^2_\epsilon(T)\mathrm{d}x+ \mathbb{E}\int_{\mathcal {O}_T} \lambda\mu^{2} \xi\theta^2 | \nabla q_\epsilon |^2 \mathrm{d}x\mathrm{d}t+ \mathbb{E}\int_{\mathcal {O}_T}  \lambda^{3}\mu^{4}\xi^{3} \theta^2  q_\epsilon^2 \mathrm{d}x\mathrm{d}t \\
&  \leq C  \mathbb{E} \int_{\mathcal {O}_T} \theta^2\theta_\epsilon^{-4} r_\epsilon^2 \mathrm{d}x\mathrm{d}t + C \mathbb{E}\int_{\mathcal {O}'_T} \lambda^{-3} \mu^{-4} \xi^{-3}  \theta^{-2} v_\epsilon^2 \mathrm{d}x\mathrm{d}t.
\end{split}
\end{equation}
Putting the estimate \eqref{7.7} into \eqref{7.6} and taking $\epsilon>0$ small enough, we obtain
\begin{equation}\label{7.8}
\begin{split}
&\frac{1}{\epsilon}\mathbb{E}\int_\mathcal {O} r_\epsilon^2( 0)\mathrm{d}x +\mathbb{E}\int_{\mathcal {O}_T}\theta_\epsilon^{-2} r_\epsilon^2 \mathrm{d}x\mathrm{d}t+ \mathbb{E}\int_{\mathcal {O}'_T} \lambda^{-3} \mu^{-4} \xi^{-3}  \theta^{-2} v_\epsilon^2 \mathrm{d}x\mathrm{d}t
\\
& \leq C\mathbb{E}\int_{\mathcal {O}_T}\lambda\mu^{2} \xi\theta^2( \lambda^{2}\mu^{2}\xi^{2}  z^2 +    | \nabla z|^2) \mathrm{d}x\mathrm{d}t  +C\mathbb{E}\int_{\mathcal {O}} \lambda^{-2}\mu^{-3} e^{-2\lambda \varphi(T)}  r_T^2 \mathrm{d}x.
\end{split}
\end{equation}

\textsc{Estimate for $R_\epsilon$}. By applying the It\^{o} formula to $\lambda^{-2} \mu^{-2} \xi^{-2}  \theta_\epsilon^{-2} r_\epsilon^2$, it follows from the asumption (A$_1$) and the fact of $\theta^{2}\theta_\epsilon^{-2}\leq 1$ that
\begin{fontsize}{9.5pt}{9.5pt}\begin{equation}\label{7.9}
\begin{split}
& \mathbb{E}\int_\mathcal {O}\lambda^{-2} \mu^{-2} (\xi^{-2} \theta_\epsilon^{-2}  r_\epsilon^2)(0) \mathrm{d}x +\mathbb{E}\int_{\mathcal {O}_T}\lambda^{-2} \mu^{-2} \xi^{-2}\theta_\epsilon^{-2}(  R_\epsilon^2+ 2c_0 |\nabla r_{\epsilon} |^2)\mathrm{d}x\mathrm{d}t\\
&  \leq  \mathbb{E}\int_{\mathcal {O}_T}\big[2\lambda  \mu^{2} \xi  \theta_\epsilon^{-2}\theta^{2}  |r_\epsilon z|-\lambda^{-2} \mu^{-2} (\xi^{-2}  \theta_\epsilon^{-2})_t r_\epsilon^2  \big]  \mathrm{d}x \mathrm{d}t +  \mathbb{E} \int_{\mathcal {O} } \lambda^{-2} \mu^{-2} \theta^{-2}(T)  r_T^ 2 \mathrm{d}x \\
&  + \mathbb{E}\int_{\mathcal {O}_T}2\lambda^{-1}  \xi^{-1}\theta_\epsilon^{-2}\theta^{2} |\nabla r_\epsilon\cdot\nabla z| \mathrm{d}x \mathrm{d}t+ \mathbb{E}\int_{\mathcal {O}_T}2\lambda^{-1} \xi\theta^2 |r_\epsilon  \nabla(\xi^{-2}  \theta_\epsilon^{-2}) \cdot\nabla z| \mathrm{d}x \mathrm{d}t\\
&   + \mathbb{E}\int_{\mathcal {O}'_T}2\lambda^{-2} \mu^{-2} \xi^{-2}  \theta_\epsilon^{-2} |r_\epsilon   v_\epsilon| \mathrm{d}x \mathrm{d}t+\mathbb{E}\int_{\mathcal {O}_T}\sum_{i,j}2\lambda^{-2} \mu^{-2} |r_\epsilon a^{ij}r_{\epsilon,x_i}(\xi^{-2}  \theta_\epsilon^{-2})_{x_j}|\mathrm{d}x \mathrm{d}t .
\end{split}
\end{equation}
\end{fontsize}
Note that for any $t\in [3T/4,T]$, we have $\gamma_\epsilon=\gamma$, $\gamma_t \geq 0$ and $\varphi\leq 0$,  which implies
$
(\xi^{-2}\theta_\epsilon^{-2})_t \geq-\lambda\varphi\frac{\gamma_t}{\gamma}\xi^{-2}\theta ^{-2}\geq0$, for all $  t\in [3T/4,T]$.
The first term on the R.H.S. of \eqref{7.9} can be estimated as
\begin{equation}\label{7.10}
\begin{split}
-\mathbb{E}\int_{\mathcal {O}_T}\lambda^{-2} \mu^{-2} (\xi^{-2}  \theta_\epsilon^{-2})_t r_\epsilon^2 \mathrm{d}x\mathrm{d}t
&\leq C\mathbb{E}\int_0^{3T/4} \int_{\mathcal {O}}\lambda ^{-1}\mu ^{-1}\theta_\epsilon^{-2} r_\epsilon^2 \mathrm{d}x\mathrm{d}t.
\end{split}
\end{equation}
By applying the Young inequality to the other terms on the R.H.S. of \eqref{7.9}, and using the property:
$
|\nabla(\xi^{-2}  \theta_\epsilon^{-2})| \leq |2\mu\nabla\beta\xi^{-2} \theta_\epsilon^{-2}|+ |2\lambda\mu\nabla\beta\xi^{-2}\theta_\epsilon ^{-2}\xi_\epsilon|\leq C \mu \xi^{-2} \theta_\epsilon^{-2} + C\lambda\mu \xi^{-1}\theta_\epsilon ^{-2},
$
we get from \eqref{7.9} and \eqref{7.10} that
\begin{equation}\label{7.11}
\begin{split}
& \mathbb{E}\int_\mathcal {O}\lambda^{-2} \mu^{-2} (\xi^{-2} \theta_\epsilon^{-2}  r_\epsilon^2)(0) \mathrm{d}x+\mathbb{E}\int_{\mathcal {O}_T}\lambda^{-2} \mu^{-2} \xi^{-2}\theta_\epsilon^{-2}(  R_\epsilon^2+ 2c_0 |\nabla r_{\epsilon} |^2)\mathrm{d}x\mathrm{d}t \\
& \leq \delta\mathbb{E}\int_{\mathcal {O}_T} \lambda^{-2} \mu^{-2} \theta_\epsilon ^{-2} \xi^{-2} | \nabla r_{\epsilon}|^2\mathrm{d}x \mathrm{d}t +C\mathbb{E}\int_0^{3T/4} \int_{\mathcal {O}}\lambda ^{-1}\mu ^{-1}\theta_\epsilon^{-2} r_\epsilon^2 \mathrm{d}x\mathrm{d}t \\
&  +C\mathbb{E}\int_{\mathcal {O}_T}\theta_\epsilon^{-2} r_\epsilon^2 \mathrm{d}x\mathrm{d}t+ \mathbb{E}\int_{\mathcal {O}_T}4\lambda^2 \mu^{4} \xi^2 \theta^{2}z^2 \mathrm{d}x \mathrm{d}t + \mathbb{E}\int_{\mathcal {O}_T}\mu^2  \theta^{2}|\nabla z|^2 \mathrm{d}x \mathrm{d}t \\
&   + \mathbb{E}\int_{\mathcal {O}'_T}\lambda^{-4} \mu^{-4} \xi^{-4}  \theta_\epsilon^{-2} v_\epsilon^2 \mathrm{d}x \mathrm{d}t  +\mathbb{E} \int_{\mathcal {O} } \lambda^{-2} \mu^{-2}   \theta ^{-2}(T)  r_T^2  \mathrm{d}x  .
\end{split}
\end{equation}
By taking the parameters $\delta>0$ small enough and $\lambda,\mu>1$  large enough, we get from the estimates \eqref{7.8}, \eqref{7.11} and the fact of $\|\xi^{-1}\|_{L^\infty}\leq1$ that
\begin{equation}\label{7.12}
\begin{split}
& \mathbb{E}\int_\mathcal {O}\lambda^{-2} \mu^{-2} (\xi^{-2} \theta_\epsilon^{-2}  r_\epsilon^2)(0) \mathrm{d}x+\mathbb{E}\int_{\mathcal {O}_T}\lambda^{-2} \mu^{-2} \xi^{-2}\theta_\epsilon^{-2}(   R_\epsilon^2+  |\nabla r_{\epsilon} |^2) \mathrm{d}x\mathrm{d}t \\
&  \leq C\mathbb{E}\int_{\mathcal {O}_T}\lambda\mu^{2} \xi\theta^2( \lambda^{2}\mu^{2}\xi^{2}  z^2 +    | \nabla z|^2) \mathrm{d}x\mathrm{d}t  +\mathbb{E} \int_{\mathcal {O} } \lambda^{-2} \mu^{-2} \theta ^{-2}(T)  r_T^2 \mathrm{d}x  .
\end{split}
\end{equation}

\textsc{Taking the limit as $\epsilon\rightarrow0$}. By \eqref{7.8} and \eqref{7.12}, we get the   uniform bound
\begin{equation}\label{um}
\begin{split}
&\frac{1}{\epsilon}\mathbb{E}\int_\mathcal {O} r_\epsilon^2(0)  \mathrm{d}x+\mathbb{E}\int_{\mathcal {O}_T}\big( \theta_\epsilon^{-2}r_\epsilon^2+\lambda^{-2} \mu^{-2} \xi^{-2} \theta_\epsilon^{-2} |\nabla r_{\epsilon} |^2\big)\mathrm{d}x\mathrm{d}t\\
& +\mathbb{E}\int_{\mathcal {O}_T}\lambda^{-2} \mu^{-2} \xi^{-2}  \theta_\epsilon^{-2} R_\epsilon^2\mathrm{d}x\mathrm{d}t+ \mathbb{E}\int_{\mathcal {O}'_T} \lambda^{-3} \mu^{-4} \xi^{-3}  \theta^{-2} v_\epsilon^2 \mathrm{d}x\mathrm{d}t
\\
& \leq C\mathbb{E}\int_{\mathcal {O}_T}\lambda\mu^{2} \xi\theta^2( \lambda^{2}\mu^{2}\xi^{2}  z^2 +    | \nabla z|^2) \mathrm{d}x\mathrm{d}t  +\mathbb{E} \int_{\mathcal {O} } \lambda^{-2} \mu^{-2} \theta ^{-2}(T)  r_T^2 \mathrm{d}x .
\end{split}
\end{equation}
As a consequence, there exists a subsequence of $(r_\epsilon,R_\epsilon,v_\epsilon)$ (still denoted by itself) and a triple $(\tilde{v},\tilde{r},\tilde{R})$ such that, as $\epsilon\rightarrow0$,
\begin{equation}\label{7.13}
\left\{
\begin{aligned}
&v_\epsilon \rightarrow \tilde{v}  ~\textrm{weakly in}~ L^2_\mathbb{F}(0,T;L^2(\mathcal {O}')),~~ r_\epsilon \rightarrow \tilde{r} ~\textrm{weakly in}~ L^2_\mathbb{F}(0,T;H^1_0(\mathcal {O})),\\
&R_\epsilon \rightarrow \tilde{R}  ~ \textrm{weakly in}~ L^2_\mathbb{F}(0,T;L^2(\mathcal {O})).
\end{aligned}
\right.
\end{equation}

We \textit{claim} that $(\tilde{r},\tilde{R}) $ is the unique solution to \eqref{7.2}. Indeed, we denote by $(\hat{r},\hat{R})\in L^2_\mathbb{F}(\Omega; \mathcal {C}([0,T];L^2(\mathcal {O})))\times L^2_\mathbb{F}(0,T;L^2(\mathcal {O}))$ the unique solution to \eqref{7.2} associated with  the control $\tilde{v}$.  For any $h_1,h_2 \in L^2_\mathbb{F}(0,T;L^2(\mathcal {O}))$ and $\vartheta_0\in L^2_\mathbb{F}(\Omega;L^2(\mathcal {O}))$, we consider the forward system
\begin{equation} \label{7.14}
\left\{
\begin{aligned}
&\mathrm{d} \vartheta-\nabla\cdot(\mathcal {A}\nabla \vartheta)  \mathrm{d}t= h_1 \mathrm{d}t+ h_2\mathrm{d}W_t ~\textrm{in}~\mathcal {O}_T,\\
&\vartheta=0 ~\textrm{on}~\Sigma_T,\quad
\vartheta(0)=0~\textrm{in}~\mathcal {O}.
\end{aligned}
\right.
\end{equation}
By applying the It\^{o} formula to the processes $r_\epsilon\vartheta$ and $\tilde{r}\vartheta$, respectively, we get from \eqref{7.2} and \eqref{7.14} that
$\mathbb{E}\int_{\mathcal {O}_T}[r_\epsilon h_1  + \vartheta (\lambda^{3}\mu^{4}\xi^{3} \theta^2  z + \textbf{1}_{\mathcal {O}'} v_\epsilon) +\lambda\mu^{2}  \xi\theta^2\nabla\vartheta \cdot \nabla z  + R_\epsilon h_2] \mathrm{d}x\mathrm{d}t=0$,
and
$\mathbb{E}\int_{\mathcal {O}_T}[\hat{r}h_1 +\vartheta (\lambda^{3}\mu^{4}\xi^{3} \theta^2  z + \textbf{1}_{\mathcal {O}'} \tilde{v} ) +\lambda\mu^{2}  \xi\theta^2\nabla\vartheta \cdot \nabla z  + \hat{R} h_2] \mathrm{d}x\mathrm{d}t=0$. By taking the limit as $\epsilon\rightarrow 0$, we get from \eqref{7.13} and the last two identities that
$
\mathbb{E}\int_{\mathcal {O}_T}[(\tilde{r}-\hat{r})h_1 +(\tilde{R}-\hat{R})h_2] \mathrm{d}x\mathrm{d}t=0.
$
Since $h_1,h_2 \in L^2_\mathbb{F}(0,T;L^2(\mathcal {O}))$ are arbitrary, we obtain that $\tilde{r}=\hat{r}$ and $\tilde{R}=\hat{R}$ in $\mathcal {O}_T$, $\mathbb{P}$-a.s. Finally, by taking the limit $\epsilon\rightarrow 0$, one can conclude from \eqref{um}, \eqref{7.13} and the Fatou Lemma that
$\tilde{r}(0)=0$ in $\mathcal {O}$, $\mathbb{P}$-a.s. The proof of Lemma \ref{lem7.1} is completed.
\end{proof}

\begin{proof}[Proof of Theorem \ref{thm2}]
In the following proof, we consider the special case of $r_T\equiv0$ in \eqref{7.2}. Let $(\tilde{r},\tilde{R})$ be the solution of \eqref{7.2} corresponding to the control $\tilde{v}$ obtained in Lemma \ref{lem7.1}, such that $\tilde{r}( 0)=0$ in $\mathcal {O}$, $\mathbb{P}$-a.s. By \eqref{2.1} and \eqref{7.2}, we have
\begin{equation*}
\begin{split}
0 & = \mathbb{E}\int_{\mathcal {O}_T}\tilde{r}\Big(\nabla\cdot (\mathcal {A}\nabla z) \mathrm{d}x\mathrm{d}t+ \left(\langle \textbf{a}, \nabla z\rangle+ \alpha z+ \phi_1+ \nabla\cdot \textbf{b} \right)\mathrm{d}x\mathrm{d}t \Big)+ \mathbb{E}\int_{\mathcal {O}_T}\tilde{R}\phi_2\mathrm{d}x\mathrm{d}t\\
&+ \mathbb{E}\int_{\mathcal {O}_T}z\Big(-\nabla\cdot (\mathcal {A}\nabla \tilde{r})\mathrm{d}x\mathrm{d}t+ \left[\lambda^{3}\mu^{4}\xi^{3} \theta^2  z-\lambda\mu^{2} \nabla\cdot\left(\xi\theta^2 \nabla z\right)+ \textbf{1}_{\mathcal {O}'} \tilde{v}   \right] \mathrm{d}x\mathrm{d}t \Big).
\end{split}
\end{equation*}
After integrating by parts and taking the expectation, the last identity implies
\begin{equation}\label{7.17}
\begin{split}
&\mathbb{E}\int_{\mathcal {O}_T} \lambda\mu^{2} \xi\theta^2(\lambda^{2}\mu^{2}\xi^{2} z^2  + |\nabla z|^2) \mathrm{d}x\mathrm{d}t\leq \epsilon  \mathbb{E}\int_{\mathcal {O}_T} \lambda\mu^{2} \xi\theta^2 |\nabla z|^2\mathrm{d}x\mathrm{d}t\\
&   +C\|\textbf{a}\|_{L^\infty_\mathbb{F}(0,T;L^\infty(\mathcal {O};\mathbb{R}^n))}^2
\mathbb{E}\int_{\mathcal {O}_T} \lambda^{-1}\mu^{-2} \xi^{-1}\theta^{-2}\tilde{r}^2 \mathrm{d}x\mathrm{d}t\\
&  + \epsilon\mathbb{E}\int_{\mathcal {O}_T} \lambda^{3}\mu^{4}\xi^{3} \theta^2  z^2\mathrm{d}x\mathrm{d}t+C\|\alpha\|_{L^\infty_\mathbb{F}(0,T;L^\infty(\mathcal {O}))}^2\mathbb{E}\int_{\mathcal {O}_T} \lambda^{-3}\mu^{-4}\xi^{-3} \theta^{-2}  \tilde{r}^2\mathrm{d}x\mathrm{d}t\\
&  + \epsilon\mathbb{E}\int_{\mathcal {O}_T} \theta^{-2}  \tilde{r}^2\mathrm{d}x\mathrm{d}t+ C\mathbb{E}\int_{\mathcal {O}_T} \theta^{2}  \phi_1^2\mathrm{d}x\mathrm{d}t+ C\mathbb{E}\int_{\mathcal {O}_T} \lambda^{2}\mu^{2}\xi^{2} \theta^{2}  (|\textbf{b}|^2+ \phi_2^2)\mathrm{d}x\mathrm{d}t\\
&  + \epsilon\mathbb{E}\int_{\mathcal {O}_T}\lambda^{-2} \mu^{-2} \xi^{-2}  \theta ^{-2} (|\nabla \tilde{r}|^2+\tilde{R}^2)\mathrm{d}x\mathrm{d}t\\
&  +\epsilon\mathbb{E}\int_{\mathcal {O}'_T} \lambda^{-3} \mu^{-4} \xi^{-3}  \theta^{-2} \tilde{v}^2 \mathrm{d}x\mathrm{d}t +C\mathbb{E}\int_{\mathcal {O}'_T}\lambda^{3}\mu^{4}\xi^{3} \theta^2  z^2 \mathrm{d}x\mathrm{d}t.
\end{split}
\end{equation}
First, by applying the Carleman estimate \eqref{umd} with $r_T=0$, the fifth, eighth and ninth terms on the R.H.S. of \eqref{7.17} can be bounded by $\epsilon (\mathbb{E}\int_{\mathcal {O}_T} \lambda^{3}\mu^{4}\xi^{3} \theta^2  z^2 \mathrm{d}x\mathrm{d}t+\mathbb{E}\int_{\mathcal {O}_T} \lambda\mu^{2} \xi\theta^2 |\nabla z|^2 \mathrm{d}x\mathrm{d}t)$. Then, by choosing small $\epsilon >0$ in \eqref{7.17}, we get
\begin{equation}\label{7.18}
\begin{split}
&\mathbb{E}\int_{\mathcal {O}_T}\lambda\mu^{2} \xi\theta^2( \lambda^{2}\mu^{2}\xi^{2}  z^2 +    | \nabla z|^2) \mathrm{d}x\mathrm{d}t\\
&  \leq C \lambda^{-1}\mu^{-2}
\mathbb{E}\int_{\mathcal {O}_T} ( \xi^{-1}\theta^{-2}  +   \xi^{-3} \theta^{-2})  \tilde{r}^2\mathrm{d}x\mathrm{d}t+C\mathbb{E}\int_{\mathcal {O}_T} \theta^{2}  \phi_1^2\mathrm{d}x\mathrm{d}t\\
& + C\mathbb{E}\int_{\mathcal {O}_T} \lambda^{2}\mu^{2}\xi^{2} \theta^{2}  (|\textbf{b}|^2+ \phi_2^2)\mathrm{d}x\mathrm{d}t +C\mathbb{E}\int_{\mathcal {O}'_T} \lambda^{3}\mu^{4}\xi^{3} \theta^2  z^2 \mathrm{d}x\mathrm{d}t.
\end{split}
\end{equation}
Since $\|\xi^{-1}\|_{L^\infty}\leq 1$ and $r_T=0$, one sees that, by   \eqref{umd}, the first term on the R.H.S. of \eqref{7.18} can be absorbed by taking $\lambda,\mu>1$ large enough, namely, we can obtain
\begin{equation}\label{wooo}
\begin{split}
&\mathbb{E}\int_{\mathcal {O}_T}\lambda\mu^{2} \xi\theta^2( \lambda^{2}\mu^{2}\xi^{2}  z^2 +    | \nabla z|^2) \mathrm{d}x\mathrm{d}t\leq C\mathbb{E}\int_{\mathcal {O}'_T} \lambda^{3}\mu^{4}\xi^{3} \theta^2  z^2 \mathrm{d}x\mathrm{d}t \\
&+C\mathbb{E}\int_{\mathcal {O}_T}  \theta^{2}  \phi_1^2\mathrm{d}x\mathrm{d}t + C\mathbb{E}\int_{\mathcal {O}_T} \lambda^{2}\mu^{2}\xi^{2} \theta^{2} ( |\textbf{b}|^2 +\phi_2^2)\mathrm{d}x\mathrm{d}t.
\end{split}
\end{equation}
Applying the It\^{o} formula to $\lambda\mu^2\xi\theta^2 z^2$ and integrating by parts over $\mathcal {O}_T$, we obtain
\begin{fontsize}{9pt}{9pt}\begin{equation}\label{wo}
\begin{split}
&\mathbb{E}\int_\mathcal {O}\lambda\mu^2(\xi\theta^2)(T)z^2(T)\mathrm{d}x+\mathbb{E}\int_{\mathcal {O}_T}2\lambda\mu^2\xi\theta^2 (\mathcal {A}\nabla z)\cdot\nabla z \mathrm{d}x\mathrm{d}t 
\\
&  =  \mathbb{E}\int_{\mathcal {O}_T}2\lambda\mu^2\xi\theta^2 z   \langle \textbf{a}, \nabla z\rangle \mathrm{d}x\mathrm{d}t+\mathbb{E}\int_{\mathcal {O}_T}2\lambda\mu^2\alpha\xi\theta^2 z^2 \mathrm{d}x\mathrm{d}t+\mathbb{E}\int_{\mathcal {O}_T}2\lambda\mu^2\xi\theta^2 z \phi_1 \mathrm{d}x\mathrm{d}t\\
&  +  \mathbb{E}\int_{\mathcal {O}_T}2\lambda\mu^2\nabla(\xi\theta^2) \cdot  \textbf{b}z \mathrm{d}x\mathrm{d}t+  \mathbb{E}\int_{\mathcal {O}_T}2\lambda\mu^2\xi\theta^2 \textbf{b}\cdot\nabla z \mathrm{d}x\mathrm{d}t+ \mathbb{E}\int_{\mathcal {O}_T}\lambda\mu^2\xi\theta^2 \phi_2^2\mathrm{d}x\mathrm{d}t\\
&   +\mathbb{E}\int_{\mathcal {O}_T}\lambda\mu^2(\xi\theta^2)_t z^2 \mathrm{d}x\mathrm{d}t- \mathbb{E}\int_{\mathcal {O}_T}2\lambda \mu^2 z\nabla(\xi\theta^2) \cdot\nabla z \mathrm{d}x\mathrm{d}t   = J_1+\cdots+J_8.
\end{split}
\end{equation}
\end{fontsize}
By using the Young inequality, the terms $J_1$-$J_5$ can be estimated as
\begin{equation*}
\begin{split}
&J_1+J_2
 \leq  \delta\mathbb{E}\int_{\mathcal {O}_T} \lambda\mu^{2} \xi\theta^2 |\nabla z|^2 \mathrm{d}x\mathrm{d}t\\
 &+ C \big(\|\textbf{a}\|_{L^\infty_\mathbb{F}(0,T;L^\infty(\mathcal {O};\mathbb{R}^n))}^2+\|\alpha\|_{L^\infty_\mathbb{F}(0,T;L^\infty(\mathcal {O} ))}^2\big)\mathbb{E}\int_{\mathcal {O}_T} \lambda\mu^{2} \xi\theta^2 z^2 \mathrm{d}x\mathrm{d}t,\\
&J_3\leq 4 \mathbb{E}\int_{\mathcal {O}_T}\lambda^2\mu^4\xi^2\theta^2 z^2 \mathrm{d}x\mathrm{d}t+\mathbb{E}\int_{\mathcal {O}_T} \theta^2  \phi_1^2 \mathrm{d}x\mathrm{d}t,\\
&J_4\leq C\Big( \mathbb{E}\int_{\mathcal {O}_T} \lambda^2\mu^4 \xi^2\theta^2  z^2 \mathrm{d}x\mathrm{d}t+ \mathbb{E}\int_{\mathcal {O}_T} \lambda^{2}\mu^{2}\xi^{2} \theta^{2}  |\textbf{b}|^2\mathrm{d}x\mathrm{d}t\Big),\\
&J_5\leq \delta \mathbb{E}\int_{\mathcal {O}_T} \lambda\mu^{2} \xi\theta^2 |\nabla z|^2 \mathrm{d}x\mathrm{d}t+ C\mathbb{E}\int_{\mathcal {O}_T} \lambda\mu^2\xi\theta^2 |\textbf{b}|^2\mathrm{d}x\mathrm{d}t,
\end{split}
\end{equation*}
where the estimate for $J_4$ used the fact of $|\nabla(\xi\theta^2)|=|\mu\xi\theta^2 \nabla \beta+2\lambda\mu\xi^2\theta^2\nabla \beta|\leq C \lambda\mu\xi^2\theta^2 $. Owing to this last property, we also have
\begin{equation*}
\begin{split}
J_8\leq \delta \mathbb{E}\int_{\mathcal {O}_T}\lambda \mu^2\xi \theta^2 |\nabla z|^2 \mathrm{d}x\mathrm{d}t+C\mathbb{E}\int_{\mathcal {O}_T} \lambda^3\mu^4 \xi^3\theta^2 z^2 \mathrm{d}x\mathrm{d}t.
\end{split}
\end{equation*}
To estimate $J_7$, we note that
$
(\xi\theta^2)_t= \frac{\gamma_t}{\gamma} \xi\theta^2+ 2\frac{\gamma_t}{\gamma}\lambda\xi\theta^2 \varphi$, for $ t \in (0,T/2]\cup [3T/4,T].
$
On the one hand, since $\gamma_t>0$ on $[3T/4,T]$ and $\varphi<0$, we have $(\xi\theta^2)_t\leq \frac{\gamma_t}{\gamma} \xi\theta^2 \leq C \xi^2\theta^2$; On the other hand, as $|\gamma_t|\leq C \gamma^2$  on $(0,T/2]$, we get $|(\xi\theta^2)_t|\leq C \lambda\mu\xi^3\theta^2$. In both of cases, we deduce that
$
J_7\leq C \mathbb{E}\int_{\mathcal {O}_T}\lambda^2\mu^3 \xi^3\theta^2 z^2 \mathrm{d}x\mathrm{d}t.
$

Putting the above estimates for $J_1$-$J_8$ together and taking $\delta>0$ small enough, we get from \eqref{wo}, the fact of $\|\xi^{-1}\|_{L^\infty}<\infty$ and the assumption (A$_1$) that
\begin{equation*} 
\begin{split}
&\mathbb{E}\int_\mathcal {O}\lambda\mu^2(\xi\theta^2)(T)z^2(T)\mathrm{d}x+ \mathbb{E}\int_{\mathcal {O}_T}\lambda\mu^2\xi\theta^2 |\nabla z|^2\mathrm{d}x\mathrm{d}t 
  \\
  &\leq C\Big[\mathbb{E}\int_{\mathcal {O}_T} \lambda^3\mu^4 \xi^3\theta^2 z^2 \mathrm{d}x\mathrm{d}t+ \mathbb{E}\int_{\mathcal {O}_T} \theta^{2}\big[  \phi_1^2  + \lambda^{2}\mu^{2}\xi^{2}    (\phi_2^2  +    |\textbf{b}|^2) \big] \mathrm{d}x\mathrm{d}t \Big].
\end{split}
\end{equation*} 
Multiplying both sides of \eqref{wooo} by $2C$ and then adding to the last inequality, we get the desired Carleman estimates \eqref{carleman2}. The proof of Theorem \ref{thm2} is completed.
\end{proof}

\subsection{The linear controlled system}

Based on  the Carleman estimate obtained in Theorem \ref{thm2}, one can now prove the null controllability for  \eqref{4.1}.
\begin{proof}[Proof of Theorem \ref{lem4.1}] Consider the weighted function $\theta_\epsilon$ defined by \eqref{7.3}. Apparently, $\theta \theta_\epsilon^{-1}\leq1$ for all $(x,t)\in \mathcal {O}_T$, $\gamma(t)\geq\gamma_{\epsilon}(t)$ for any $t\in (0,T]$, and $\theta_\epsilon(T)\neq0$.
Let us define the cost functional $J_\epsilon(\cdot): L^2_\mathbb{F}(0,T;L^2(\mathcal {O}')) \mapsto \mathbb{R}_+$ by
\begin{equation}
\begin{split}
&J_\epsilon(u)= \frac{1}{2\epsilon}\mathbb{E}\int_{\mathcal {O}} y^2(0)\mathrm{d}x+\frac{1}{2}\mathbb{E}\int_{\mathcal {O}_T} \theta_\epsilon^{-2} y^2 \mathrm{d}x\mathrm{d}t\\
&+\frac{1}{2}\mathbb{E}\int_{\mathcal {O}_T} \lambda ^{-2}\mu ^{-2}\xi ^{-3} \theta_\epsilon^{-2} |\nabla y|^2 \mathrm{d}x\mathrm{d}t+\frac{1}{2}\mathbb{E}\int_{\mathcal {O}'_T} \lambda^{-3} \mu^{-4} \xi^{-3} \theta ^{-2}u^2 \mathrm{d}x\mathrm{d}t.
\end{split}
\end{equation}
Then we introduce the  extremal problem:
$\min_{u\in \mathcal {H}} J_\epsilon(u)$ subject to \eqref{4.1}, where
$
\mathcal {H}=\{u\in L^2_\mathbb{F}(0,T;L^2(\mathcal {O}'));\mathbb{E}\int_{\mathcal {O}'_T} \lambda^{-3} \mu^{-4} \xi^{-3} \theta ^{-2}u^2 \mathrm{d}x\mathrm{d}t<\infty\}.
$
It can be readily seen that the functional $J_\epsilon(u)$ is convex, continuity and coercive over $\mathcal {H}$ (cf. \cite{luzhang2021}), which implies that the above optimal control problem admits a unique optimal control $\hat{u}_\epsilon$. The corresponding solution to the controlled system \eqref{4.1} is denoted by $(\hat{y}_\epsilon,\hat{Y}_\epsilon)$. 

By using the duality argument and Euler-Lagrange principle, the control $\hat{u}_\epsilon$ can be 
characterized by
$
\hat{u}_\epsilon= \lambda^{3} \mu^{4} \xi^{3} \theta ^{2}v_\epsilon \textbf{1}_{\mathcal {O}'} ,
$
where $ v_\epsilon $ solves the equation 
\begin{equation}\label{4.7}
\left\{
\begin{aligned}
&\mathrm{d} v_\epsilon- \nabla\cdot(\mathcal {A}\nabla v_{\epsilon}) \mathrm{d}t= [-\alpha v_\epsilon+\theta_{\epsilon}^{-2}\hat{y}_\epsilon+\nabla\cdot(\textbf{a}v_\epsilon - \lambda^{-2} \mu^{-2} \xi^{-3} \theta_{\epsilon}^{-2}\nabla \hat{y}_\epsilon )]\mathrm{d}t ~\textrm{in}~\mathcal {O}_T,\\
&v_\epsilon=0~\textrm{on}~\Sigma_T,\quad
v_\epsilon(0) = \frac{1}{\epsilon}\hat{y}_\epsilon(0) ~\textrm{in}~\mathcal {O}.
\end{aligned}
\right.
\end{equation}
Here $\hat{y}_\epsilon \in \mathcal {W}_T$ denotes the unique solution to \eqref{4.1} associated with the control   $\hat{u}_\epsilon$. Observing that \eqref{4.7} can be viewed as a special case of \eqref{2.1}. Applying   It\^{o}'s formula to $\hat{y}_\epsilon v_\epsilon$ and integrating by parts over $\mathcal {O}_T$, we get from \eqref{4.1} and \eqref{4.7} that
\begin{equation}\label{4.8}
\begin{split}
&\mathbb{E}\int_{\mathcal {O}_T}  \theta_{\epsilon}^{-2}(\hat{y}_\epsilon^2   +  \lambda^{-2} \mu^{-2} \xi^{-3}  |\nabla\hat{y}_\epsilon|^2) \mathrm{d}x\mathrm{d}t+\mathbb{E}\int_{\mathcal {O}'_T} \lambda^{-3} \mu^{-4} \xi^{-3} \theta ^{-2}\hat{u}_\epsilon^2 \mathrm{d}x\mathrm{d}t\\
&   + \frac{1}{\epsilon}\mathbb{E}\int_{\mathcal {O}} \hat{y}_\epsilon^2 (0) \mathrm{d}x= \mathbb{E}\int_{\mathcal {O}}y_Tv_\epsilon(T)\mathrm{d}x-\mathbb{E}\int_{\mathcal {O}_T} \phi v_\epsilon \mathrm{d}x\mathrm{d}t+ \mathbb{E}\int_{\mathcal {O}_T} \textbf{b} \cdot\nabla v_\epsilon \mathrm{d}x\mathrm{d}t.
\end{split}
\end{equation}
By using the Young inequality, we have for any $\delta>0$
\begin{equation} \label{4.9}
\begin{split}
&\textrm{R.H.S. of \eqref{4.8}} \leq \delta \Big(\mathbb{E}\int_\mathcal {O} \lambda \mu^2 (\xi \theta ^2)(T) v_\epsilon^2(T) \mathrm{d}x+  \mathbb{E}\int_{\mathcal {O}_T} \lambda \mu^2 \xi \theta ^2 |\nabla v_\epsilon|^2\mathrm{d}x\mathrm{d}t\Big)\\
&+ C\mathbb{E}\int_{\mathcal {O}_T} \lambda^{-1} \mu^{-2} \xi^{-1} \theta ^{-2}|\textbf{b}|^2\mathrm{d}x\mathrm{d}t
+ C\mathbb{E}\int_{\mathcal {O}_T} \lambda^{-3} \mu^{-4} \xi^{-3} \theta ^{-2}\phi^2\mathrm{d}x \mathrm{d}t  \\&+ C\mathbb{E}\int_\mathcal {O} \lambda^{-1} \mu^{-2} (\xi^{-1} \theta ^{-2})(T)  y_T^2 \mathrm{d}x .
\end{split}
\end{equation}
Let us  apply the   Carleman estimate in \eqref{carleman2} to \eqref{4.7} (with $\phi_1=\theta_{\epsilon}^{-2} \hat{y}_\epsilon $, $\textbf{b}=\textbf{a} v_\epsilon- \lambda^{-2} \mu^{-2} \xi^{-3} \theta_{\epsilon}^{-2}\nabla\hat{y}_\epsilon$ and $\phi_2=0$),  we find that for any $\lambda,\mu>1$ large enough 
\begin{equation} \label{4.10}
\begin{split}
&\mathbb{E}\int_\mathcal {O}\lambda\mu^2(\xi \theta^2)(T)v_\epsilon^2(T)dx+\mathbb{E}\int_{\mathcal {O}_T} \lambda \mu^2 \xi \theta^2  (\lambda^2 \mu^2 \xi^2  v_\epsilon^2 +|\nabla v_\epsilon|^2)\mathrm{d}x\mathrm{d}t\\
&\leq   C \mathbb{E} \int_{\mathcal {O}'_T}\lambda ^{-3}\mu ^{-4}\theta^{-2} \xi ^{-3}\hat{u}_\epsilon^2 \mathrm{d}x \mathrm{d}t + C\mathbb{E}\int_{\mathcal {O}_T}  \theta_{\epsilon}^{-2}(\hat{y}_\epsilon^2   +  \lambda^{-2} \mu^{-2} \xi^{-3}  |\nabla\hat{y}_\epsilon|^2) \mathrm{d}x\mathrm{d}t ,
\end{split}
\end{equation}
where the R.H.S. of \eqref{4.10} used the fact of $\theta\theta_{\epsilon}^{-1}\leq 1$ and $\|\xi^{-1}\|_{L^\infty}<1$.

Putting \eqref{4.10} into \eqref{4.9} and choosing $\delta >0$ small enough, we obtain
\begin{equation}\label{4.11}
\begin{split}
&\mathbb{E}\int_{\mathcal {O}_T}  \theta_{\epsilon}^{-2}(\hat{y}_\epsilon^2   +  \lambda^{-2} \mu^{-2} \xi^{-3}  |\nabla\hat{y}_\epsilon|^2) \mathrm{d}x\mathrm{d}t+\mathbb{E} \int_{\mathcal {O}'_T} \lambda^{-3} \mu^{-4} \xi^{-3} \theta ^{-2}\hat{u}_\epsilon^2 \mathrm{d}x\mathrm{d}t \\
&+ \frac{1}{\epsilon}\mathbb{E}\int_{\mathcal {O}} \hat{y}_\epsilon^2 (0) \mathrm{d}x \leq C\mathbb{E}\int_\mathcal {O} \lambda^{-1} \mu^{-2} (\xi^{-1} \theta ^{-2})(T)  y_T^2 \mathrm{d}x\\
&+ C\mathbb{E}\int_{\mathcal {O}_T} \lambda^{-3} \mu^{-4} \xi^{-3} \theta ^{-2}\phi^2\mathrm{d}x \mathrm{d}t + C\mathbb{E}\int_{\mathcal {O}_T} \lambda^{-1} \mu^{-2} \xi^{-1} \theta ^{-2}|\textbf{b}|^2\mathrm{d}x\mathrm{d}t,
\end{split}
\end{equation}
for any   $\lambda,\mu>1$ large enough.

To obtain a suitable estimate for the component $\hat{Y}_\epsilon$, let us apply the It\^{o} formula to $\lambda^{-2} \mu^{-2} \xi^{-2} \theta ^{-2}\hat{y}_\epsilon^2$ and integrate by parts over $\mathcal {O}_T$. This leads us to infer that
\begin{equation} \label{jj3}
\begin{split}
&\mathbb{E}\int_{\mathcal {O}_T} \lambda^{-2} \mu^{-2} \xi^{-2} \theta_\epsilon^{-2}\hat{Y}_\epsilon^2  \mathrm{d}x\mathrm{d}t+  2c_0\mathbb{E}\int_{\mathcal {O}_T} \lambda^{-2} \mu^{-2} \xi^{-2} \theta_\epsilon^{-2}|\nabla\hat{y}_\epsilon|^2  \mathrm{d}x\mathrm{d}t\\
&  \leq -\mathbb{E}\int_{\mathcal {O}_T} \lambda^{-2} \mu^{-2} \big[(\xi^{-2} \theta_\epsilon^{-2})_t\hat{y}_\epsilon^2  +2 \hat{y}_\epsilon \nabla( \xi^{-2}\theta_\epsilon^{-2})\cdot (\mathcal {A}\nabla\hat{y}_\epsilon)  \big]  \mathrm{d}x\mathrm{d}t\\
&   +\mathbb{E}\int_{\mathcal {O}_T} \lambda^{-2} \mu^{-2} \xi^{-2} \theta_\epsilon^{-2}\big[\hat{y}_\epsilon^2  -\hat{y}_\epsilon (\langle \textbf{a}, \nabla \hat{y}_\epsilon\rangle+ \phi  + \nabla\cdot \textbf{b} )\big] \mathrm{d}x\mathrm{d}t\\
& -\mathbb{E} \int_{\mathcal {O}'_T} \lambda^{-2} \mu^{-2} \xi^{-2} \theta_\epsilon^{-2}\hat{y}_\epsilon \hat{u}_\epsilon \mathrm{d}x\mathrm{d}t   +\mathbb{E}\int_{\mathcal {O}_T} \lambda^{-2} \mu^{-2}  \theta_\epsilon^{-2}(T)y_T^2  \mathrm{d}x\mathrm{d}t.
\end{split}
\end{equation}

We now proceed to estimate the integral terms on the R.H.S. of \eqref{jj3}: (\textbf{1})  By adopting an argument analogous to the one   presented in \eqref{7.9}, we   deduce that
$-\mathbb{E}\int_{\mathcal {O}_T}\lambda^{-2} \mu^{-2} (\xi^{-2}  \theta_\epsilon^{-2})_t \hat{y}_\epsilon^2 \mathrm{d}x\mathrm{d}t \leq C\lambda ^{-1}\mu ^{-1}\mathbb{E}\int_0^{3T/4} \int_{\mathcal {O}}\theta_\epsilon^{-2} \hat{y}_\epsilon^2 \mathrm{d}x\mathrm{d}t$.
(\textbf{2}) From the definition of weighted functions, we see that $ |\nabla(\xi^{-2}  \theta_\epsilon^{-2})| \leq C  \lambda\mu \xi^{-1}\theta_\epsilon ^{-2}$, which implies 
$ -\mathbb{E}\int_{\mathcal {O}_T} 2\lambda^{-2} \mu^{-2}\hat{y}_\epsilon \nabla( \xi^{-2}\theta_\epsilon^{-2})\cdot (\mathcal {A}\nabla\hat{y}_\epsilon)  \mathrm{d}x\mathrm{d}t
  \leq \frac{c_0}{8}\mathbb{E}\int_{\mathcal {O}_T} \lambda^{-2} \mu^{-2} \xi^{-2} \theta_\epsilon^{-2}|\nabla\hat{y}_\epsilon|^2  \mathrm{d}x\mathrm{d}t + C\mathbb{E}\int_{\mathcal {O}_T} \theta_\epsilon^{-2}  \hat{y}_\epsilon ^2  \mathrm{d}x\mathrm{d}t$.
(\textbf{3}) By using the Young inequality and boundedness of $\textbf{a}$, we get
$-\mathbb{E}\int_{\mathcal {O}_T} \lambda^{-2} \mu^{-2} \xi^{-2} \theta_\epsilon^{-2}\hat{y}_\epsilon(\langle \textbf{a}, \nabla \hat{y}_\epsilon\rangle+ \phi) \mathrm{d}x\mathrm{d}t\leq \frac{c_0}{8}\mathbb{E}\int_{\mathcal {O}_T} \lambda^{-2} \mu^{-2} \xi^{-2} \theta_\epsilon^{-2}|\nabla\hat{y}_\epsilon|^2  \mathrm{d}x\mathrm{d}t +C \mathbb{E}\int_{\mathcal {O}_T} \lambda^{-3} \mu^{-4} \xi^{-3} \theta_\epsilon ^{-2}\phi^2\mathrm{d}x \mathrm{d}t+ C(\lambda^{-1}+\lambda^{-2} \mu^{-2})\mathbb{E}\int_{\mathcal {O}_T} \theta_\epsilon^{-2} \hat{y}_\epsilon^2  \mathrm{d}x\mathrm{d}t$.
(\textbf{4}) Integrating by parts and using the fact of $ \theta_\epsilon^{-1}\leq \theta^{-1}$, there holds
$-\mathbb{E}\int_{\mathcal {O}_T} \lambda^{-2} \mu^{-2} \xi^{-2} \theta_\epsilon^{-2}\hat{y}_\epsilon \nabla\cdot \textbf{b} \mathrm{d}x\mathrm{d}t  \leq \frac{c_0}{8}\mathbb{E}\int_{\mathcal {O}_T} \lambda^{-2} \mu^{-2} \xi^{-2} \theta_\epsilon^{-2}|\nabla\hat{y}_\epsilon|^2  \mathrm{d}x\mathrm{d}t + C \mathbb{E}\int_{\mathcal {O}_T} \lambda^{-1} \theta ^{-2}(\mu^{-2} \xi^{-1} |\textbf{b}|^2  + \hat{y}_\epsilon^2 ) \mathrm{d}x\mathrm{d}t$.
(\textbf{5}) Using the fact of $\|\xi^{-1}\|_{L^\infty}<\infty$, one can derive that
$-\mathbb{E} \int_{\mathcal {O}'_T} \lambda^{-2} \mu^{-2} \xi^{-2} \theta_\epsilon^{-2}\hat{y}_\epsilon \hat{u}_\epsilon \mathrm{d}x\mathrm{d}t \leq  \mathbb{E} \int_{\mathcal {O}_T}  \theta_\epsilon^{-2}\hat{y}_\epsilon^2 \mathrm{d}x\mathrm{d}t+ \mathbb{E}\int_{\mathcal {O}'_T} \lambda^{-4} \mu^{-4} \xi^{-3} \theta_\epsilon^{-2} \hat{u}_\epsilon^2 \mathrm{d}x\mathrm{d}t$. 

After inserting the above estimates explored in the steps (\textbf{1})-(\textbf{5}) into \eqref{jj3} and absorbing the gradient term $\mathbb{E}\int_{\mathcal {O}_T} \lambda^{-2} \mu^{-2} \xi^{-2} \theta_\epsilon^{-2}|\nabla\hat{y}_\epsilon|^2  \mathrm{d}x\mathrm{d}t$ on the R.H.S. of the resulted inequality, we arrive at
\begin{equation*}
\begin{split}
&\mathbb{E}\int_{\mathcal {O}_T} \lambda^{-2} \mu^{-2} \xi^{-2} \theta_\epsilon^{-2}(\hat{Y}_\epsilon^2   +    |\nabla\hat{y}_\epsilon|^2)  \mathrm{d}x\mathrm{d}t\leq C \Big(\mathbb{E}\int_{\mathcal {O}_T} \lambda^{-1} \mu^{-2} \xi^{-1} \theta ^{-2}|\textbf{b}|^2  \mathrm{d}x\mathrm{d}t\\
&+  \mathbb{E} \int_{\mathcal {O}_T}  \theta ^{-2} \hat{y}_\epsilon^2 \mathrm{d}x\mathrm{d}t +\mathbb{E}\int_{\mathcal {O}_T} \lambda^{-3} \mu^{-4} \xi^{-3} \theta  ^{-2}\phi^2\mathrm{d}x \mathrm{d}t+\mathbb{E}\int_{\mathcal {O}'_T} \lambda^{-3} \mu^{-4} \xi^{-3} \theta ^{-2} \hat{u}_\epsilon^2 \mathrm{d}x\mathrm{d}t \Big),
\end{split}
\end{equation*}
which combined with the estimate \eqref{4.11} yield that
\begin{equation}\label{jj4}
\begin{split}
&\mathbb{E}\int_{\mathcal {O}_T}  \theta_{\epsilon}^{-2}(\hat{y}_\epsilon^2   +  \lambda^{-2} \mu^{-2} \xi^{-3}  |\nabla\hat{y}_\epsilon|^2 +  \lambda^{-2} \mu^{-2} \xi^{-2}  \hat{Y}_\epsilon^2)  \mathrm{d}x\mathrm{d}t+ \frac{1}{\epsilon}\mathbb{E}\int_{\mathcal {O}} \hat{y}_\epsilon ^2(0) \mathrm{d}x\\
&  +\mathbb{E} \int_{\mathcal {O}'_T} \lambda^{-3} \mu^{-4} \xi^{-3} \theta ^{-2}\hat{u}_\epsilon^2 \mathrm{d}x\mathrm{d}t  \leq C\mathbb{E}\int_\mathcal {O} \lambda^{-1} \mu^{-2} (\xi^{-1} \theta^{-2}) (T)  y_T^2 \mathrm{d}x\\
&+ C\mathbb{E}\int_{\mathcal {O}_T} \lambda^{-3} \mu^{-4} \xi^{-3} \theta ^{-2}\phi^2\mathrm{d}x \mathrm{d}t  + C\mathbb{E}\int_{\mathcal {O}_T} \lambda^{-1} \mu^{-2} \xi^{-1} \theta ^{-2}|\textbf{b}|^2\mathrm{d}x\mathrm{d}t.
\end{split}
\end{equation}

Observing that the R.H.S. of \eqref{jj4} is independent of $\epsilon$, so there exist a subsequence of $(\hat{u}_\epsilon,\hat{y}_\epsilon,\hat{Y}_\epsilon)$ (still denoted by itself) and an element
$(\hat{u},\hat{y},\hat{Y})$ such that as $\epsilon\rightarrow 0$,
$\hat{u}_\epsilon \rightarrow \hat{u}$ weakly in $L^2_\mathbb{F}(0,T;L^2(\mathcal {O}'))$, $\hat{y}_\epsilon \rightarrow \hat{y}$ weakly in $L^2_\mathbb{F}(0,T;H^1_0(\mathcal {O}))$ and $\hat{Y}_\epsilon \rightarrow \hat{Y}$ weakly in $L^2_\mathbb{F}(0,T;L^2_0(\mathcal {O}))$. Following the   argument in Lemma \ref{lem7.1}, we can show that $(\hat{y},\hat{Y})$ is the solution to \eqref{4.1} associated with $\hat{u}$. Due to the uniform boundedness of $\frac{1}{\epsilon}\mathbb{E}\|\hat{y}_\epsilon(0) \|_{L^2(\mathcal {O})}^2$, one can take  $\epsilon \rightarrow0$ to deduce the null controllability.  Finally, since
$\xi^{-1}(T)\leq e^{-6\mu m} $ and $\theta^{-2}(T)\leq e^{2\lambda \mu e^{6\mu (m+1)}}$, the  
 estimate \eqref{4..3} is a consequence of \eqref{jj4}.  This completes the proof of Theorem \ref{lem4.1}.
\end{proof}

\subsection{The semi-linear controlled system}
As a byproduct of Theorem \ref{lem4.1}, we are ready to establish the null controllability of system \eqref{jj6}.

\begin{proof}[Proof of Theorem \ref{thm-nonlinear1}] Given any $\varphi \in L^2_\mathbb{F}(0,T;L^2(\mathcal {O}))$, we consider the system
\begin{equation}\label{jj7}
\left\{
\begin{aligned}
&\mathrm{d} y+\nabla\cdot (\mathcal {A}\nabla y) \mathrm{d}t= \left( \varphi +\textbf{1}_{\mathcal {O}'}u\right)\mathrm{d}t+Y\mathrm{d}W_t ~\textrm{in}~\mathcal {O}_T,\\
&y=0~\textrm{on}~\Sigma_T,\quad
y(T)=y_T~\textrm{in}~\mathcal {O}.
\end{aligned}
\right.
\end{equation}
Thanks to Theorem \ref{lem4.1}, for any $y_T \in L^2_\mathbb{F}(\Omega;L^2(\mathcal {O}))$, there exists     $u\in L^2_\mathbb{F}(0,T;L^2(\mathcal {O}'))$ such that the associated solution $(y, Y)$ to \eqref{jj6} satisfies $
y(0)=0$ in $\mathcal {O}$, $\mathbb{P}$-a.s.

Define  
$
\mathscr{B}_{\lambda,\mu}= \{\varphi\in L^2_\mathbb{F}(0,T;L^2(\mathcal {O}));~\mathbb{E}\int_{\mathcal {O}_T} \lambda^{-3} \mu^{-4} \xi^{-3} \theta ^{-2}\varphi^2\mathrm{d}x \mathrm{d}t<\infty\},
$
which is a Banach space equipped with the canonical norm. We \textit{claim} that the mapping 
\begin{equation}\label{fff}
\begin{split}
\mathscr{K}: \varphi \in \mathscr{B}_{\lambda,\mu} \mapsto F(\omega,t,x,y,\nabla y,Y)\in \mathscr{B}_{\lambda,\mu}
\end{split}
\end{equation}
is well-defined, where $(y,Y)$ denotes the solution to \eqref{jj7} with respect to $u$ and $y_T$, satisfying $y(0)=0$ in $\mathcal {O}$, $\mathbb{P}$-a.s. Indeed, by using assumption (A$_3$) and  \eqref{4..3} in Theorem \ref{lem4.1}, we get that for any fixed parameters $\lambda,\mu>1$ large enough
\begin{equation*}
\begin{split}
&\|\mathscr{K} \varphi\|_{\mathscr{B}_{\lambda,\mu}}
\leq \mathbb{E}\int_{\mathcal {O}_T} \lambda^{-3} \mu^{-4} \xi^{-3} \theta ^{-2}(y^2+|\nabla y|^2+Y^2)\mathrm{d}x \mathrm{d}t\\
&\leq C \frac{\exp\{4\lambda \mu e^{6\mu (m+1)}-6\mu m\}}{\lambda\mu^2} \mathbb{E}\|y_T\|_{L^2}^2+ C\mathbb{E}\int_{\mathcal {O}_T} \lambda^{-3} \mu^{-4} \xi^{-3} \theta ^{-2}\varphi^2\mathrm{d}x \mathrm{d}t
<  \infty.
\end{split}
\end{equation*}

Next, we show that $\mathscr{K}$ is a contraction mapping in $\mathscr{B}_{\lambda,\mu}$. For any $\varphi_1$ and $\varphi_2 \in \mathscr{B}_{\lambda,\mu}$, the corresponding solutions are denoted by $(y_1,U_1)$ and $(y_2,Y_2)$, respectively. Setting $\widetilde{\varphi}=\varphi_1-\varphi_2$, $\widetilde{u}=u_1-u_2$, $\widetilde{y}=y_1-y_2$ and $\widetilde{Y}=Y_1-Y_2$, then
\begin{equation}
\left\{
\begin{aligned}
&\mathrm{d} \widetilde{y}+\nabla\cdot (\mathcal {A}\nabla \widetilde{y}) \mathrm{d}t= \left( \widetilde{\varphi} +\textbf{1}_{\mathcal {O}'}\widetilde{u}\right)\mathrm{d}t+\widetilde{Y}\mathrm{d}W_t ~\textrm{in}~\mathcal {O}_T,\\
&\widetilde{y}=0~\textrm{on}~\Sigma_T,\quad
\widetilde{y}(T) =0~\textrm{in}~\mathcal {O},
\end{aligned}
\right.
\end{equation}
and the first component of the solution $(\widetilde{y},\widetilde{Y})$ satisfies $\widetilde{y}(0)=0$ in $\mathcal {O}$, $\mathbb{P}$-a.s. By using \eqref{4..3}, the  condition (A$_2$) and the fact of $\|\xi^{-1}\|_{L^\infty}<\infty$, we have
\begin{equation*}
\begin{split}
&\|\mathscr{K} \varphi_1-\mathscr{K} \varphi_2\|_{\mathscr{B}_{\lambda,\mu}}
\leq  C\Big(\lambda^{-3} \mu^{-4} \mathbb{E}\int_{\mathcal {O}_T} \theta ^{-2} \widetilde{y}^2 \mathrm{d}x \mathrm{d}t\\
&+\lambda^{-1} \mu^{-2}\mathbb{E}\int_{\mathcal {O}_T} \lambda^{-2} \mu^{-2} \xi^{-3} \theta ^{-2} (|\nabla \widetilde{y}|^2  +  \widetilde{Y}^2) \mathrm{d}x \mathrm{d}t\Big)\\
&\leq  C (\lambda^{-3} \mu^{-4}+\lambda^{-1} \mu^{-2})\mathbb{E}\int_{\mathcal {O}_T} \lambda^{-3} \mu^{-4} \xi^{-3} \theta ^{-2}\widetilde{\varphi}^2\mathrm{d}x \mathrm{d}t \leq C  \lambda^{-1} \mu^{-2} \|\varphi_1- \varphi_2\|_{\mathscr{B}_{\lambda,\mu}},
\end{split}
\end{equation*}
where $C>0$ is independent of $\lambda,\mu$. By choosing $\lambda,\mu>1$ sufficiently large such that $C \lambda^{-1} \mu^{-2} <1$,   $\mathscr{K}$ is a contraction mapping on $\mathscr{B}_{\lambda,\mu}$. According to the Banach Fixed-point Theorem, $\mathscr{K}$ has a unique fixed-point $\varphi \in \mathscr{B}_{\lambda,\mu}$ such that
$
\mathscr{K}\varphi=F(\omega,t,x,y,\nabla y,Y)=\varphi,
$
where $(y,Y)$ is the solution to \eqref{jj7} associated  to $y_T$ and $\varphi$, such that $y(0)=0$ in $\mathcal {O}$, $\mathbb{P}$-a.s. Therefore, $(y,Y)$ is a solution to  \eqref{jj6} such that the null controllability   holds. The proof of Theorem \ref{thm-nonlinear1} is completed.
\end{proof}

\section{Controllability of forward parabolic  SPDEs}

\subsection{Carleman estimates for backward system}

To prove the Carleman estimates in Theorem \ref{thm4}, let us introduce the controlled system 
\begin{equation}\label{3.1}
\left\{
\begin{aligned}
&\mathrm{d} y-\nabla\cdot(\mathcal {A}\nabla y) \mathrm{d}t= (\lambda^3 \mu^4 \mathring{\xi}^3 \mathring{\theta}^2z+ \textbf{1}_{\mathcal {O}'}u) \mathrm{d}t+ (\lambda^2 \mu^2 \mathring{\xi}^3\mathring{\theta}^2 Z+U) \mathrm{d}W_t~\textrm{in}~\mathcal {O}_T,\\
&y=0 ~\textrm{on}~\Sigma_T,\quad
y(0) =y_0 ~\textrm{in}~\mathcal {O},
\end{aligned}
\right.
\end{equation}
where $y$ denotes the state variable associated with $y_0 \in L^2_{\mathcal {F}_0}(\Omega; L^2(\mathcal {O}))$ and control pair $(u,U)$, and $(z,Z)$ is the unique solution of \eqref{back} with respect to $z_T\in L^2_{\mathcal {F}_T}(\Omega; L^2(\mathcal {O}))$. By using an argument similar to Theorem \ref{thm1} and the original work \cite{tang2009null}, one can establish the following $L^2$-Carleman estimate for \eqref{back} with $\textbf{b}\equiv \textbf{0}$.

\begin{lemma} \label{lem3.0}
For any $T>0$, assume that $\textbf{c}\in L_\mathbb{F}^\infty(0,T;L^\infty(\mathcal {O};\mathbb{R}^n))$, $\rho_1$, $\rho_2\in L_\mathbb{F}^\infty(0,T;L^\infty(\mathcal {O}))$ and $\phi \in L_\mathbb{F}^2(0,T;L^2(\mathcal {O}))$. If $
\textbf{b}\equiv\textbf{0}$ and the assumption (A$_1$) holds, then there exist $\lambda_0>0$ and $\mu_0>0$ such that the unique solution $(z,Z)$ to \eqref{back} with respect to $z_T \in L^2_{\mathcal {F}_T}(\Omega;L^2(\mathcal {O}))$ satisfies
\begin{equation}\label{backcar}
\begin{split}
& \mathbb{E}\int_{\mathcal {O}} e^{2\lambda \varphi (0)}(\lambda^{2} \mu^{3} e^{6 \mu m}  z^2(0)  +  |\nabla z(0)|^2 )\mathrm{d}x +\mathbb{E} \int_{\mathcal {O}_T}\lambda\mu^{2} \mathring{\xi}\mathring{\theta}^2 ( | \nabla z |^2 + \lambda^2\mu^2\mathring{\xi}^2   z^2) \mathrm{d}x\mathrm{d}t \\
&  \leq C \Big(\mathbb{E}\int_{\mathcal {O}'_T} \lambda^{3}\mu^{4} \mathring{\xi}^{3}\mathring{\theta}^2 z^2 \mathrm{d}x\mathrm{d}t+\mathbb{E}\int_{\mathcal {O}_T}\mathring{\theta}^2(\phi ^2+\lambda^{2} \mu^{2}  \mathring{\xi}^{3}Z^2 ) \mathrm{d}x\mathrm{d}t \Big),
\end{split}
\end{equation}
for all $\lambda \geq \lambda_0$ and $\mu \geq \mu_0$.
\end{lemma}


\begin{lemma}\label{lem3.1}
Let $(z,Z)$ be solution to  \eqref{back} associated with  $z_T\in L^2_{\mathcal{F}_T}(\Omega; L^2(\mathcal {O}))$. Then there exists $(\tilde{u},\tilde{U}) \in L^2_\mathbb{F}(0,T;L^2(\mathcal {O}'))\times L^2_\mathbb{F}(0,T;L^2(\mathcal {O}))$ such that the corresponding solution $\tilde{y}$ to \eqref{3.1} verifies
$
\tilde{y}(T)=0$ in $\mathcal {O}$, $\mathbb{P}$-a.s.
Moreover, there exists a positive constants $C$, $\lambda_0,\mu_0$, depending only on $\mathcal {O},\mathcal {O}'$ and $T$, such that
\begin{equation} \label{3.2}
\begin{split}
& \mathbb{E} \int_{\mathcal {O}'_T} \lambda^{-3} \mu^{-4} \mathring{\xi}^{-3} \mathring{\theta}^{-2}\tilde{u}^2 \mathrm{d}x\mathrm{d}t +\mathbb{E}\int_{\mathcal {O}_T} \mathring{\theta}^{-2} (\tilde{y}^2  + \lambda ^{-2}\mu ^{-2} \mathring{\xi} ^{-3} |\nabla\tilde{y}|^2)   \mathrm{d}x \mathrm{d}t \\
&+ \mathbb{E}\int_{\mathcal {O}_T} \lambda^{-2} \mu^{-2} \mathring{\xi}^{-3} \mathring{\theta}^{-2}\tilde{U}^2 \mathrm{d}x \mathrm{d}t  \leq C\Big(\mathbb{E}\int_{\mathcal {O}} \lambda^{-2} \mu^{-3} e^{-2 \mu(6 m+1)} \mathring{\theta}^{-2}(0) y_0 ^2 \mathrm{d}x \\
&+ \mathbb{E}\int_{\mathcal {O}_T} \lambda^3 \mu^4 \mathring{\xi}^3 \mathring{\theta}^2z^2   \mathrm{d}x \mathrm{d}t+  \mathbb{E}\int_{\mathcal {O}_T}  \lambda^2 \mu^2 \mathring{\xi}^3\mathring{\theta}^2Z^2 \mathrm{d}x \mathrm{d}t\Big),
\end{split}
\end{equation}
for all $\lambda\geq\lambda_0$ and $\mu \geq \mu_0$.
\end{lemma}

\begin{proof}
Setting $\mathring{\theta}_\epsilon=e^{\lambda \mathring{\varphi} _{\epsilon}}$, where
$
\mathring{\varphi}_{\epsilon}(x,t)= \mathring{\gamma}_{\epsilon}(t)\big(e^{\mu (\beta(x)+6m)}-\mu e^{6\mu (m+1)}\big) 
$
and
\begin{equation} \label{3.3}
\mathring{\gamma}_{\epsilon}(t)=\left\{
\begin{aligned}
&1+(1- 4t/T)^\sigma  && \textrm{in}~ {[0,T/4]},\\
&1  && \textrm{in}~{[T/4,T/2+\epsilon]},\\
&\textrm{is increasing}  &&  \textrm{in}~{[T/2+\epsilon,3T/4 ]},\\
&1/(T-t+\epsilon)^m  &&  \textrm{in}~{[3T/4 ,T)}.
\end{aligned}
\right.
\end{equation}
From the definition of \eqref{3.3} and the property of $\mathring{\gamma}$, we are readily to see that $\mathring{\varphi}_{\epsilon} $ is non-degenerate at $t=T$, and $\mathring{\gamma}(t)\geq\mathring{\gamma}_{\epsilon}(t)$ for any $t\in [0,T]$.
Define an admissible control set $\mathcal {U}$, which consists of all $(u,U)\in L^2_\mathbb{F}(0,T;L^2(\mathcal {O}'))\times L^2_\mathbb{F}(0,T;L^2(\mathcal {O}))$ such that
$\mathbb{E} \int_{\mathcal {O}'_T} (\lambda^{-3} \mu^{-4} \mathring{\xi}^{-3} \mathring{\theta} ^{-2}u^2+\lambda^{-2} \mu^{-2} \mathring{\xi}^{-3} \mathring{\theta} ^{-2}U^2)\mathrm{d}x\mathrm{d}t<\infty $.

Consider the problem $
(\overline{\textrm{\textbf{P}}}_\epsilon):\inf_{(u,U) \in \mathcal {U}} J_\epsilon(u,U) \hspace{2mm}\textrm{subject to \eqref{3.1}}
$
with
$
J_\epsilon(u,U) = \frac{1}{2} (\mathbb{E}\int_{\mathcal {O}_T} \mathring{\theta}_\epsilon^{-2} y^2 \mathrm{d}x\mathrm{d}t+\mathbb{E} \int_{\mathcal {O}'_T} \lambda^{-3} \mu^{-4} \mathring{\xi}^{-3} \mathring{\theta} ^{-2}u^2 \mathrm{d}x\mathrm{d}t + \mathbb{E}\int_{\mathcal {O}_T} \lambda^{-2} \mu^{-2} \mathring{\xi}^{-3} \mathring{\theta} ^{-2} U^2\mathrm{d}x\mathrm{d}t+\frac{1}{\epsilon}\mathbb{E}\int_{\mathcal {O}} y^2(T)\mathrm{d}x ) 
$. Similar to the proof of Lemma \ref{lem7.1}, 
  $(\overline{\textrm{\textbf{P}}}_\epsilon)$ admits a unique optimal control $(u_\epsilon,U_\epsilon)=(\lambda^{3} \mu^{4} \mathring{\xi}^{3} \mathring{\theta} ^{2}z_\epsilon\textbf{1}_{\mathcal {O}'},\lambda^{2} \mu^{2} \mathring{\xi}^{3} \mathring{\theta} ^{2}Z_\epsilon)\in \mathcal {U}$, where $(z_\epsilon,Z_\epsilon)$ solves
\begin{equation}\label{3.6}
\left\{
\begin{aligned}
&\mathrm{d} z_\epsilon+\nabla\cdot(\mathcal {A}\nabla z_{\epsilon})  \mathrm{d}t= \mathring{\theta}_\epsilon^{-2} y_\epsilon  \mathrm{d}t+ Z_\epsilon \mathrm{d}W_t ~\textrm{in}~\mathcal {O}_T,\\
&z_\epsilon=0~ \textrm{on}~\Sigma_T,\quad
z_\epsilon(T)= -\frac{1}{\epsilon} y_\epsilon (T) ~\textrm{in}~\mathcal {O},
\end{aligned}
\right.
\end{equation}
and $ y_\epsilon $ is the solution to \eqref{3.1} associated with $(u_\epsilon,U_\epsilon)$.  

\textsc{Estimate for $y_\epsilon$.} Let us define  
$\mathcal {F}_\epsilon=  \frac{1}{\epsilon}\mathbb{E}\int_\mathcal {O} y_\epsilon^2 (T)\mathrm{d}x+\mathbb{E} \int_{\mathcal {O}'_T} \lambda^{3} \mu^{4} \mathring{\xi}^{3} \mathring{\theta}^{2}z_\epsilon^2 \mathrm{d}x\mathrm{d}t+ \mathbb{E}\int_{\mathcal {O}_T} \lambda^{2} \mu^{2} \mathring{\xi}^{3} \mathring{\theta}^{2}Z_\epsilon^2 \mathrm{d}x \mathrm{d}t+\mathbb{E}\int_{\mathcal {O}_T} \mathring{\theta}_\epsilon^{-2}  y_\epsilon ^2 \mathrm{d}x \mathrm{d}t 
$. Applying the It\^{o} formula to $y_\epsilon  z_\epsilon$, it follows from \eqref{3.1}, \eqref{backcar}  and \eqref{3.6} that, for any $\delta>0$,
\begin{equation}\label{3.8}
\begin{split}
\mathcal {F}_\epsilon&\leq   \delta \mathbb{E} \int_{\mathcal {O}'_T} \lambda^3 \mu^4 \mathring{\xi}^3 \mathring{\theta}^2z_\epsilon^2 \mathrm{d} x \mathrm{d} t+  \delta \mathbb{E}\int_{\mathcal {O}_T}  (\theta_\epsilon^{-2} y_\epsilon^2 + \lambda^2 \mu^2 \mathring{\xi}^3\mathring{\theta}_\epsilon^2 Z_\epsilon^2 ) \mathrm{d} x \mathrm{d} t   \\
&  +\delta\mathbb{E}\int_{\mathcal {O}_T}  \lambda^2 \mu^2 \mathring{\xi}^3\mathring{\theta}_\epsilon^2Z_\epsilon^2 \mathrm{d}x \mathrm{d}t +C\mathbb{E}\int_{\mathcal {O}} \lambda^{-2} \mu^{-3} e^{-2 \mu(6 m+1)} \mathring{\theta}^{-2}(0) y_0 ^2 \mathrm{d}x\\
&  + C\mathbb{E}\int_{\mathcal {O}_T}  (\lambda^3 \mu^4 \mathring{\xi}^3 \mathring{\theta}^2z^2  + \lambda^2 \mu^2 \mathring{\xi}^3\mathring{\theta}^2Z^2 ) \mathrm{d}x \mathrm{d}t.
\end{split}
\end{equation}
Choosing $\delta>0$ sufficiently small, and using the representation of $(u_\epsilon,U_\epsilon)$, we get
\begin{fontsize}{9pt}{9pt}\begin{equation}\label{3.10}
\begin{split}
& \frac{1}{\epsilon}\mathbb{E}\int_\mathcal {O}y_\epsilon^2 (T)\mathrm{d}x+\mathbb{E} \int_{\mathcal {O}'_T} \lambda^{-3} \mu^{-4} \mathring{\xi}^{-3} \mathring{\theta}_{\epsilon}^{-2}u_\epsilon^2 \mathrm{d}x\mathrm{d}t+ \mathbb{E}\int_{\mathcal {O}_T} \mathring{\theta}_\epsilon^{-2} ( y_\epsilon ^2 +   \lambda^{-2} \mu^{-2} \mathring{\xi}^{-3} U_\epsilon^2) \mathrm{d}x \mathrm{d}t\\
&  \leq C\mathbb{E}\int_{\mathcal {O}} \lambda^{-2} \mu^{-3} e^{-2 \mu(6 m+1)} \mathring{\theta}^{-2}(0) y_0 ^2 \mathrm{d}x +C\mathbb{E}\int_{\mathcal {O}_T} \lambda^2 \mu^2 \mathring{\xi}^3\mathring{\theta}^2(\lambda \mu^2 z^2  + Z^2) \mathrm{d}x \mathrm{d}t.
\end{split}
\end{equation}
\end{fontsize}
	
\textsc{Estimate for $\nabla y_\epsilon$.} By applying the It\^{o} formula to $\lambda ^{-2}\mu ^{-2}\mathring{\theta}_\epsilon^{-2} \mathring{\xi} ^{-3}  y_\epsilon ^2$, we have
\begin{fontsize}{9.5pt}{9.5pt}\begin{equation}\label{3.11}
\begin{split}
&\mathbb{E} \int_{\mathcal {O}_T}\sum_{i,j} 2\lambda ^{-2}\mu ^{-2}\mathring{\theta}_\epsilon^{-2} \mathring{\xi} ^{-3} a^{ij} y_{\epsilon,x_i} y_{\epsilon,x_j}   \mathrm{d}x \mathrm{d}t-\mathbb{E} \int_{\mathcal {O}_T}(\lambda ^{-2}\mu ^{-2}\mathring{\theta}_\epsilon^{-2} \mathring{\xi} ^{-3})_t y_\epsilon ^2 \mathrm{d}x \mathrm{d}t \\
&= \mathbb{E} \int_{\mathcal {O}_T}2 \lambda  \mu^2 \mathring{\theta}_\epsilon^{-2} \mathring{\theta}^2  y_\epsilon  z  \mathrm{d}x \mathrm{d}t    +\mathbb{E} \int_{\mathcal {O}_T}(\lambda ^{-2}\mu ^{-2}\mathring{\theta}_\epsilon^{-2} \mathring{\xi} ^{-3} U_\epsilon ^2 + \lambda^2 \mu^2 \mathring{\xi}^3 \mathring{\theta}_\epsilon ^{-2}\mathring{\theta}^4 Z^2)  \mathrm{d}x\mathrm{d} t \\
&+ \mathbb{E} \int_{\mathcal {O}'_T}2\lambda ^{-2}\mu ^{-2}\mathring{\theta}_\epsilon^{-2} \mathring{\xi} ^{-3} y_\epsilon    u_\epsilon \mathrm{d}x \mathrm{d}t+ \mathbb{E} \int_{\mathcal {O}} \lambda ^{-2}\mu ^{-2}\mathring{\theta}_\epsilon^{-2}(0) \mathring{\xi} ^{-3}(0)  y_0 ^2  \mathrm{d}x\\
&+\mathbb{E} \int_{\mathcal {O}_T}2  \mathring{\theta}_\epsilon ^{-2}\mathring{\theta}^2 Z U_\epsilon\mathrm{d}x\mathrm{d} t    - \mathbb{E} \int_{\mathcal {O}_T}2\sum_{i,j} a^{ij}(\lambda ^{-2}\mu ^{-2}\mathring{\theta}_\epsilon^{-2} \xi ^{-3})_{x_j} y_{\epsilon,x_i}   y_\epsilon   \mathrm{d}x \mathrm{d}t,
\end{split}
\end{equation}
where we used the fact that $\mathring{\varphi}^{-2}$ is non-degenerate at $t=0$ and $\mathring{\varphi}^{-2}(T)=0$.
\end{fontsize}

To estimate the second term on the L.H.S. of \eqref{3.11}, we divide the time interval into $[0,T]=[0,T/4]\cup [T/4,T]$. Since $1\leq\mathring{\gamma} \leq 2$ over $[0,T/4]$, $\mathring{\varphi}<0$ over $\mathcal {O}_T$  and $
\mathring{\gamma}_t  =-\frac{4}{T}\lambda\mu^2\sigma(1- \frac{4t}{T})^{\sigma-1}e^{\mu(6m-4)} \leq 0$, we find
$- (\lambda ^{-2}\mu ^{-2}\mathring{\theta}_\epsilon^{-2} \mathring{\xi} ^{-3} )_t 
\geq0$, which implies
\begin{equation}\label{3.13}
\begin{split}
-\mathbb{E} \int_0^{T/4}\int_{\mathcal {O}}(\lambda ^{-2}\mu ^{-2}\mathring{\theta}_\epsilon^{-2} \mathring{\xi} ^{-2})_t y_\epsilon ^2 \mathrm{d}x \mathrm{d}t\geq 0.
\end{split}
\end{equation}
Moreover, as
$| \mathring{\gamma}_t |\leq C \mathring{\gamma}^2$, $|\mathring{\gamma}_{\epsilon,t}|\leq C \mathring{\gamma}_\epsilon^2\leq C \mathring{\gamma} ^2$,  for all $t\in [T/4,T]$,
we have
\begin{equation*}
\begin{split}
|(\lambda ^{-2}\mu ^{-2}\mathring{\theta}_\epsilon^{-2} \xi ^{-3})_t|
\leq & C \left(\lambda ^{-1}\mu ^{-1}  +  \lambda ^{-2}\mu ^{-1} \right) \mathring{\theta}_\epsilon^{-2} \mathring{\xi} ^{-1}  \frac{e^{6\mu (m+1)}}{e^{2\mu (\beta+6m)}}
\leq  C  \lambda ^{-1}\mu ^{-1}\mathring{\theta}_\epsilon^{-2},
\end{split}
\end{equation*}
which implies that
\begin{equation}\label{3.14}
\begin{split}
 \mathbb{E} \int_{T/4}^T\int_{\mathcal {O}}(\lambda ^{-2}\mu ^{-2}\mathring{\theta}_\epsilon^{-2} \mathring{\xi} ^{-3})_t y_\epsilon ^2 \mathrm{d}x \mathrm{d}t  &\leq \mathbb{E} \int_{T/4}^T\int_{\mathcal {O}}\lambda ^{-1}\mu ^{-1}\mathring{\theta}_\epsilon^{-2}  y_\epsilon ^2 \mathrm{d}x \mathrm{d}t.
\end{split}
\end{equation}
From the estimates \eqref{3.13}-\eqref{3.14}, we infer that
\begin{equation}\label{3.15}
\begin{split}
-\mathbb{E} \int_{\mathcal {O}_T}(\lambda ^{-2}\mu ^{-2}\mathring{\theta}_\epsilon^{-2} \mathring{\xi} ^{-3})_t y_\epsilon ^2 \mathrm{d}x \mathrm{d}t \geq -\mathbb{E}  \int_{\mathcal {O}_T} \mathring{\theta}_\epsilon^{-2}  y_\epsilon ^2 \mathrm{d}x \mathrm{d}t.
\end{split}
\end{equation}
By \eqref{3.11} and \eqref{3.15}, we get from the assumption (A$_1$) that
\begin{fontsize}{9.6pt}{9.6pt}\begin{equation} \label{3.17}
\begin{split}
&\mathbb{E} \int_{\mathcal {O}_T}2c_0\lambda ^{-2}\mu ^{-2}\mathring{\theta}_\epsilon^{-2} \mathring{\xi} ^{-3} |\nabla y_\epsilon|^2   \mathrm{d}x \mathrm{d}t
\leq C\Big(
\mathbb{E} \int_{T/4}^T\int_{\mathcal {O}}\lambda ^{-1}\mu ^{-1}\mathring{\theta}_\epsilon^{-2} \mathring{\xi} ^{-1} y_\epsilon ^2 \mathrm{d}x \mathrm{d}t\\
&+ \mathbb{E} \int_{\mathcal {O}_T} \lambda ^{-1}\mu ^{-1}\mathring{\theta}_\epsilon^{-2}\mathring{\xi}^{-2}|\nabla y_\epsilon| | y_\epsilon |  \mathrm{d}x \mathrm{d}t
  + \mathbb{E} \int_{\mathcal {O}'_T} \lambda ^{-2}\mu ^{-2}\mathring{\theta}_\epsilon^{-2} \mathring{\xi} ^{-3}| y_\epsilon    u_\epsilon| \mathrm{d}x \mathrm{d}t\\
  &+ \mathbb{E} \int_{\mathcal {O}_T} \lambda  \mu^2 \mathring{\theta}_\epsilon^{-1}\mathring{\theta} | y_\epsilon  z | \mathrm{d}x \mathrm{d}t
   +\mathbb{E} \int_{\mathcal {O}_T}  \lambda^2 \mu^2 \mathring{\xi}^3 \mathring{\theta}^2 Z^2  \mathrm{d}x \mathrm{d} t+\mathbb{E} \int_{\mathcal {O}_T}  \mathring{\theta}_\epsilon ^{-1}\mathring{\theta} | Z U_\epsilon|\mathrm{d}x\mathrm{d} t
\\
&+\mathbb{E} \int_{\mathcal {O}_T}\lambda ^{-2}\mu ^{-2}\mathring{\theta}_\epsilon^{-2} \mathring{\xi} ^{-3} U_\epsilon ^2  \mathrm{d}x\mathrm{d} t + \mathbb{E} \int_{\mathcal {O}} \lambda ^{-2}\mu ^{-2}\mathring{\theta}_\epsilon^{-2}(0) \mathring{\xi} ^{-3}(0) y_0 ^2  \mathrm{d}x\Big).
\end{split}
\end{equation}
\end{fontsize}
where we used the fact of $
|\nabla (\lambda ^{-2}\mu ^{-2}\mathring{\theta}_\epsilon^{-2} \mathring{\xi} ^{-3})|\leq C\lambda ^{-1}\mu ^{-1}\mathring{\theta}_\epsilon^{-2}\mathring{\xi}^{-2}
$. By using $\mathring{\theta}\mathring{\theta}_\epsilon^{-1}\leq1$ and applying the Young inequality to the R.H.S. of \eqref{3.17}, we get from \eqref{3.10} that
\begin{fontsize}{9.5pt}{9.5pt}\begin{equation*} 
\begin{split}
& \frac{1}{\epsilon}\mathbb{E}\int_\mathcal {O}y_\epsilon^2 (T)\mathrm{d}x+\mathbb{E}\int_{\mathcal {O}_T} \mathring{\theta}_\epsilon^{-2}  y_\epsilon ^2 \mathrm{d}x \mathrm{d}t+\mathbb{E} \int_{\mathcal {O}'_T} \lambda^{-3} \mu^{-4} \mathring{\xi}^{-3} \mathring{\theta}_{\epsilon}^{-2}u_\epsilon^2 \mathrm{d}x\mathrm{d}t\\
& +\mathbb{E} \int_{\mathcal {O}_T} \lambda ^{-2}\mu ^{-2}\mathring{\xi} ^{-3}\mathring{\theta}_\epsilon^{-2}  |\nabla y_\epsilon|^2   \mathrm{d}x \mathrm{d}t+ \mathbb{E}\int_{\mathcal {O}_T} \lambda^{-2} \mu^{-2} \mathring{\xi}^{-3} \mathring{\theta}_{\epsilon}^{-2}U_\epsilon^2 \mathrm{d}x \mathrm{d}t\\
&  \leq C\Big(\mathbb{E}\int_{\mathcal {O}} \lambda^{-2} \mu^{-3} e^{-2 \mu(6 m+1)} \mathring{\theta}^{-2}(0) y_0 ^2 \mathrm{d}x+ \mathbb{E}\int_{\mathcal {O}_T} \lambda^2 \mu^2 \mathring{\xi}^3\mathring{\theta}^2 (\lambda \mu^2z^2  + Z^2 ) \mathrm{d}x \mathrm{d}t\Big).
\end{split}
\end{equation*}\end{fontsize}
Since the R.H.S. of last inequality is uniformly bounded in $\epsilon$, the subsequent proof proceeds analogously to that of Lemma \ref{lem7.1}, we shall omit the details here.
\end{proof}

\begin{proof}[Proof of Theorem \ref{thm4}] The proof proceeds by analogy with the argument  for Theorem \ref{thm2}, we only emphasize the key estimates. Let $(z,Z)$ be   solution to \eqref{back}, and $\tilde{y}$ be solution to \eqref{3.1} associated with the control $(\tilde{u},\tilde{U})$ (cf. Lemma \ref{lem3.1}).  

\textsc{Estimate for $z$.} By applying the It\^{o} formula to $\tilde{y}z$ and considering the special case of $ y_0 \equiv0$ in Lemma \ref{lem3.1}, we obtain that for any $\delta>0$
\begin{fontsize}{8.5pt}{8.5pt}\begin{equation*} 
\begin{split}
&\mathbb{E}\int_{\mathcal {O}_T} \lambda^3 \mu^4 \mathring{\xi}^3 \mathring{\theta}^2z^2\mathrm{d}x \mathrm{d}t
 \leq  C\Big(\mathbb{E} \int_{\mathcal {O}'_T}\lambda ^{3}\mu ^{4}\mathring{\theta}^{2} \mathring{\xi} ^{3}z^2 \mathrm{d}x \mathrm{d}t  +   \mathbb{E} \int_{\mathcal {O}_T} \mathring{\theta}^{2}\big[ \phi^2 +\lambda^{2} \mu^{2} \mathring{\xi}^{3} (Z^2 +     |\textbf{b}|^2) \big]  \mathrm{d}x \mathrm{d}t\Big).
\end{split}
\end{equation*}
\end{fontsize}

\textsc{Estimate for $\nabla z$.} Applying the It\^{o} formula to  $\lambda \mu^2 \mathring{\xi} \mathring{\theta}^2z^2$ and then integrating by parts over $\mathcal {O}_T$, we infer that
\begin{fontsize}{8.5pt}{8.5pt}\begin{equation}\label{3.30}
\begin{split}
&\mathbb{E}\int_\mathcal {O} \lambda \mu^2 (\mathring{\xi}\mathring{\theta}^2)(0) z^2(0) \mathrm{d}x+ \mathbb{E}\int_{\mathcal {O}_T}\sum_{i,j} 2\lambda \mu^2 \mathring{\xi} \mathring{\theta}^2 a^{ij}z_{x_i} z _{x_j}\mathrm{d}x\mathrm{d}t+\mathbb{E}\int_{\mathcal {O}_T} 2\lambda^2 \mu^2 \frac{ \mathring{\gamma}_t }{\mathring{\gamma}}  \mathring{\varphi} \mathring{\theta}^2 z^2 \mathrm{d}x\mathrm{d}t \\
&  =  \mathbb{E}\int_{\mathcal {O}_T} 2 \lambda \mu^2 z\textbf{b}\cdot\nabla(\mathring{\xi} \mathring{\theta}^2) \mathrm{d}x\mathrm{d}t-\mathbb{E}\int_{\mathcal {O}_T} \lambda \mu^2  \frac{ \mathring{\gamma}_t }{\mathring{\gamma}}\mathring{\xi} \mathring{\theta}^2    z^2 \mathrm{d}x\mathrm{d}t-\mathbb{E}\int_{\mathcal {O}_T} 2\lambda \mu^2 \mathring{\xi} \mathring{\theta}^2 z \phi \mathrm{d}x\mathrm{d}t\\
&   +\mathbb{E}\int_{\mathcal {O}_T} 2 \lambda \mu^2 \mathring{\xi} \mathring{\theta}^2 z \textbf{c}\cdot\nabla z   \mathrm{d}x\mathrm{d}t-\mathbb{E}\int_{\mathcal {O}_T} 2\sum_{i,j} \lambda \mu^2 a^{ij}(\mathring{\xi} \mathring{\theta}^2)_{x_j}z_{x_i}z \mathrm{d}x\mathrm{d}t\\
&-\mathbb{E}\int_{\mathcal {O}_T} 2\lambda \mu^2 \mathring{\xi} \mathring{\theta}^2 \left( \rho_1 z^2-\textbf{b}\cdot \nabla z+\rho_2 zZ \right) \mathrm{d}x\mathrm{d}t +\mathbb{E}\int_{\mathcal {O}_T} \lambda \mu^2 \mathring{\xi} \mathring{\theta}^2Z^2 \mathrm{d}x\mathrm{d}t ,
\end{split}
\end{equation}
\end{fontsize}
where we   used the identity $(\mathring{\xi}\mathring{\theta}^2)_t=\frac{ \mathring{\gamma}_t }{\mathring{\gamma}}\mathring{\xi} \mathring{\theta}^2+ \frac{2 \mathring{\gamma}_t }{\mathring{\gamma}}\lambda \mathring{\varphi} \mathring{\theta}^2$ and the fact of $\mathring{\theta}(T)=0$. By using the assumption (A$_1$) and the similar argument as we did for \eqref{7.9}, one can deduce the following estimate 
\begin{equation*}
\begin{split}
\textrm{L.H.S. of \eqref{3.30}}&\geq  \mathbb{E}\int_\mathcal {O} \lambda \mu^2 (\mathring{\xi}\mathring{\theta}^2)(0) z^2(0) \mathrm{d}x+\mathbb{E}\int_{\mathcal {O}_T} 2c_0\lambda \mu^2 \mathring{\xi} \mathring{\theta}^2 |\nabla z|^2\mathrm{d}x\mathrm{d}t\\
& +\mathbb{E}\int_0^{T/4} \int_{\mathcal {O}} \lambda^2\mu^2\frac{\mathring{\gamma}_t }{\mathring{\gamma}}\mathring{\varphi}\mathring{\theta}^2  z^2\mathrm{d}x\mathrm{d}t .
\end{split}
\end{equation*}
For the R.H.S. of \eqref{3.30}, by using $|\nabla(\mathring{\xi} \mathring{\theta}^2) |\leq \lambda\mu\mathring{\xi}^2\mathring{\theta}^2$, $|\mathring{\xi}_t|\leq C \lambda\mu\mathring{\xi}^3$ and the Young inequality, we have for any $\delta>0$
\begin{fontsize}{8.6pt}{8.6pt}\begin{equation*}
\begin{split}
&\mathbb{E}\int_{\mathcal {O}_T}\sum_{i,j} 2\lambda \mu^2 a^{ij}(\mathring{\xi} \mathring{\theta}^2)_{x_j}z_{x_i}z \mathrm{d}x\mathrm{d}t \leq \delta\mathbb{E}\int_{\mathcal {O}_T}  \lambda  \mu^2\mathring{\xi} \mathring{\theta}^2 |\nabla z|^2 \mathrm{d}x\mathrm{d}t+C\mathbb{E}\int_{\mathcal {O}_T}  \lambda^3 \mu^4\mathring{\xi}^3\mathring{\theta}^2 z^2 \mathrm{d}x\mathrm{d}t,\\
&\mathbb{E}\int_{\mathcal {O}_T} 2\lambda \mu^2 \mathring{\xi} \mathring{\theta}^2 \left( \rho_1 z^2-\textbf{b}\cdot \nabla z+\rho_2 zZ \right) \mathrm{d}x\mathrm{d}t \leq \delta\mathbb{E}\int_{\mathcal {O}_T}  \lambda \mu^2  \mathring{\xi}  \mathring{\theta}^2  |\nabla z|^2 \mathrm{d}x\mathrm{d}t\\
&\quad +C\mathbb{E}\int_{\mathcal {O}_T}  \lambda \mu^2 \mathring{\xi} \mathring{\theta}^2 \left(z^2 + Z^2\right) \mathrm{d}x\mathrm{d}t+ C\mathbb{E}\int_{\mathcal {O}_T}  \lambda  \mu^2 \mathring{\xi}  \mathring{\theta}^2 |\textbf{b} |^2 \mathrm{d}x\mathrm{d}t ,\\
&\mathbb{E}\int_{\mathcal {O}_T}  2\lambda \mu^2 \mathring{\xi} \mathring{\theta}^2 z \phi \mathrm{d}x\mathrm{d}t  \leq C  \mathbb{E}\int_{\mathcal {O}_T} \lambda^2 \mu^4 \mathring{\xi}^2 \mathring{\theta}^2 z^2 \mathrm{d}x\mathrm{d}t+C\mathbb{E}\int_{\mathcal {O}_T} \mathring{\theta}^2 \phi^2 \mathrm{d}x\mathrm{d}t , \\
&\mathbb{E}\int_{\mathcal {O}_T} \lambda \mu^2 \mathring{\xi}_t \mathring{\theta}^2    z^2 \mathrm{d}x\mathrm{d}t \leq C \mathbb{E}\int_{\mathcal {O}_T} \lambda^2 \mu^3   \mathring{\xi}^3 \mathring{\theta}^2    z^2 \mathrm{d}x\mathrm{d}t,\\
&\mathbb{E}\int_{\mathcal {O}_T} 2\lambda \mu^2z\textbf{b}\cdot\nabla(\mathring{\xi} \mathring{\theta}^2)  \mathrm{d}x\mathrm{d}t   \leq C \mathbb{E}\int_{\mathcal {O}_T}  \lambda^2\mu^4 \mathring{\xi} \mathring{\theta}^2 z^2 \mathrm{d}x\mathrm{d}t+ C\mathbb{E}\int_{\mathcal {O}_T}  \lambda^2 \mu^2 \mathring{\xi}^3 \mathring{\theta}^2 |\textbf{b} |^2 \mathrm{d}x\mathrm{d}t ,\\
&\mathbb{E}\int_{\mathcal {O}_T} 2 \lambda \mu^2 \mathring{\xi} \mathring{\theta}^2z \textbf{c}\cdot \nabla z \mathrm{d}x\mathrm{d}t  \leq \delta\mathbb{E}\int_{\mathcal {O}_T}  \mathring{\theta}^2 |\nabla z |^2 \mathrm{d}x\mathrm{d}t+ C\mathbb{E}\int_{\mathcal {O}_T}  \lambda^2 \mu^4 \mathring{\xi}^2 \mathring{\theta}^2 z^2 \mathrm{d}x\mathrm{d}t.
\end{split}
\end{equation*}
\end{fontsize}
Putting the above estimates into \eqref{3.30} and choosing $\delta>0$ small enough, we get the desired estimate \eqref{carleman3}.
\end{proof}

\subsection{Controllability of forward semi-linear parabolic SPDEs}

Thanks to Theorem \ref{thm4}, we can now establish null controllability for the parabolic SPDEs\begin{equation} \label{back1}
	\left\{
	\begin{aligned}
		&\mathrm{d} \hat{z}-\nabla\cdot (\mathcal {A}\nabla \hat{z}) \mathrm{d}t= \left(\langle \textbf{c}, \nabla \hat{z}\rangle+ \rho \hat{z}+ \phi + \nabla\cdot \textbf{b}+\textbf{1}_{\mathcal {O}'}\hat{u} \right)\mathrm{d}t+\hat{U}\mathrm{d}W_t~\textrm{in}~\mathcal {O}_T,\\
		&\hat{z}=0 ~\textrm{on}~\Sigma_T,\quad
		\hat{z}(0)  =z_0 ~\textrm{in}~\mathcal {O}.
	\end{aligned}
	\right.
\end{equation}

\begin{lemma} \label{thm5}
Assume that condition (A$_1$) holds, $\textbf{c}\in L_\mathbb{F}^\infty(0,T;L^\infty(\mathcal {O};\mathbb{R}^n))$, $\rho \in L_\mathbb{F}^\infty(0,T;L^\infty(\mathcal {O}))$, $\phi  \in L_\mathbb{F}^2(0,T;L^2(\mathcal {O}))$ and
$
\textbf{b} \in L^2_\mathbb{F}( 0,T;L^2(\mathcal {O};\mathbb{R}^n)) .
$
Then for any $z_0\in L^2_{\mathcal {F}_0}(\Omega; L^2(\mathcal {O}))$, there exists  $(\hat{u},\hat{U})$ such that the solution $\hat{z}$ to \eqref{back1} satisfies  $\hat{z}( T)=0$ in $\mathcal {O}$, $\mathbb{P}$-a.s. Moreover, there holds 
\begin{equation}\label{rb}
\begin{split}
&\mathbb{E}\int_{\mathcal {O}_T} \mathring{\theta}^{-2}\big[ \hat{z} ^2 +\lambda^{-2} \mu^{-2}\mathring{\xi}^{-3} (|\nabla \hat{z}|^2+\hat{U} ^2)\big] \mathrm{d}x\mathrm{d}t +\mathbb{E}\int_{\mathcal {O}'_T} \lambda^{-3}\mu^{-4}\mathring{\xi}^{-3}\mathring{\theta}^{-2}\hat{u}^2\mathrm{d}x\mathrm{d}t\\
&  \leq C\Big(\mathbb{E}\int_\mathcal {O} \lambda^{-1} \mu^{-2} e^{-6\mu m} e^{-2\lambda\varphi(0)}  z^2_0 \mathrm{d}x +   \mathbb{E}\int_{\mathcal {O}_T} \lambda^{-3} \mu^{-4} \mathring{\xi}^{-3} \mathring{\theta}^{-2}\phi^2\mathrm{d}x \mathrm{d}t\\
&+ \mathbb{E}\int_{\mathcal {O}_T} \lambda^{-1} \mu^{-2} \mathring{\xi}^{-1} \mathring{\theta}^{-2} |\textbf{b}|^2\mathrm{d}x\mathrm{d}t\Big),
\end{split}
\end{equation}
for all sufficiently large parameters $\lambda$ and $\mu$.
\end{lemma}

\begin{proof} Let us consider the problem $(\mathring{\textrm{\textbf{P}}}_\epsilon): \min_{(u,U)\in \mathcal {H}'} \mathring{J}_\epsilon(u,U)$  subject  to \eqref{back1}, 
where
$
\mathring{J}_\epsilon(u,U)= \frac{1}{2\epsilon}\mathbb{E}\int_{\mathcal {O}} z^2(T)\mathrm{d}x+\frac{1}{2}\mathbb{E} \int_{\mathcal {O}'_T} \lambda^{-3} \mu^{-4} \mathring{\xi}^{-3} \mathring{\theta}^{-2}u^2 \mathrm{d}x\mathrm{d}t+\frac{1}{2}\mathbb{E}\int_{\mathcal {O}_T} (\mathring{\theta}_\epsilon^{-2} z^2 + \lambda ^{-2}\mu ^{-2}\mathring{\xi} ^{-3} \mathring{\theta}_\epsilon^{-2} |\nabla z|^2 +  \lambda^{-2} \mu^{-2} \mathring{\xi}^{-3} \mathring{\theta}^{-2}U^2) \mathrm{d}x\mathrm{d}t$, and
 $\mathcal {H}'$ is the space of   $(u,U)$ satisfying $\mathbb{E}( \int_{\mathcal {O}'_T} (\lambda^{-3} \mu^{-4} \mathring{\xi}^{-3} \mathring{\theta}^{-2}u^2 \mathrm{d}x\mathrm{d}t+  \int_{\mathcal {O}_T}\lambda^{-2} \mu^{-2} \mathring{\xi}^{-3} \mathring{\theta}^{-2}U^2 \mathrm{d}x\mathrm{d}t)<\infty$.  By using the Euler-Lagrange principle as we did before, one can show that the optimal control
$
(u_\epsilon,U_\epsilon)= (-\lambda^3\mu^4\mathring{\xi}^{3}\mathring{\theta}^{2}r_\epsilon \textbf{1}_{\mathcal {O}'},- \lambda^2\mu^2\mathring{\xi}^{3}\mathring{\theta}^{2}R_\epsilon),
$
where $(r_\epsilon,R_\epsilon)$ solves the backward system
\begin{equation}\label{8.1}
\left\{
\begin{aligned}
&\mathrm{d} r_\epsilon+ \nabla\cdot (\mathcal {A}\nabla r_\epsilon) \mathrm{d}t=  \big[-\rho r_\epsilon-\mathring{\theta}_\epsilon^{-2} z_\epsilon+\nabla\cdot(r_\epsilon\textbf{c} + \lambda^{-2} \mu^{-2}\mathring{\xi}^{-3}\mathring{\theta}_\epsilon^{-2} \nabla z_\epsilon)\big]\mathrm{d}t \\
&\quad\quad\quad\quad\quad\quad\quad\quad\quad+ R_\epsilon \mathrm{d} W_t~\textrm{in}~\mathcal {O}_T,\\
& r_\epsilon=0\hspace{2mm}\textrm{on}~\Sigma_T,\quad
r_\epsilon(T)= \frac{1}{\epsilon}z_\epsilon(T) \hspace{2mm}\textrm{in}~\mathcal {O}.
\end{aligned}
\right.
\end{equation}
Here, $ z_\epsilon \in \mathcal {W}_T$ denotes the unique solution to \eqref{back1} associated with the control pair $(u_\epsilon,U_\epsilon)$.
By applying the It\^{o} formula to the process $z_\epsilon r_\epsilon$ and proceeding as in Theorem \ref{lem4.1}, we get an upper bound for $(z_\epsilon,u_\epsilon, U_\epsilon)$, uniformly in $\epsilon$. Then the desired result follows from the same argument for Lemma \ref{lem7.1}.
\end{proof}

\begin{proof}[Proof of Theorem \ref{thm-nonlinear2}] The proof proceeds in exactly the same manner as that for Theorem \ref{thm-nonlinear1}. Indeed, let us consider the system 
\begin{equation*} 
	\left\{
	\begin{aligned}
		&\mathrm{d} y-\nabla\cdot (\mathcal {A}\nabla y) \mathrm{d}t= \left( \phi +\textbf{1}_{\mathcal {O}'}u\right)\mathrm{d}t+U\mathrm{d}W_t~\textrm{in}~\mathcal {O}_T,\\
		&y=0~\textrm{on}~\Sigma_T,\quad
		y(0)  =y_0~\textrm{in}~\mathcal {O},
	\end{aligned}
	\right.
\end{equation*}	
and define a mapping $
\mathscr{J}:\phi\in \mathscr{D}_{\lambda,\mu} \mapsto F_1(\omega,t,x,y,\nabla y)$, where 
$
\mathscr{D}_{\lambda,\mu}= \{\phi\in L^2_\mathbb{F}(0,T;L^2(\mathcal {O}));~\mathbb{E}\int_{\mathcal {O}_T} \lambda^{-3} \mu^{-4} \mathring{\xi}^{-3} \mathring{\theta}^{-2}\phi^2\mathrm{d}x \mathrm{d}t<\infty\}.
$ By using the assumption assumption (A$_3$) and Lemma \ref{thm5}, one can show that $\mathscr{J}$ is a contraction mapping in $\mathscr{D}_{\lambda,\mu}$, which leads to the desired result by using the Banach Fixed-point Theorem.
\end{proof}

\section*{Acknowledgments}
The authors would like to thank the editors and anonymous referees for their valuable comments and constructive suggestions, which have significantly enhanced both the results of this work and its presentation.

\bibliographystyle{siamplain}
\bibliography{references}
\end{document}